\documentclass[10pt]{elsarticle}
\RequirePackage{snapshot}
\usepackage[utf8]{inputenc}
\usepackage[T1]{fontenc}
\usepackage{fixltx2e}
\usepackage{graphicx}
\usepackage{longtable}
\usepackage{float}
\usepackage{wrapfig}
\usepackage[normalem]{ulem}
\usepackage{amssymb}
\usepackage{amstext}
\usepackage{hyperref}
\usepackage{amsmath}
\usepackage{amsthm}
\usepackage{tabularx}
\usepackage{xfrac}
\usepackage{array}
\usepackage{appendix}
\usepackage[capitalize]{cleveref}
\crefname{appsec}{Appendix}{Appendices}
\crefname{lemma}{Lemma}{Lemmata}
\Crefname{lemma}{Lemma}{Lemmata}
\crefname{thm}{theorem}{theorems}
\Crefname{thm}{Theorem}{Theorems}
\crefname{alg}{algorithm}{algorithms}
\Crefname{alg}{Algorithm}{Algorithms}
\crefname{table}{Table}{Tables}
\Crefname{table}{Table}{Tables}

\renewcommand{\vec}[1]{\boldsymbol{\mathbf{#1}}}
\newcommand{\mat}[1]{\boldsymbol{\mathbf{#1}}}
\newcommand{\Bezier}{{B\'{e}zier} }
\usepackage[margin=1in]{geometry}
\newcommand{\fspace}[1]{\mathcal{#1}}
\newcommand{\trans}{\mathrm{T}}
\newtheorem{theorem}{Theorem}[section]

\newtheorem{remark}[theorem]{Remark}
\newtheorem{lemma}[theorem]{Lemma}

\newtheorem{algorithm}[theorem]{Algorithm}
\usepackage[numbers]{natbib}
\usepackage{pdfcomment}
\date{}
\begin{document}
\begin{frontmatter}
\title{\Bezier projection: a unified approach for local projection
  and quadrature-free refinement and coarsening of NURBS and T-splines with particular
  application to isogeometric design and analysis}

\author[byu1]{D. C. Thomas\corref{cor1}}
\ead{dthomas@byu.edu}

\author[byu2]{M. A. Scott}

\author[cuboulder]{J. A. Evans}

\author[byu3]{K. Tew}

\author[byu4]{E. J. Evans}

\cortext[cor1]{Corresponding author}

\address[byu1]{Department of Physics and Astronomy,
  Brigham Young University,
  Provo, Utah 84602, USA}
\address[byu2]{Department of Civil and Environmental Engineering,
  Brigham Young University,
  Provo, Utah 84602, USA}

\address[cuboulder]{Department of Aerospace Engineering Sciences,
  University of Colorado, Boulder,
  Boulder, Colorado 80309, USA}

\address[byu3]{Department of Information Technology,
  Brigham Young University,
  Provo, Utah 84602, USA}
\address[byu4]{Department of Mathematics,
  Brigham Young University,
  Provo, Utah 84602, USA}

\begin{abstract}
  We introduce B\'{e}zier projection as an element-based local projection methodology for B-splines, NURBS, and T-splines.
This new approach relies on the concept of B\'{e}zier extraction and an associated operation introduced here, spline reconstruction, enabling the use of B\'{e}zier projection in standard finite element codes.
B\'{e}zier projection exhibits provably optimal convergence and yields projections that are virtually indistinguishable from global $L^2$ projection.
B\'{e}zier projection is used to develop a unified framework for spline operations including cell subdivision and merging, degree elevation and reduction, basis roughening and smoothing, and spline reparameterization.
In fact, B\'{e}zier projection provides a \emph{quadrature-free} approach to refinement and coarsening of splines.
In this sense, B\'{e}zier projection provides the fundamental building block for $hpkr$-adaptivity in isogeometric analysis.
\end{abstract}

\begin{keyword}
 \Bezier extraction, spline reconstruction, isogeometric analysis, local refinement and coarsening, local projection, quasi-interpolation
\end{keyword}

\end{frontmatter}
%\tableofcontents

\section{Introduction}
Projection is a ubiquitous operation in numerical analysis and scientific computing.  Consequently, there is a great need to develop projection technologies that are cheap, accurate, and reliable.
This is particularly true in isogeometric analysis where the spline basis may not be interpolatory.
In this paper we present \Bezier projection, a methodology for local projection (i.e., quasi-interpolation) onto and between spline
spaces. \Bezier projection converges optimally, is
virtually indistinguishable from global projection (which requires the
solution of a global system of equations), and provides a unified framework for
\emph{quadrature-free} refinement \textit{and} coarsening of B-splines, NURBS, and
T-splines. In particular \Bezier projection can accommodate
\begin{itemize}
\item (Cell) Subdivision or $h$-refinement,
\item (Cell) Merging or $h$-coarsening,
\item (Degree) Elevation or $p$-refinement,
\item (Degree) Reduction or $p$-coarsening,
\item (Basis) Roughening or $k$-refinement,
\item (Basis) Smoothing or $k$-coarsening,
\item Reparameterization or $r$-adaptivity.
\end{itemize}
These operations can be combined in a straightforward fashion to produce isogeometric
$hpkr$-adaptivity. Note that only $h$-refinement of T-splines has
appeared previously in the
literature~\cite{SeCaFiNoZhLy04,ScLiSeHu10}. \Bezier projection is an
extension of \Bezier
extraction~\cite{ScBoHu10,Borden:2010nx} in that it is derived entirely
in terms of \Bezier elements and element extraction operators. This
means it can be applied to existing finite 
element frameworks in straightforward fashion as an
\emph{element-level} technology. 
Additionally, we formulate \Bezier projection in terms of Kronecker
products facilitating its application in high-dimensional settings.

Potential applications of \Bezier projection are varied and include:
\begin{itemize}
\item Curve and surface fitting,
\item Mesh adaptivity,
\item Enforcement of boundary conditions, 
\item Solution methods with nonconforming meshes,
\item Multi-level solver technology, 
\item Data compression for image, signal, and data processing.
\end{itemize}

In the following, we briefly give a basic background on the building blocks of our approach, namely isogeometric analysis, NURBS and T-splines, and quasi-interpolation.  We additionally present a summary of our paper.

\subsection{Isogeometric analysis}
Isogeometric analysis~\cite{HuCoBa04,Cottrell:2009rp} is a
generalization of finite element analysis which improves the link 
between geometric design and analysis. The isogeometric paradigm is
simple: the smooth spline basis used to define the geometry is used as
the basis for analysis. As a result, exact 
geometry is introduced into the analysis. The 
smooth basis can be leveraged by the
analysis~\cite{EvBaBaHu09,HuEvRe13,CoHuRe07} leading to innovative approaches to 
model design~\cite{CoMaKiLyRi10,WaZhScHu11,LiZhHuScSe14},
analysis~\cite{SchDeScEvBoRaHu12,ScSiEvLiBoHuSe12,ScWuBl12,BeBaDeHsScHuBe09},
optimization~\cite{Wall08}, and
adaptivity~\cite{Bazilevs2009,DoJuSi09,ScThEv13, ScThEv13}.

Many of the early isogeometric developments were restricted to NURBS
but the use of T-splines as an isogeometric basis has gained
widespread attention across a number of
application areas~\cite{Bazilevs2009,ScBoHu10,ScLiSeHu10,Verhoosel:2010vn,
 Verhoosel:2010ly,BoScLaHuVe11,BeBaDeHsScHuBe09,SchDeScEvBoRaHu12,
 ScSiEvLiBoHuSe12,SiScTaThLi14,DiLoScWrTaZa13,HoReVeBo14,BaHsSc12,BuSaVa12,GiKoPoKaBeGeScHu14}.
Particular focus has been placed on the use of T-spline local refinement in an analysis
context~\cite{sc11,ScLiSeHu10,BoScLaHuVe11,Verhoosel:2010ly,Verhoosel:2010vn}. 

\subsection{B-splines, NURBS, T-splines, and more}
\Bezier curves and surfaces \cite{dec63,bez66,bez67}, B-splines
\cite{Bo72,Ri73}, and NURBS \cite{Ve75,Pi91,PiegTil97} have become the
standard for computer graphics and computer-aided design \cite{Pi91}. 
This ubiquity has driven the development of many spline-based
algorithms. Important examples include knot
insertion and knot removal \cite{CoLyRi80,GoLy93,PiegTil97,EcHa95} to
subdivide and merge cells in the mesh as well as modify the smoothness
of the spline functions and degree elevation and reduction
\cite{Pr84,HuHuMa05,PiegTil97} to modify the polynomial degree of the
basis.

T-splines, introduced in the CAD community~\cite{SeZhBaNa03}, are a
generalization of non-uniform 
rational B-splines (NURBS) which address fundamental
limitations in NURBS-based design. For example, a T-spline can model a complicated design as a single,
watertight geometry and are also locally
refineable~\cite{SeCaFiNoZhLy04, ScLiSeHu10}. Since their advent they
have emerged as an important technology across multiple disciplines
and can be found in several major commercial CAD products~\cite{TSManual12,Autodesk360}.
Recent developments include analysis-suitable
T-splines~\cite{LiZhSeHuSc10,ScLiSeHu10,BeBuChSa12,BeBuSaVa12,LiScSe12},
and their hierarchical extension~\cite{EvScLiTh13}. 

We note that while NURBS and T-splines have become standard technology in IGA there exist a growing number of alternative spline technologies which have been proposed as a basis for IGA.
These are not considered in this paper however it should be noted that \Bezier projection can be used in all of these cases for which \Bezier extraction can be defined.
Hierarchical B-splines~\cite{FoBa88,VuGiJuSi11,SchEvReScHu13,GiJuSp12,
KiGiJu14,GiJuSp13,BeKiBrChMoOhKi14} are a multi-level extension of
B-splines. B-spline forests~\cite{ScThEv13} are a generalization of
hierarchical B-splines to surfaces and volumes of arbitrary
topological genus. Subdivision surfaces generalize smooth B-splines to
arbitrary
topology~\cite{CaCl78,Lo87,GrKrSch02,CiOrShr00,BuHaUm10}. Splines
posed over triangulations have been pursued in the context of
piecewise quadratic $C^1$ Powell-Sabin
splines~\cite{SpMaPeSa12,SpCaFr13}, and $C^0$ \Bezier
triangles~\cite{JaQi14}. Polynomial splines over hierarchical 
T-meshes 
(PHT-splines)~\cite{htspline2, htspline3, 
  htspline4, htspline5}, modified T-splines~\cite{KaChDe13}, and locally refined splines
(LR-splines)~\cite{DoLyPe13, Br13} are closely related to T-splines
with varying levels of smoothness and approaches to local refinement. Generalized
B-splines~\cite{MaPeSa11,CoMaPeSa10} and T-splines~\cite{BrBeChOhKi14}
enhance a piecewise polynomial spline basis by including
non-polynomial functions, typically
trigonometric or hyperbolic functions.

% In addition to NURBS and T-splines, a number of alternative spline
% technologies have been proposed as a basis for IGA, with varying
% strengths and weaknesses. Hierarchical B-splines~\cite{FoBa88,VuGiJuSi11,SchEvReScHu13,GiJuSp12,
%   KiGiJu14,GiJuSp13,BeKiBrChMoOhKi14} are a multi-level extension of
% B-splines. B-spline forests~\cite{ScThEv13} are a generalization of
% hierarchical B-splines to surfaces and volumes of arbitrary
% topological genus. Subdivision surfaces generalize smooth B-splines to
% arbitrary
% topology~\cite{CaCl78,Lo87,GrKrSch02,CiOrShr00,BuHaUm10}. Splines
% posed over triangulations have been pursued in the context of
% piecewise quadratic $C^1$ Powell-Sabin
% splines~\cite{SpMaPeSa12,SpCaFr13}, and $C^0$ \Bezier
% triangles~\cite{JaQi14}. Polynomial splines over hierarchical 
% T-meshes 
% (PHT-splines)~\cite{htspline2, htspline3, 
%   htspline4, htspline5}, modified T-splines~\cite{KaChDe13}, and locally refined splines
% (LR-splines)~\cite{DoLyPe13, Br13} are closely related to T-splines
% with varying levels of smoothness and approaches to local refinement. Generalized
% B-splines~\cite{MaPeSa11,CoMaPeSa10} and T-splines~\cite{BrBeChOhKi14}
% enhance a piecewise polynomial spline basis by including
% non-polynomial functions, typically
% trigonometric or hyperbolic functions. Compatible
% splines or isogeometric discrete differential
% forms~\cite{BuRiSaVa11,BuSaVa12,BuSaVa10,EvHu12-3,EvHu12-1,EvHu12-2}
% are (vector-valued) sequences of spline spaces which possess important
% geometric structure.

\subsection{Quasi-interpolation and local least-squares projection}

Quasi-interpolation methods were originally developed as an efficient
means to obtain spline
representations~\cite{deboor1973b,Bo90,lee2000,barrera2008,sablonniere2005,CoMaPeSa10}.
\Bezier projection can be viewed as an
extension or generalization of the integral quasi-interpolants
presented by \citet{sablonniere2005}. 
Interested readers are referred to \citet{sablonniere2005} for an
overview and classification of quasi-interpolation methods.

The technique for projection onto a spline basis that is most closely
related to our work is the local least-squares
projection method of \citet{govindjee2012} in which 
 local projections onto the spline basis over an element are computed and
then averaged to obtain a global
control value. The averaging step in the method of \citeauthor{govindjee2012} was not presented as
an average, but rather as the application of the pseudoinverse of the
assembly operator to the control values computed for the elements. 
It can be shown that this is equivalent to a simple average of the
local values and hence the local least-squares projection method can
be viewed as a special case of \Bezier projection. 
\subsection{Summary of the paper}
In \cref{sec:notation} required notation and conventions are established.
The Bernstein basis is defined along with expressions to relate
Bernstein basis polynomials over different intervals. 
Expressions for the Gramian matrix of the Bernstein basis and its inverse are also given.
B-splines and NURBS are defined and an informal presentation of two-dimensional T-splines is given.
\Bezier extraction is presented and the element reconstruction
operator is defined as the inverse of the element extraction
operator. 

\Cref{sec:localized-projection} introduces \Bezier projection as a
localized projection operation related to quasi-interpolation. 
A proof of the optimal convergence of the method is given in \cref{sec:proof-cont-proj}.
The \Bezier projection method has three distinct steps.
\begin{enumerate}
\item A function is projected onto the Bernstein basis over each
  element in the mesh.
\item An element reconstruction operator is used to compute the
  representation of the \Bezier curve, surface, or volume in terms of the spline basis
  functions over each element.
\item The local spline coefficients are averaged to obtain the
  coefficients or control values of the global basis. 
\end{enumerate}
Several applications of the method are given including lifting of the
normal field of a spline surface and projection between
nonconforming meshes. 

\cref{sec:proj-betw-spline} focuses on \Bezier projection between spline spaces.
The element extraction operator and the spline element reconstruction
operator are used to develop {\it quadrature-free} algorithms for knot insertion, knot
removal, degree elevation, degree reduction, and reparameterization. 
The \Bezier projection algorithms developed are summarized in \cref{table:spline-ops}.
Simple one-dimensional examples are given for all spline operations in
the associated sections.

\begin{table}
\small
  \caption{Summary of \Bezier projection algorithms developed in this paper.}
\label{table:spline-ops}
\begin{center}
\renewcommand{\arraystretch}{1.5}
    \begin{longtable}{ | m{3cm}  | m{5.5cm} |}%| m{6.5cm}
      \hline
    Operation & Algorithmic Description \\ \hline
    General projection &
    \parbox[t]{5.5cm}
    {
      \begin{tabular}{l l}
        B-splines, NURBS, & \cref{alg:local-proj}\\
        and T-splines&
      \end{tabular}
    } \\ \hline
    degree elevation\newline $p$-refinement & 
% $
% \vec{P}^{e,b} = (\mat{R}^{e,b})^{\trans}(\mat{E}^{\vec{p},\vec{p}^{\prime}})^{\trans}(\mat{C}^{e,a})^{\trans}\vec{P}^{e,a} 
% $
%  &
 \parbox[t]{5.5cm}
 {
\begin{tabular}{l l}
  % \cref{sec:degr-elev-reduct}&\\
  B-splines/NURBS & \cref{alg:spline-elev}\\
  T-splines & \cref{alg:t-spline-elev}% \\
  % $\mat{E}^{\vec{p},\vec{p}^{\prime}}$ & \Cref{eq:elev-mat-kp}
\end{tabular}
}
\\
\hline
    degree reduction \newline $p$-coarsening & 
% $
% \vec{P}^{e,b} = (\mat{R}^{e,b})^{\trans}(\mat{R}^{\vec{p},\vec{p}^{\prime}})^{\trans}(\mat{C}^{e,a})^{\trans}\vec{P}^{e,a} 
% $
%  &
 \parbox[t]{5.5cm}
 {
\begin{tabular}{l l}
  % \cref{sec:degr-elev-reduct}&\\
  B-splines/NURBS & \cref{alg:spline-red}\\
  T-splines & \cref{alg:t-spline-red}% \\
  % $\mat{R}^{\vec{p},\vec{p}^{\prime}}$ & \Cref{eq:deg-red-mat-kp}
\end{tabular}
}
\\
    \hline
    knot insertion\newline basis roughening\newline $k$-refinement & 
% $
% \vec{P}^{e,b} = (\mat{R}^{e,b})^{\trans}(\mat{C}^{e,a})^{\trans}\vec{P}^{e,a} 
% $
%  &
 \parbox[t]{5.5cm}
 {
\begin{tabular}{l l}
  % \cref{sec:knot-insert-remov}&\\
  B-splines/NURBS & \cref{alg:knot-insertion-spline}\\
  T-splines & \cref{alg:knot-insertion-t-spline}
\end{tabular}
}
\\
    \hline
    knot removal\newline basis smoothing \newline $k$-coarsening& 
% $
% \vec{P}^{e,b} = (\mat{R}^{e,b})^{\trans}(\mat{C}^{e,a})^{\trans}\vec{P}^{e,a} 
% $
%  &
 \parbox[t]{5.5cm}
 {
\begin{tabular}{l l}
  % \cref{sec:knot-insert-remov}&\\
  B-splines/NURBS & \cref{alg:knot-removal-spline}\\
  T-splines & \cref{alg:knot-removal-t-spline}
\end{tabular}
}
\\
    \hline
    knot insertion\newline cell subdivision \newline $h$-refinement& 
% $
% \vec{P}^{e_i} = (\mat{R}^{e_i})^{\trans}\mat{A}_i(\mat{C}^e)^{\trans}\vec{P}^{e}
% $
%  &
 \parbox[t]{5.5cm}
 {
\begin{tabular}{l l}
  % \Cref{sec:knot-insert-remov-1}&\\
  B-splines/NURBS & \cref{alg:h-refine-spline}\\
  T-splines & \cref{alg:h-refine-t-spline}\\
  % $\mat{A}_i$ & \Cref{eq:interval-op}\\& \Cref{eq:big-A-def}
\end{tabular}
}
\\
    \hline
    knot removal\newline cell merging\newline $h$-coarsening& 
% $
%   \vec{P}^{\bar{e}} = (\mat{R}^{\bar{e}})^{\trans}\left[\sum_{i=1}^n \phi_i\mat{G}^{-1}\mat{A}_i^{-\trans}\mat{G}(\mat{C}^{e_i})^\trans\vec{P}^{e_i}\right]
% $
%  &
 \parbox[t]{5.5cm}
 {
\begin{tabular}{l l}
  % \Cref{sec:knot-insert-remov-1}&\\
  B-splines/NURBS & \cref{alg:h-coarsen-spline}\\
  T-splines & \cref{alg:h-coarsen-t-spline}% \\
  % $\mat{A}_i^{-1}$ & \Cref{eq:inv-interval-op}\\
  % $\mat{G}$ & \Cref{eq:gramian-expr}\\
  % $\mat{G}^{-1}$&\Cref{eq:inv-gramian-expr}\\
  % $\phi_i$&\Cref{eq:overlap-frac-def}
\end{tabular}
}
\\
    \hline
    reparameterization\newline $r$-refinement& 
% \Cref{eq:large-to-small,eq:small-to-large}
%  &
 \parbox[t]{5.5cm}
 {
\begin{tabular}{l l}
  B-splines, NURBS, & \cref{alg:reparam}\\
 and T-splines&
\end{tabular}
}
\\
    \hline
    \end{longtable}
\end{center}
\end{table}
\clearpage
\section{Notation and preliminaries}\label{sec:notation}
\subsection{Univariate Bernstein basis}\label{sec:univ-bernst-basis}
The univariate Bernstein basis functions are defined as
\begin{equation}
  B_{i}^p(\xi) =\frac{1}{2^p}{p \choose
    i-1}(1-\xi)^{p-(i-1)}(1 + \xi)^{i-1},
\end{equation}
where $\xi \in [-1,1]$ and the binomial coefficient ${p \choose i-1} = \frac{p!}{(i-1)!(p+1-i)!}$, $1 \le i \le p + 1$.
We choose to define the Bernstein basis over the biunit interval to facilitate Gaussian quadrature in finite element analysis
rather than use the CAGD convention where the Bernstein polynomials are defined over the unit interval $[0,1]$.
The univariate Bernstein basis has the following properties:
\begin{itemize}
\item {\textit{Partition of unity}}. $$\sum_{i=1}^{p+1} B_{i}^p(\xi)=1 \quad \forall
  \xi \in [-1, 1]$$
\item {\textit{Pointwise nonnegativity}}. $$B_{i}^p(\xi) \ge 0 \quad \forall \xi \in
  [-1,1]$$
\item {\textit{Endpoint interpolation}}. $$B_{1}^p(-1)=B_{p+1}^p(1)=1$$
\item {\textit{Symmetry}}. $$B_{i}^p(\xi) = B_{p+1-i, p}(-\xi) \quad \forall
  \xi \in [-1, 1]$$
\end{itemize}
The Bernstein basis functions for polynomial degrees $p=1,2,3$ are shown in \cref{fig:bernstein}.
\begin{figure}[htb]
  \centering
  \includegraphics[width=5in]{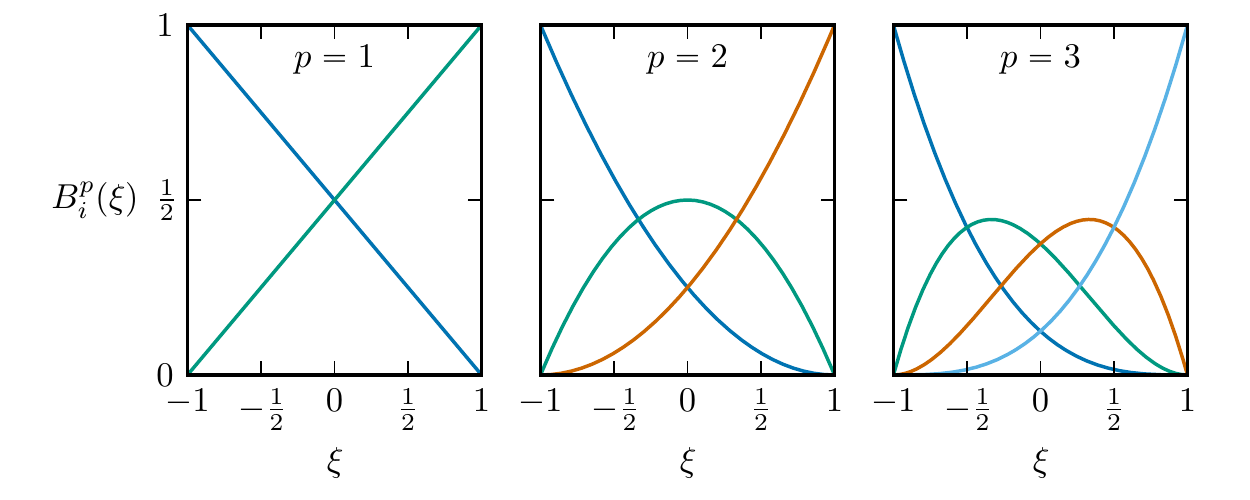}
  \caption{\label{fig:bernstein}The Bernstein basis for polynomial degrees $p=1,2,3$.}
\end{figure}
It is often useful to define a vector of basis functions
\begin{equation}
  \vec{B}^p(\xi) = 
  \begin{bmatrix}
    B_{1}^p(\xi)\\
    B_{2}^p(\xi)\\
    \vdots\\
    B_{p+1}^p(\xi)
  \end{bmatrix}.
\end{equation}
The degree superscript is suppressed when unnecessary.
\begin{lemma}\label{lemma:bern-lin-indep}
  The Bernstein polynomials of degree $p$ are linearly independent and
  form a complete basis for the polynomials of degree $p$ over the
  biunit interval. 
\end{lemma}
We denote the space of functions over the biunit interval spanned by
the Bernstein basis of degree $p$ by $\fspace{B}^p$. 
A useful review of Bernstein polynomials and their properties is
provided by \citet{farouki2012}.

\subsection{Multivariate Bernstein basis}\label{sec:mult-bernst-basis}
We define a multivariate Bernstein basis over the box of dimension $d_p$, $[-1,1]^{d_p}$, by the tensor product.
The polynomial degree may be different in each direction and so we define the vector of degrees $\vec{p}=\left\{ p_\ell \right\}_{\ell=1}^{d_p}$.
The vector of multivariate Bernstein basis functions is defined by the Kronecker product
\begin{equation}
\label{eq:5}
\vec{B}^{\vec{p}}=\vec{B}^{p_{d_p}}(\xi_{d_p})\otimes\cdots\otimes\vec{B}^{p_1}(\xi_1).
\end{equation}
Thus, there are $n_b =\prod_{\ell=1}^{d_p}(p_{\ell}+1)$ basis functions in the vector.
All of the properties of the univariate Bernstein basis are inherited by the multivariate Bernstein basis.
For two dimensions, the Bernstein basis functions can be indexed by the map 
\begin{equation}
\label{eq:bern-map-2d}
a\left(i,j\right) = (p_1+1)(j-1) + i
\end{equation}
so that
\begin{equation}
\label{eq:bern-basis-2d}
B^{\vec{p}}_{a(i,j)}(\xi_1,\xi_2)=B^{p_1}_i(\xi_1) B^{p_2}_j(\xi_2),
\end{equation}
where $\vec{p}=\left\{ p_1, p_2 \right\}$.
For the trivariate case we have that
\begin{equation}
B_{a(i,j,k)}^{\vec{p}} (\boldsymbol{\xi}) =  B_i^{p_1}(\xi_1)
B_j^{p_2} (\xi_2 ) B_k^{p_3} (\xi_3 ),
\end{equation}
where $\vec{p}=\left\{ p_1, p_2, p_3 \right\}$
with the map
\begin{equation}
\label{eq:bern-map-3d}
a\left(i,j,k\right) = (p_1 + 1 )(p_2 + 1)(k-1) + (p_1+1)(j-1) + i.
\end{equation}
\subsubsection{Relations between Bernstein polynomials over different intervals}\label{sec:relat-betw-bernst}
The Bernstein polynomials over the interval $[a,b]$ are
\begin{equation}
  \label{eq:bern-def}
  B_{i}^p(t) = \binom{p}{i}\frac{(b-t)^{p-(i-1)}(t-a)^{i-1}}{(b-a)^p}.
\end{equation}
Given another interval $[\tilde{a},\tilde{b}]$, the Bernstein polynomials are
\begin{equation}
  \label{eq:bern-def-til}
  \tilde{B}_{i}^p(t) = \binom{p}{i}\frac{(\tilde{b}-t)^{p-(i-1)}(t-\tilde{a})^{i-1}}{(\tilde{b}-\tilde{a})^p}.
\end{equation}
A polynomial function $f$ of degree $p$ can be represented by a linear combination of the Bernstein polynomials over $[a,b]$ or by a combination of the Bernstein polynomials over $[\tilde{a},\tilde{b}]$
\begin{equation}
  f(t) = \sum_{i = 1}^{p+1}c_iB_i^p(t) = \sum_{i = 1}^{p+1}\tilde{c}_i\tilde{B}_i^p(t).
\end{equation}
As shown by \citet{farouki1990}, the coefficient vectors $\vec{c} = \{c_i\}_{i=1}^{p+1}$ and $\tilde{\vec{c}} = \{\tilde{c}_i\}_{i=1}^{p+1}$ can be related by the transformation matrix $\mat{A}$
\begin{equation}
  \tilde{\vec{c}} = \mat{A}\vec{c}
\end{equation}
where the entries of $\mat{A}$ are given by
\begin{equation}
  \label{eq:interval-op}
  A_{jk} = \sum_{i=\max(1,j+k-p+1)}^{\min(j,k)}B_i^{j-1}(\tilde{b})B_{k-i}^{p-j-1}(\tilde{a})\quad\textrm{for $j,k=1,2,\hdots,p+1$}.
\end{equation}
This can be extended to multiple dimensions by a tensor product.
The elements of the inverse of $\mat{A}$ are given by 
\begin{equation}
  \label{eq:inv-interval-op}
  [\mat{A}^{-1}]_{jk} = \sum_{i=\max(1,j+k-p+1)}^{\min(j,k)}\tilde{B}_i^{j-1}(b)\tilde{B}_{k-i}^{p-j-1}(a)\quad\textrm{for $j,k=1,2,\hdots,p+1$}.
\end{equation}
The inverse matrix provides a relationship between the basis functions over each interval
\begin{equation}
\label{eq:inv-interval-op-basis}
\tilde{\vec{B}}^p=\mat{A}^{-\trans}\vec{B}^p.
\end{equation}
Both \cref{eq:interval-op,eq:inv-interval-op} are defined using one-based indexing for both the matrix entries and the Bernstein basis as opposed to the zero-based indexing used by \citet{farouki1990}.
\subsubsection{The Gramian of the Bernstein basis and its inverse}
When computing the projection of an arbitrary function onto the Bernstein polynomials, an expression for the Gramian matrix $\mat{G}^p$ for the basis of degree $p$ is required.
The entries in the matrix are
\begin{equation}
\label{eq:gramian-entries}
G_{jk}^p=\int_{-1}^1B_j^p(\xi)B_k^p(\xi)d\xi\quad\textrm{for $j,k=1,2,\hdots,p+1$}.
\end{equation}
Expressions for products and integrals of the Bernstein polynomials given by \citet{doha2011,farouki2012} permit \cref{eq:gramian-entries} to be written in closed form as
\begin{equation}
  \label{eq:gramian-expr}
  G_{jk}^p=\frac{2}{2p+1}{2p\choose j+k-2}^{-1}{p\choose j-1}{p\choose k-1}.
\end{equation}
The polynomial degree $p$ will generally be suppressed.
The Gramian matrix for a multivariate Bernstein basis of dimension $d_p$ and with the vector of polynomial degrees $\vec{p}=\left\{ p_1,\dots,p_{d_p} \right\}$ is obtained from a Kronecker product
\begin{equation}
\label{eq:multi-var-gramian}
\mat{G}^{\vec{p}}=\mat{G}^{p_{d_p}}\otimes\cdots\otimes\mat{G}^{p_1}.
\end{equation} 

% zero based expression
% \begin{equation}
% [\mat{G}^{-1}]_{jk}=\frac{(-1)^{j+k}}{2}\left[ {p\choose j} {p\choose k} \right]^{-1}\sum_{i=0}^{\min(j,k)}(2i+1){p-i\choose p-j}{p-i\choose p-k}{p+i+1\choose p-j}{p+i+1\choose p-k}
% \end{equation}
An expression for the inverse of the Gramian of the Bernstein basis can be obtained by considering the Bernstein-\Bezier representation of the dual basis given by \citet{juttler1998}.
The \Bezier coefficients of the dual basis are precisely the entries in the inverse of the Gramian and so the expression for the dual basis can be used to obtain
\begin{equation}
\label{eq:inv-gramian-expr}
[(\mat{G}^p)^{-1}]_{jk}=\frac{(-1)^{j+k}}{2}\left[ {p\choose j-1} {p\choose k-1} \right]^{-1}\sum_{i=1}^{\min(j,k)}(2i-1){p-i+1\choose p-j+1}{p-i+1\choose p-k+1}{p+i\choose p-j+1}{p+i\choose p-k+1}
\end{equation}
after modification to use one-based indexing and the Bernstein basis over the biunit interval.
The inverse of a Kronecker product of matrices is given by the Kronecker product of the inverses and so \cref{eq:inv-gramian-expr} can be used to compute the inverse of a multivariate Gramian matrix
\begin{equation}
\label{eq:multi-var-inv-gramian}
(\mat{G}^{\vec{p}})^{-1}=(\mat{G}^{p_{d_p}})^{-1}\otimes\cdots\otimes(\mat{G}^{p_1})^{-1}.
\end{equation} 
\subsection{Univariate spline basis}\label{sec:univ-basis-funct}
A univariate spline is defined by the polynomial degree of the spline $p$ and the knot vector $\mathsf{G}$, a set of non-decreasing parametric coordinates $\mathsf{G}=\left\{ \sigma_i \right\}_{i=1}^{n+p+1}$, $\sigma_i\leq \sigma_{i+1}$ where $n$ is the number of spline basis functions.
We require that the knot vector be open, that is the first $p+1$ knots are equal $\sigma_1=\cdots=\sigma_{p+1}$ and the last $p+1$ knots are equal $\sigma_{n+1}=\cdots=\sigma_{n+p+1}$.
The $A$th spline basis function over the knot vector is defined using the Cox-de Boor recursion formula:
\begin{align}\label{eq:cox-de-boor}
  N_A^0(s) = 
  \begin{cases}
    1 & \sigma_{A} \leq s < \sigma_{A} \\
    0 & \text{otherwise}.
  \end{cases}
\end{align}
\begin{align}
  N_A^p(s) = \frac{s - \sigma_{A}}{\sigma_{A+p} - \sigma_{A}}N_A^{p-1}(s) 
  + \frac{\sigma_{A+p+1} - s}{\sigma_{A+p+1}-\sigma_{A+1}}N_{A+1}^{p-1}(s).
\end{align}
It is also possible to associate a local knot vector with each spline basis function.
The local knot vector $\mathsf{g}_A\subset\mathsf{G}$ is the set of $p+2$ knots chosen contiguously from the knot vector $\mathsf{G}$ that defines the function $N_A$.
Application of the Cox-de Boor recursion formula to the local knot vector $\mathsf{g}_A$ produces the basis function $N_A$ and so it is also possible to index a basis function by its local knot vector:
\begin{equation}
N_A(s)=N_{\mathsf{g}_A}(s).
\end{equation}
The B-spline basis enjoys the following properties:
\begin{itemize}
\item \textit{Partition of unity}. $$\sum_{A=1}^{n} N_A^p(s) = 1,
  \quad \forall s \in [\sigma_1,\sigma_{n+p+1}]$$
\item \textit{Pointwise nonnegativity}. $$N_A^p(s) \geq 0, \quad
  j = 1,\ldots, n, \quad \forall s \in [\sigma_1,\sigma_{n+p+1}]$$
\item \textit{Global linear independence}. $$\sum_{j=1}^{n} c_{j}
  N_A^p(s) = 0 \Leftrightarrow c_i=0, \quad 
  i=1,\ldots,n, \quad \forall s \in [\sigma_1,\sigma_{n+p+1}]$$
\item \textit{Local linear independence}.
Given an open set $\hat{\Omega}' \subseteq \hat{\Omega}$ the B-spline
basis functions having some support in $\hat{\Omega}'$ are linearly independent on
$\hat{\Omega}'$.
\item \textit{Compact support}.
$$\{s \in [\sigma_1,\sigma_{n+p+1}] : N_A^p(s)>0\}\subset[\sigma_i, \sigma_{i+p+1}]$$
\item \textit{Control of continuity}. If $\sigma_{i}$ has multiplicity $k $
  (i.e., $\sigma_i=\sigma_{i+1}=\ldots=\sigma_{i+k-1}$), then 
  the basis functions are $C^{p-k}$-continuous at $\sigma_i$.  When
  $k=p$, the basis is $C^0$ and interpolatory at that location. 
\end{itemize}
It is interesting to observe that the Bernstein basis defined previously is the set of spline functions given by the knot vector
\begin{equation}
\label{eq:bern-kv}
\{\underbrace{-1,\dots,-1}_{p+1}, \underbrace{1,\dots,1}_{p+1} \}.
\end{equation}
As with the Bernstein basis, it is possible to define a vector of basis functions
\begin{equation}
\label{eq:8}
\vec{N}^p(s) = \left\{ N^p_A(s) \right\}_{A=1}^n.
\end{equation}
Note that the number of basis functions is $p+1$ less than the number of knots in the knot vector $\mathsf{G}$.
The space of functions spanned by the functions in $\vec{N}(s)$ is called a spline space.

A spline curve of dimension $d_s$ is a function mapping $\mathbb{R}$ to $\mathbb{R}^{d_s}$.
The curve $\vec{x}(s)$ is defined by a set of $d_s$ dimensional control points $\vec{P}_A$ as
\begin{align}
\label{eq:spline-curve-def}
\vec{x}(s) = \sum_{A=1}^n\vec{P}_AN_A(s).
\end{align}
Due to the variation diminishing property of the spline basis, the curve will generally only interpolate the control points at the ends of the curve or at locations where the spline basis is $C^0$.
An alternate form of \cref{eq:spline-curve-def} can be obtained by defining the vector of control points $\vec{P}=\left\{\vec{P}_A\right\}_{A=1}^n$ so that
\begin{align}
\label{eq:spline-curve-def-2}
\vec{x}(s) = \vec{P}^T\vec{N}(s).
\end{align}
The vector of control points $\vec{P}$ can be interpreted as a matrix of dimension $n\times d_s$.
\subsection{Rational univariate splines}
\label{sec:rat-univ-splines}
The spline basis defined in the previous section provides a flexible means to represent curves, however certain curves of interest such as circular arcs cannot be represented by a polynomial basis.
A rational spline basis can be used to remedy this deficiency.
The rational basis is defined by associating a weight with each basis function $N_A$ and introducing the weight function
\begin{equation}
\label{eq:weight-func-def}
w(s)=\sum_{A=1}^nw_AN_A(s)
\end{equation}
The rational basis functions are then defined as
\begin{equation}
\label{eq:rat-basis-def}
R_A(s)=\frac{w_AN_A(s)}{w(s)}
\end{equation}
and a rational curve is defined as
\begin{equation}
\label{eq:rat-curve-def}
\vec{x}(s)=\sum_{A=1}^n\vec{P}_AR_A(s).
\end{equation}
A rational curve of this type is commonly referred to as a Non-Uniform Rational B-Spline (NURBS) curve.
It is customary to represent the rational curve by a polynomial curve in a $d_s+1$ dimensional space known as a projective space.
The control points $\vec{P}_A$ are converted to so-called homogeneous form $\tilde{\vec{P}}_A=\{w_A\vec{P}_A^\trans,w_A\}^{\trans}$.
This definition permits the definition of the $d_s+1$-dimensional polynomial curve
\begin{equation}
\label{eq:homogen-curve-def}
\tilde{\vec{x}}(s)=\sum_{A=1}^n\tilde{\vec{P}}_AN_A(s).
\end{equation}
This construction permits the application of standard B-spline algorithms to rational spline curves.
Therefore all of the algorithms in this paper can be applied to rational splines.
\subsection{Multivariate spline basis}
Just as the multivariate Bernstein basis is defined by a tensor product of the univariate basis, a multivariate spline basis is defined from a tensor product of univariate spline bases.
The univariate spline basis in each parametric direction are defined by a polynomial degree $p_i$ and a knot vector $\mathsf{G}_i$.
The number of parametric dimensions is denoted by $d_p$.
The total number of basis functions in the spline basis is given by
\begin{equation}
\label{eq:num-spl-funcs}
n=\prod_{i=1}^{d_p}n_i
\end{equation}
where $n_i$ is the number of univariate spline basis functions in the $i$th parametric dimension.

For two dimensions, we define the map
$A(i,j)=n_1(j-1)+i$
and the $A$th basis function is then given by
\begin{equation}
\label{eq:spl-2d-def}
N^{\vec{p}}_{A(i,j)}(s,t)=N_i^{p_1}(s)N_j^{p_2}(t).
\end{equation}
Similarly, for three dimensions,
$A(i,j,k)=n_1n_2(k-1)+n_1(j-1)+i$
and the $A$th basis function is given by
\begin{equation}
\label{eq:spl-3d-def}
N^{\vec{p}}_{A(i,j,k)}(s,t,u)=N_i^{p_1}(s)N_j^{p_2}(t)N_k^{p_3}(u).
\end{equation}
It is also possible to define a single basis function in terms of the local knot vectors that define the function in each parametric direction:
\begin{equation}
\label{eq:spl-2d-local-def}
N^{\vec{p}}_{A(i,j)}(s,t)=N_{\mathsf{g}_i,\mathsf{g}_j}(s,t)=N_{\mathsf{g}_i}(s)N_{\mathsf{g}_j}(t).
\end{equation}
Here $\mathsf{g}_i$ is the $i$th set of $p+1$ contiguous entries in the knot vector $\mathsf{G}_1$ that defines the spline in the first parametric dimension and $\mathsf{g}_j$ is similarly chosen from $\mathsf{G}_2$.
A vector of spline basis functions can be define by a Kronecker product of the vectors of univariate basis functions:
\begin{equation}
\label{eq:kron-prod-multi-var-spl-basis}
\vec{N}^{\vec{p}}(s,t,u)=\vec{N}^{p_3}(u)\otimes\vec{N}^{p_2}(t)\otimes\vec{N}^{p_1}(s)
\end{equation}
A multivariate rational spline basis can be defined in the same manner as the univariate case.

Spline surfaces and volumes can be constructed by associating a control point with each basis function.
The geometric map is defined as the sum of the product of control points and spline basis functions:
\begin{equation}
\label{eq:multi-geom-map}
\vec{x}(\vec{s})=\sum_{A=1}^n\vec{P}_AN_A(\vec{s})
\end{equation}
where $\vec{x}=\left\{ x_i \right\}_{i=1}^{d_s}$ is a spatial position vector and $\vec{s}=\left\{ q_i \right\}_{i=1}^{d_p}$ is a parametric position vector.
The geometric map is given in matrix form as
\begin{equation}
\label{eq:multi-geom-map-mat}
\vec{x}(\vec{s})=\vec{P}^{\trans}\vec{N}^{\vec{p}}(\vec{s})
\end{equation}
where $\vec{P}$ is the $n\times d_s$ vector of control points.% (or $n\times(d_s+1)$ for homogeneous coordinates).

The geometric map is a bijective map from the parametric domain $\hat{\Omega}\subset \mathbb{R}^{d_p}$, which defines the spline, to the physical or spatial domain $\Omega\subset\mathbb{R}^{d_s}$.
Due to the convex hull property of splines, the spatial domain is contained in the convex hull of the control points.
The geometric map defines the spline surface or volume in physical
space and parameterizes it by the parametric coordinates. A rational spline basis can be constructed from these ideas in a manner similar to what is
described for the one-dimensional case in~\cref{sec:rat-univ-splines}.
\subsection{T-splines}
T-splines represent a generalization of B-splines and NURBS.
Whereas a B-spline is constructed from a tensor product of univariate splines, a T-spline permits meshes with hanging nodes or T-junctions.
T-splines contain standard B-splines and NURBS as special cases.
The theory of T-splines is rich and dynamic. We do not delve deeply
into T-spline theory in this paper and instead refer the interested
reader to~\citet{SeCaFiNoZhLy04,ScBoHu10,ScLiSeHu10,LiScSe12} and the references contained therein.
We limit our discussion to two-dimensional T-splines of arbitrary degree.

\subsubsection{The T-mesh}
Given polynomial degrees, $p_1$ and $p_2$, and two index vectors, $\mathsf{I}_i=\left\{ 1, 2, \dots,
  n_i+p_i+1\right\}$, $i=1,2$, we define the index domain of the T-mesh as
$\bar{\Omega}=[1,n_1+p_1+1]\otimes[1,n_2+p_2 +1]$. We associate a knot
vector $\mathsf{G}_i=\left\{\sigma_{1}^i\leq \sigma_{2}^i\leq
  \cdots\leq \sigma_{{n_i+p_i+1}}^i\right\}$, $i=1,2$, with the
corresponding index vector $\mathsf{I}_i$.
Any repeated entries in $\mathsf{G}_i$ are referred to as knots of multiplicity $m$.
We require that the first and last knots have multiplicity $p_i+1$, this is commonly called an open knot vector.
We also require that no knot in $\mathsf{G}_i$ have multiplicity greater than $p_i+1$.
It should be noted that repeated knots have unique indices. We define
the parametric domain of the T-mesh as $\hat{\Omega}=[\sigma_1^1,
\sigma_{{n_1+p_1+1}}^1] \otimes [\sigma_1^2,
\sigma_{{n_2+p_2+1}}^2]$.

A T-mesh $\mathsf{T}$ is a rectangular partition of the index domain
such that all vertices have integer coordinates taken from
$\mathsf{I}_1$ and $\mathsf{I}_2$, all cells are rectangular,
non-overlapping, and open, and all edges are horizontal or vertical
line segments which do not intersect any cell. Because there are
corresponding parametric values assigned to each vertex in the index space, the T-mesh can be mapped to the parametric domain.
Cells in the index domain that are bounded by repeated knot values are mapped to cells of zero parametric area in the parametric domain.
An example T-mesh is shown in \cref{fig:tmesh}.
Cells that are bounded by repeated knots in at least one dimension have zero parametric area and are shown in gray in the figure.
	\begin{figure}[htb]
	\centering
	{\includegraphics [width=2in]{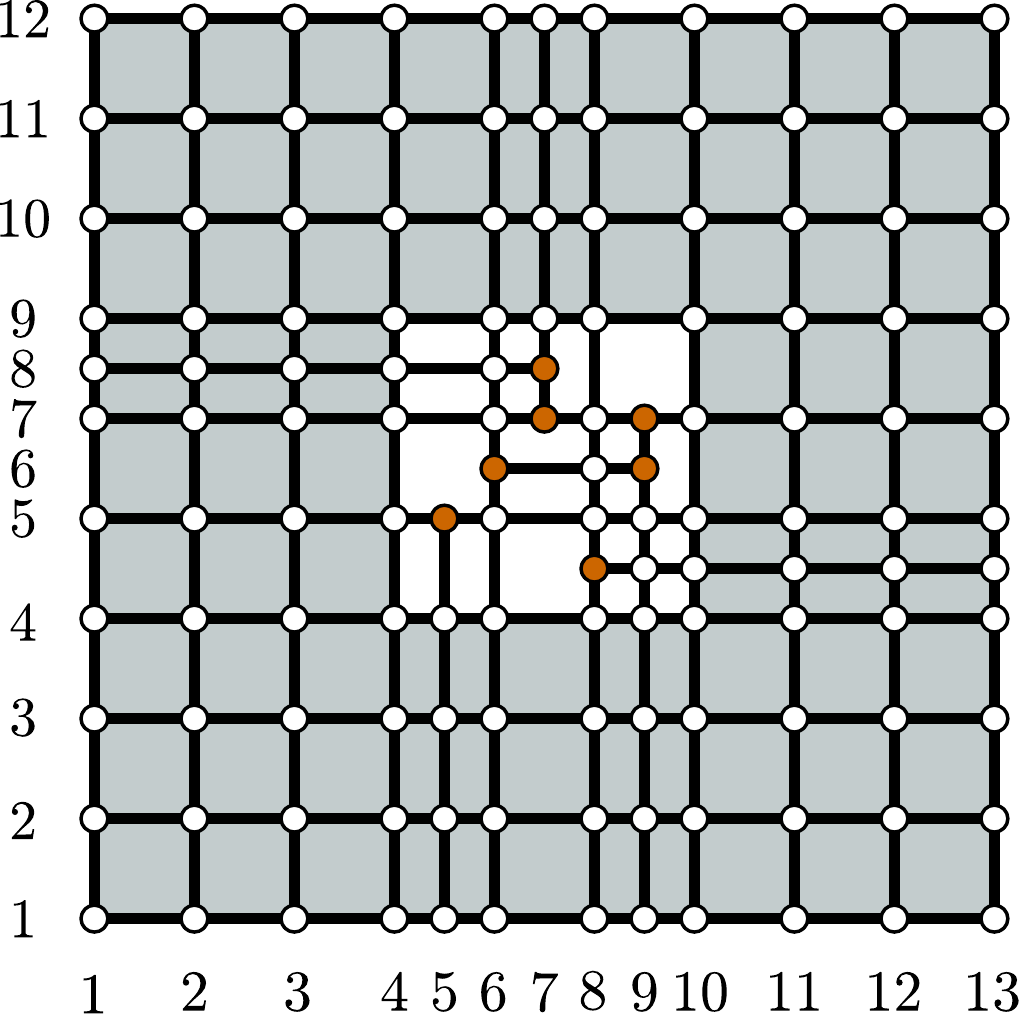}}
	\caption{An example T-mesh.  Vertices are marked with open circles and the vertices corresponding to T-junctions are marked with orange circles.  Cells with zero parametric area are gray.} \label{fig:tmesh}
	\end{figure}
We say that two T-meshes $\mathsf{T}^a$ and $\mathsf{T}^b$ are nested if $\mathsf{T}^b$ can be created by adding vertices and edges to $\mathsf{T}^a$.
We use the notation $\mathsf{T}^a\subseteq\mathsf{T}^b$ to indicate this relationship.

\subsubsection{T-spline basis functions}
The T-spline basis functions are constructed from the T-mesh and the
knot vectors. Note that we use the term basis function throughout this section
although there is no guarantee that the set of blending functions inferred
from a T-mesh form a basis for the space. This question is resolved by
analysis-suitable T-splines. A basis function is anchored to unique T-mesh entities
(i.e., vertices, edges, cells) as follows:
\begin{itemize}
\item If $p_1$ and $p_2$ are odd then the anchors are all the vertices in the T-mesh with indices greater than $i_{p_1}$ and less than ${n_1+1}$ in the first dimension and greater than ${p_2}$ and less than ${n_2+1}$ in the second dimension.
\item If $p_1$ and $p_2$ are even then the anchors are all the cells in the T-mesh bounded by vertices with indices greater than ${p_1}$ and less than ${n_1+1}$ in the first dimension and greater than ${p_2}$ and less than ${n_2+1}$ in the second dimension.
\item If $p_1$ is even and $p_2$ is odd then the anchors are all the horizontal edges in the T-mesh bounded by vertices with indices greater than ${p_1}$ and less than ${n_1+1}$ in the first dimension and greater than ${p_2}$ and less than ${n_2+1}$ in the second dimension.
\item If $p_1$ is odd and $p_2$ is even then the anchors are all the vertical edges in the T-mesh bounded by vertices with indices greater than ${p_1}$ and less than ${n_1+1}$ in the first dimension and greater than ${p_2}$ and less than ${n_2+1}$ in the second dimension.
\end{itemize}
The function anchors for these four cases are illustrated in \cref{fig:anchor-loc}.
We assume that the anchors can be enumerated and refer to each anchor by its index $A$.

\begin{figure}
  \centering
\begin{tabular}{*{3}{c}}
\includegraphics[scale=0.6]{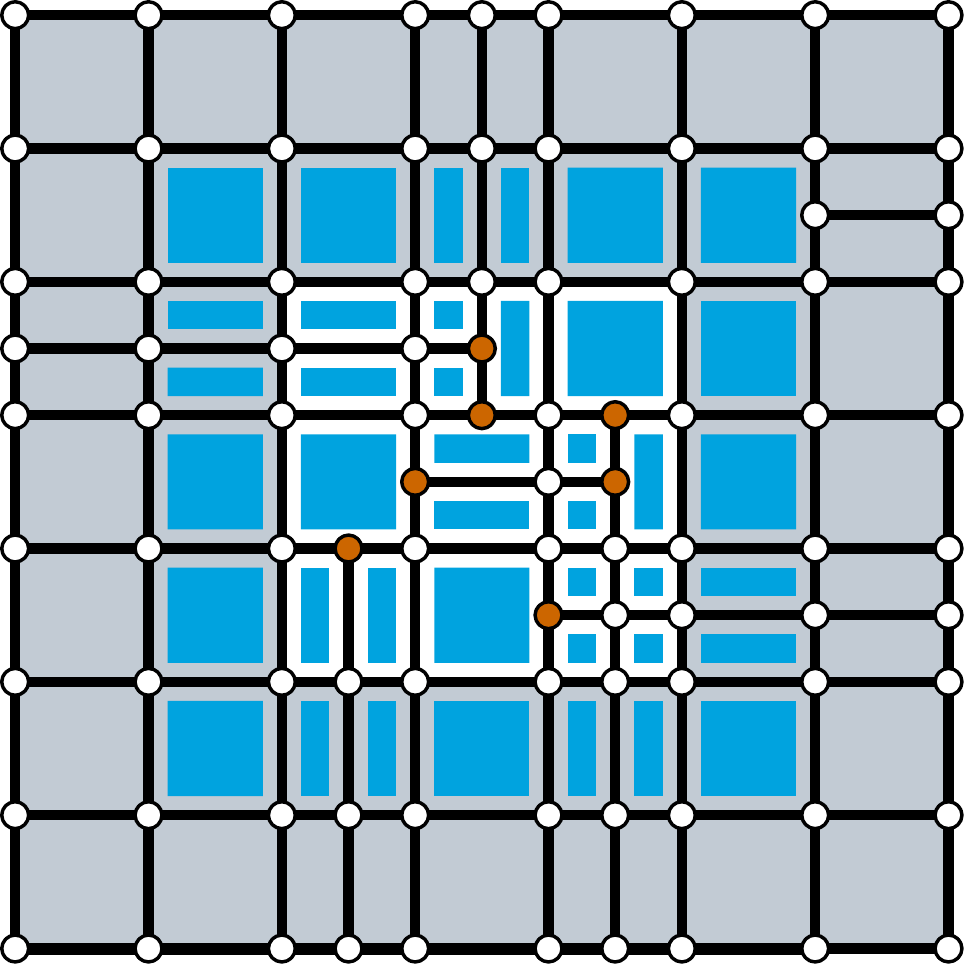}
&
\includegraphics[scale=0.6]{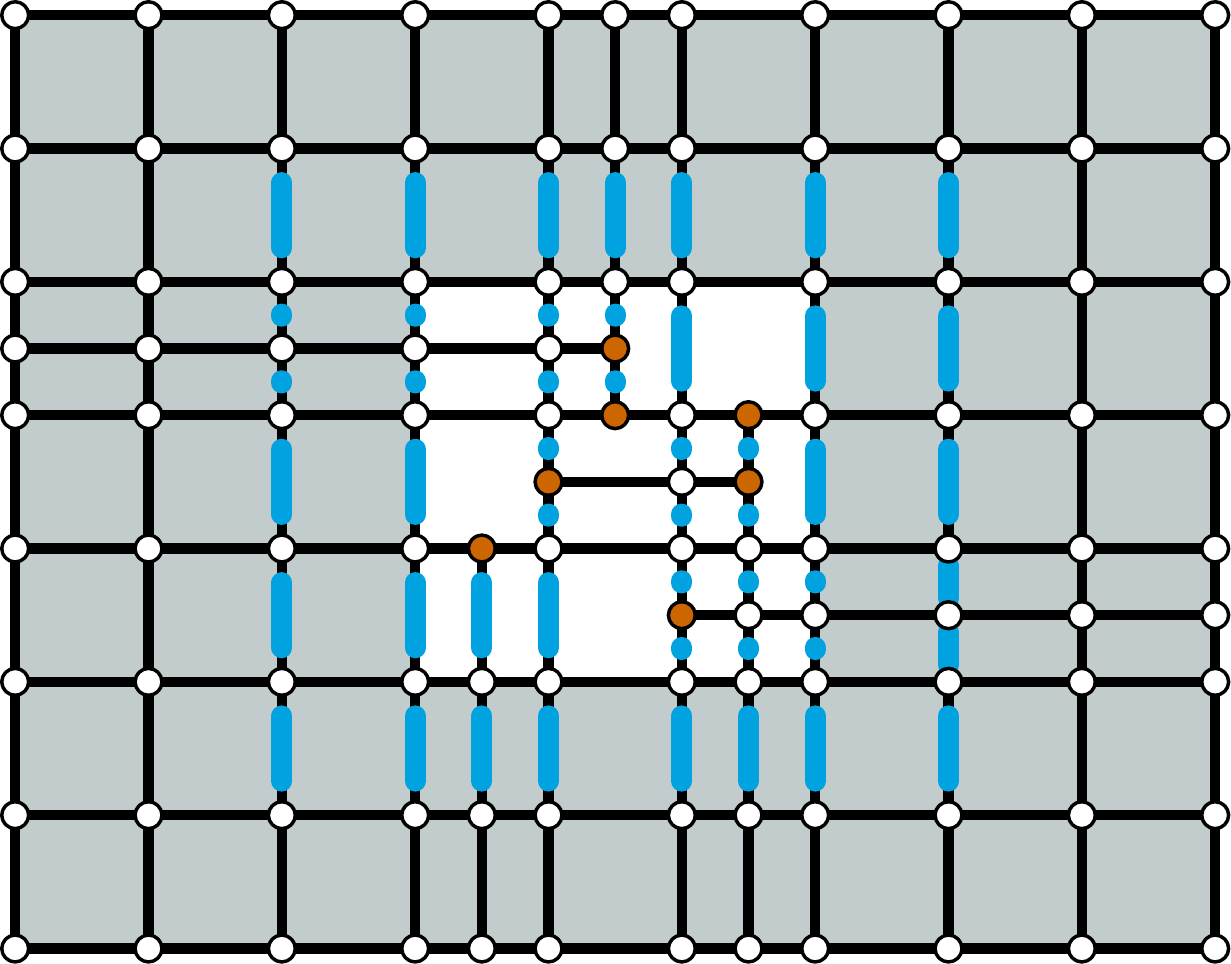}
\\
(a) $p_1 =2,\;p_2=2$ anchors are faces.
&
(b) $p_1=3,\;p_2 =2$ anchors are vertical edges.
\\
&
\\
\includegraphics[scale=0.795]{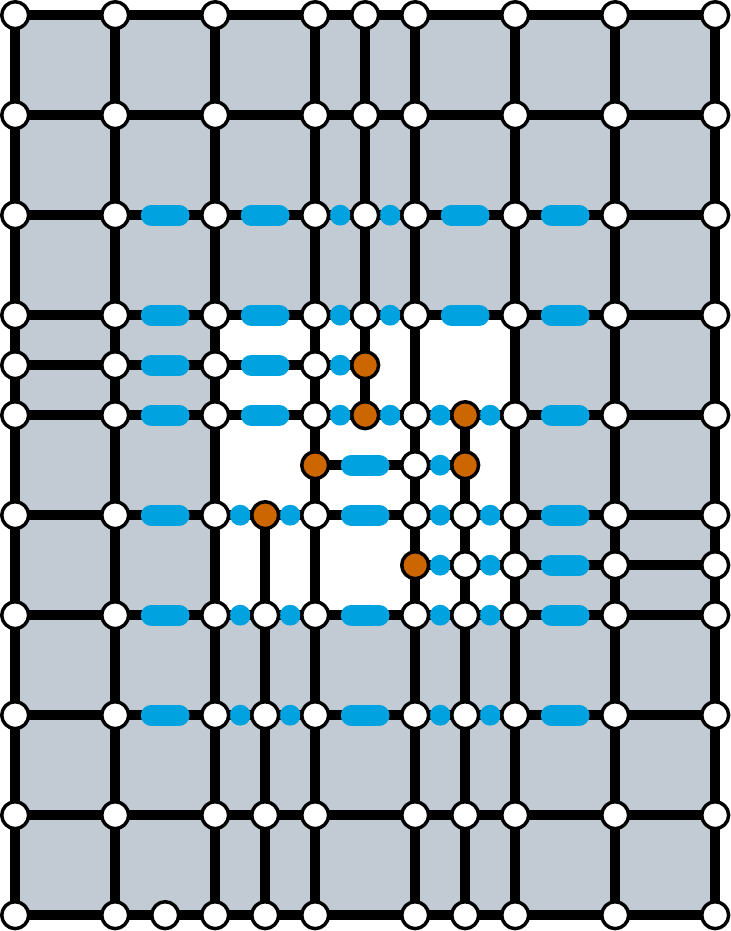}
&
\includegraphics[scale=0.795]{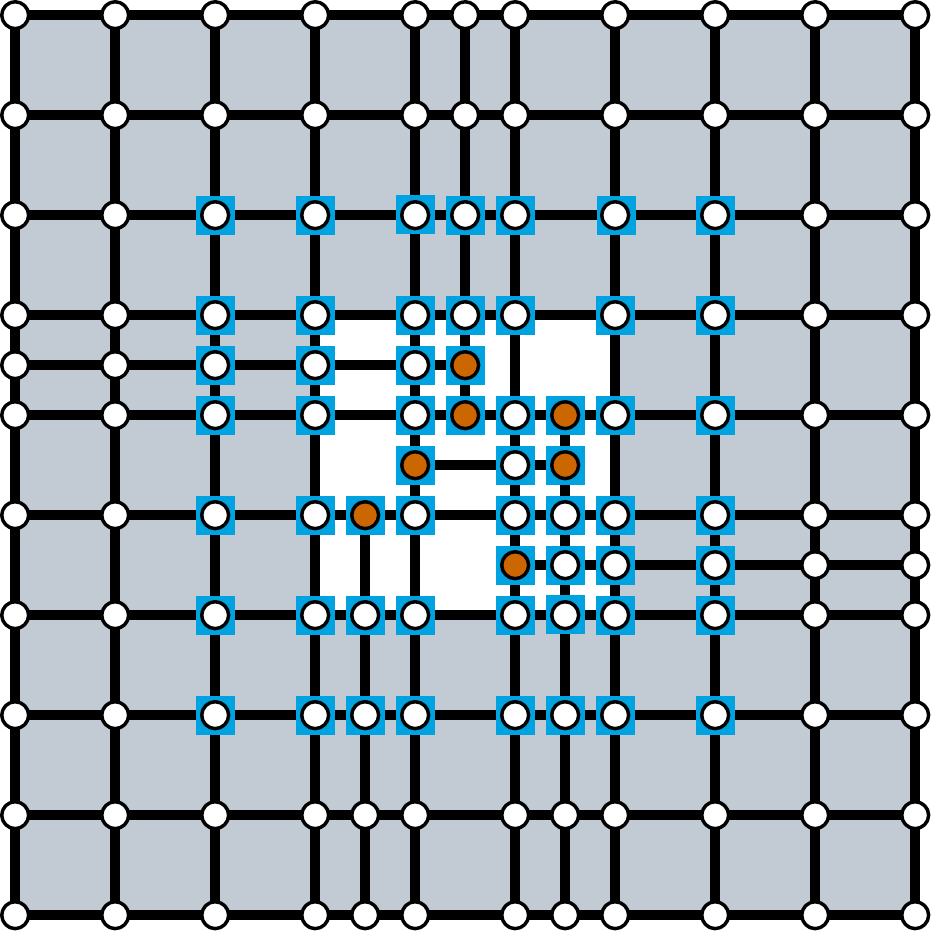}
\\
(c) $p_1=2,\;p_2=3$ anchors are horizontal edges.
&
(d) $p_1=3,\;p_2=3$ anchors are vertices.
\end{tabular}
\caption{The set of anchors for varying values of $p_1$ and $p_2$.  In this picture blue represents the anchor locations, gray is the boundary cells with zero parametric area, and the orange vertices are T-junctions.}
\label{fig:anchor-loc}
\end{figure}

The basis functions associated with each anchor are defined by constructing a local knot vector in each parametric direction.
The function anchor is indicated by its index $A$ and so the local knot vector associated with $A$ in the $i$th parametric direction is denoted by $\mathsf{g}_{A,i}$.
The algorithm for constructing the local knot vector in the $i$th parametric direction associated with the $A$th anchor is given here with examples of its application following.
\begin{algorithm}\label{alg:local-kv}
  Construction of the local knot vector in the $i$th parametric direction for the $A$th function anchor from the T-mesh.
\begin{enumerate} 
\item Find the  line that lies in the $i$th parametric direction and that passes through the center of the function anchor.  We refer to this as the anchor line associated with the $i$th parametric direction.  This line is used to find the indices that define the local knot vector.
\item Determine the width $h_i$ of the anchor of interest in the direction perpendicular to the anchor line and thicken the line so that it has width $h_i$ and is centered on the anchor line.  If the polynomial degree $p_i$ is odd in all directions then $h_i=0$ and so the anchor line and thickened anchor line coincide.
\item Find the indices of vertices or perpendicular edges in the T-mesh whose intersection with the thickened anchor line is nonempty and of length $h_i$.%, are perpendicular to it, and that support a line of length $h_i/2$ or greater drawn in the T-mesh on either side of the anchor line and lying perpendicular to it.
\item The local index vector $\mathsf{i}_{A,i}$ is the ordered set of indices formed by collecting the closest $\lceil (p_i+1)/2 \rceil$ indices found in the previous step on either side of the anchor $A$.  If $p_i$ is odd, then the index of the edge or vertex associated with the anchor is added also to the local index vector.  The local index vector $\mathsf{i}_{A,i}$ is of length $p_i+2$.
\item The local knot vector $\mathsf{g}_{A,i}$ is formed by collecting the knot entries in the global knot vector $\mathsf{G}_i$ given by the indices in the local index vector $\mathsf{i}_{A,i}$.
\end{enumerate}
\end{algorithm}
% The local knot vector is constructed by finding the $\lceil (p_i+1)/2 \rceil$ mesh edges or vertices on either side of the anchor that lie on a line in the $i$th parametric direction that passes through the anchor and forming the ordered vector of the knot values associated with the $i$th index that defines each relevant edge or vertex.
% For odd degrees, the knot value associated with the anchor is also included.
% For splines that have even degree in at least one direction, the mesh edges crossed must be large enough to span the width of the anchor.
By carrying this process out in each parametric direction, a set of local knot vectors $\left\{ \mathsf{g}_{A,1}, \mathsf{g}_{A,2} \right\}$ associated with the anchor $A$ can be obtained.
The basis function associated with the anchor $A$ is then defined in the same manner as the local spline basis function for B-splines indexed by local knot vectors:
\begin{equation}
\label{eq:t-spl-basis-def}
N_A(s,t)=N_{\mathsf{g}_{A,1},\mathsf{g}_{A,2}}(s,t)=N_{\mathsf{g}_{A,1}}(s)N_{\mathsf{g}_{A,2}}(t)
\end{equation}
where the functions $N_{\mathsf{g}_{A,i}}$ are obtained by applying the Cox-de Boor formula to the local knot vectors $\mathsf{g}_{A,i}$.
The process for constructing local knot vectors is illustrated in \cref{fig:local-func-examples} for four cases: $\vec{p}=\left\{ 2,2 \right\}$, $\vec{p}=\left\{ 3,2 \right\}$, $\vec{p}=\left\{ 2,3 \right\}$, and $\vec{p}=\left\{ 3,3 \right\}$.

T-splines with even degree in both directions have cell faces as anchors; thus for the $\vec{p}=\left\{ 2,2 \right\}$ case shown in part (a), the function anchor is marked by a large, light blue box covering the cell face.
We use the global knot vectors 
\begin{equation}
\mathsf{G}_1=\left\{ 0,0,0,1,2,3,4,5,6,7,7,7 \right\}
\end{equation} 
and 
\begin{equation}
\mathsf{G}_2=\left\{ 0,0,0,1,2,3,4,5,6,7,7,7 \right\}.
\end{equation}
The anchor line used to construct the horizontal knot vector is shown in dark blue and the anchor line used to construct the vertical knot vector is shown in green.
Because the spline has even polynomial degree, the thickened anchor line is shown in both directions.
The indices that contribute to the local knot vector are marked with a $\boldsymbol{\times}$.
The indices used to construct the horizontal local index vector for the marked function are $\mathsf{i}_{A,1}=\left\{ 3, 4, 8, {10} \right\}$ and so the local knot vector for the function is $\mathsf{g}_{A,1}=\left\{ 0,1,5,7 \right\}$.
The indices $5$ and $7$ were skipped because there are no edges associated with those indices that intersect with the horizontal anchor line used to determine the local knot vector.
Note that the index $9$ was skipped because the edges that intersect the horizontal line do not span the thickened anchor line due to the missing edge between $(9,6)$ and $(9,7)$.
Similarly, the vertical local index vector is $\mathsf{i}_{A,2}=\left\{ 2, 3, 7, 9 \right\}$ and so the vertical local knot vector for the function is $\mathsf{g}_{A,2}=\left\{ 0,0, 4, 6\right\}$.

\begin{figure}
  \centering
\begin{tabular}{*{3}{c}}
\includegraphics[scale=0.9]{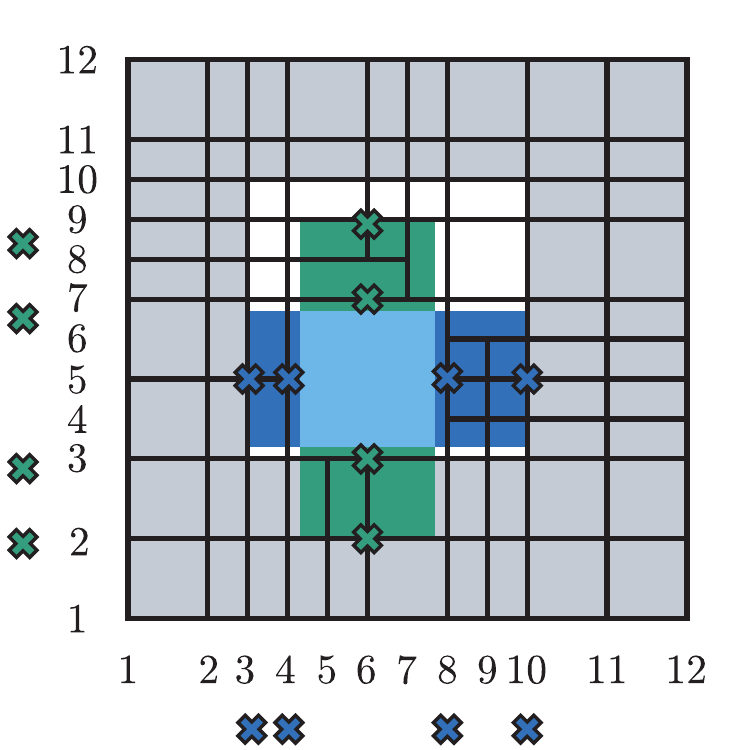}
&
\includegraphics[scale=0.9]{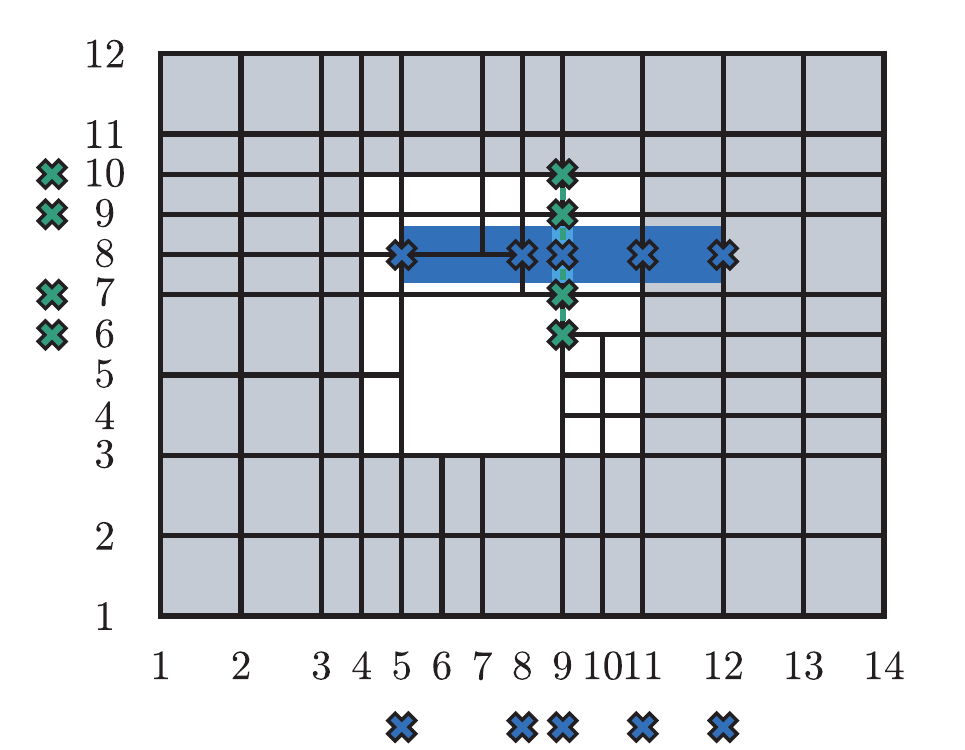}
\\
\parbox{2.5in}{(a) $p_1 =2,\;p_2=2$ anchors are faces.  Thickened anchor lines in both directions.}
&
\parbox{2.5in}{(b) $p_1=3,\;p_2 =2$ anchors are vertical edges.  Only the horizontal anchor line must be thickened.}
\\
&
\\
\includegraphics[scale=0.9]{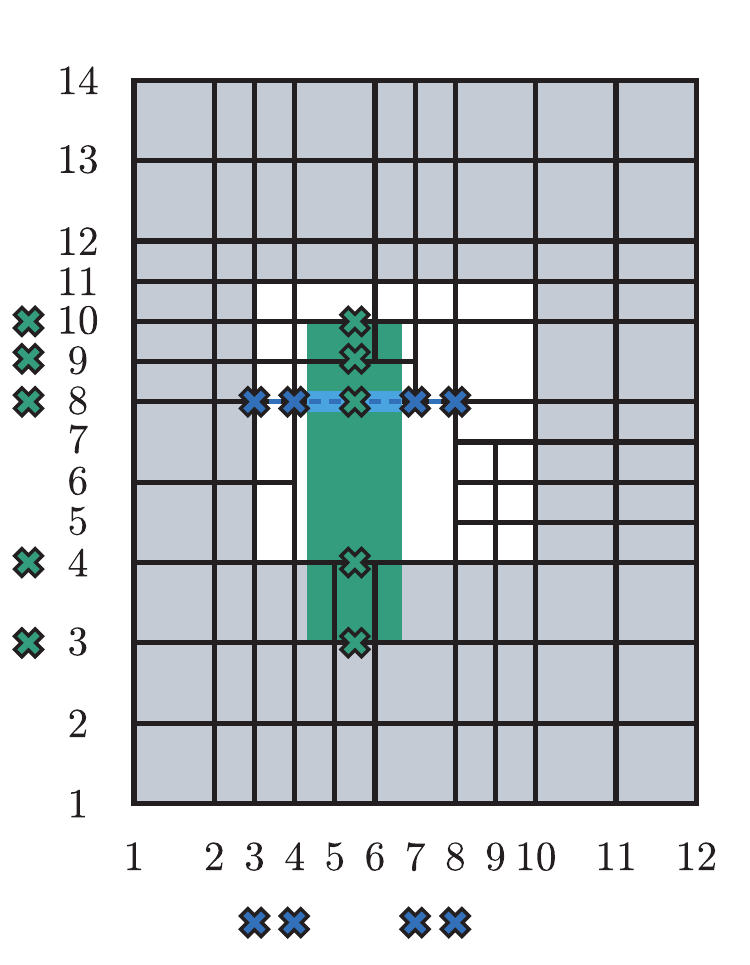}
&
\includegraphics[scale=0.9]{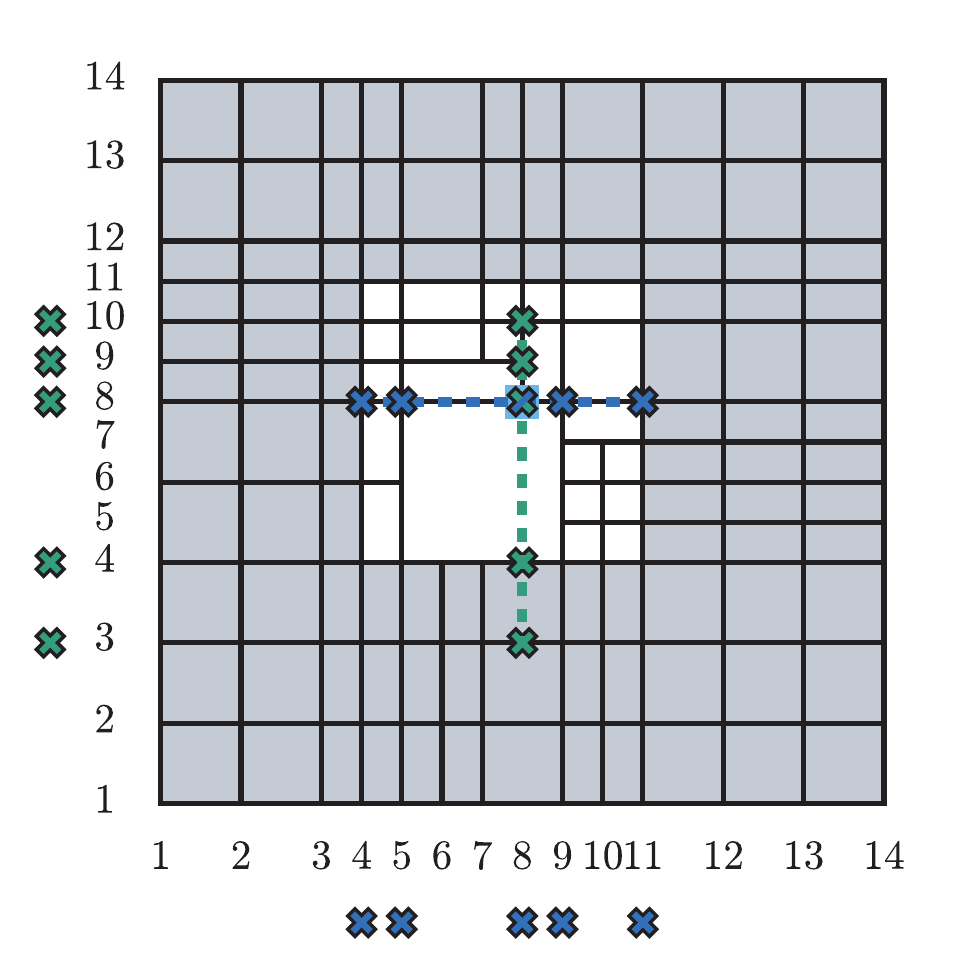}
\\
\parbox{2.5in}{(c) $p_1=2,\;p_2=3$ anchors are horizontal edges.  Only the vertical anchor line must be thickened.}
&
\parbox{2.5in}{(d) $p_1=3,\;p_2=3$ anchors are vertices.  Neither anchor line must be thickened.}
\end{tabular}
\caption{Examples for how local knot vectors are constructed for T-spline basis functions of varying values of the polynomial degrees $p_1$ and $p_2$.  The function anchors are marked with light blue.  The thickened horizontal anchor line used to determine the horizontal knot vector is indicated with a dark box where necessary while in cases that do not require thickening it is shown as a dark blue dashed line.  The vertical anchor line used to calculate the vertical knot vector is marked in green.  The indices that contribute to the local knot vectors are marked with $\times$ and colored dark blue for those in the horizontal direction and green for those in the vertical direction.}
\label{fig:local-func-examples}
\end{figure}

The mixed degree case $\vec{p}=\left\{ 3,2 \right\}$ is shown in part (b).
We now use the global knot vectors 
\begin{equation}
\mathsf{G}_1=\left\{ 0,0,0,0,1,2,3,4,5,6,7,7,7,7 \right\}
\end{equation} 
and 
\begin{equation}
\mathsf{G}_2=\left\{ 0,0,0,1,2,3,4,5,6,7,7,7 \right\}.
\end{equation}
Here the function anchor is a vertical edge and so only the horizontal anchor line is thickened (shown in dark blue in the figure).
The vertical anchor line is shown as a dashed green line.
The horizontal indices that contribute to the local index vector are $\mathsf{i}_{A,1}=\left\{ 5, 9, {11}, {12} \right\}$ and so the horizontal local knot vector is $\mathsf{g}_{A,1}=\left\{ 1, 4, 5, 7,7 \right\}$.
The index $7$ is skipped because it does not have edges that span the thickened anchor line at the intersection.
It is not necessary to check the span in the vertical direction because the anchor has no width.
The vertical index vector is $\mathsf{i}_{A,2}=\left\{ 6, 7, 9, {10} \right\}$ and the vertical local knot vector is $\mathsf{g}_{A,2}=\left\{ 3,4,6,7 \right\}$.

The opposite mixed degree case $\vec{p}=\left\{ 2,3 \right\}$ is shown in part (c).
The global knot vectors are now
\begin{equation}
\mathsf{G}_1=\left\{ 0,0,0,1,2,3,4,5,6,7,7,7 \right\}
\end{equation} 
and 
\begin{equation}
\mathsf{G}_2=\left\{ 0,0,0,0,1,2,3,4,5,6,7,7,7,7 \right\}.
\end{equation}
The function anchors are now horizontal edges and so only the vertical anchor line must be thickened.
The indices for the horizontal local index vector are $\mathsf{i}_{A,1}=\left\{ 3,4,7,8 \right\}$ and the knot vector is $\mathsf{g}_{A,1}=\left\{ 0,1,4,5 \right\}$.
The horizontal anchor line has no width perpendicular to the horizontal direction and so only intersections must be checked.
The indices for the vertical local index vector are $\mathsf{i}_{A,2}=\left\{ 3,4,8,9,10 \right\}$; here all of the edges intersected span the anchor.
The vertical local knot vector is $\mathsf{g}_{A,2}=\left\{ 0,0,4,5,6 \right\}$.

The odd degree case $\vec{p}=\left\{ 3,3 \right\}$ is shown in part (d).
We choose the global knot vectors 
\begin{equation}
\mathsf{G}_1=\left\{ 0,0,0,0,1,2,3,4,5,6,7,7,7,7 \right\}
\end{equation}
and
\begin{equation}
\mathsf{G}_2=\left\{ 0,0,0,0,1,2,3,4,5,6,7,7,7,7 \right\}.
\end{equation}
For T-splines with odd degree in both directions, the anchors are vertices, the anchor lines are not thickened, and so only intersection must be checked.
The indices for the horizontal local knot vector are $\left\{ 4,5, 8, 9, {11} \right\}$ and the local knot vector is $\mathsf{g}_{A,1}=\left\{0,1,4,5,7\right\}$.
The indices for the vertical local knot vector are $\left\{ 3, 4, 8,
  9, {10} \right\}$ and so the local knot vector is
$\mathsf{g}_{A,2}=\left\{ 0,0,4,5,6 \right\}$.
Once the T-spline basis functions have been defined and control points
have been assigned to each one, the geometric map is given by 
\begin{equation}
\label{eq:t-spl-geom-map}
\vec{x}(s,t)=\sum_{A=1}^n\vec{P}_AN_A(s,t).
\end{equation}

\subsubsection{Face and edge extensions and analysis-suitable T-splines}
Although the T-spline blending functions defined in this fashion can
be used to define smooth surfaces, many properties of the resulting
space are not immediately obvious.
In order to develop T-splines that are well-characterized and suitable for analysis, we introduce the face and edge extensions of T-junctions.
A face extension is a closed line segment that extends from the
T-junction in the direction of the face of the cell at which the
T-junction terminates and that crosses $\lfloor(p_i+1)/2\rfloor$
perpendicular edges or vertices. 
The edge extension of a T-junction is a closed line segment that extends from the T-junction in the opposite direction of the face extension and that crosses $\lceil (p_i-1)/2\rceil$ perpendicular edges or vertices.
A T-mesh is analysis suitable if no vertical T-junction extension intersects a horizontal T-junction extension.
Note that because the edge and face extensions are closed segments, they can intersect at endpoints.
The face and edge extensions are shown for a T-mesh that is not analysis-suitable and for an analysis-suitable T-mesh in \cref{fig:extensions}.
The T-mesh formed by adding all face extensions to $\mathsf{T}$ is called the extended T-mesh and is denoted by $\mathsf{T}_{\mathrm{ext}}$.
\begin{figure}
  \centering
\begin{tabular}{*{2}{c}}
\includegraphics[scale=0.5]{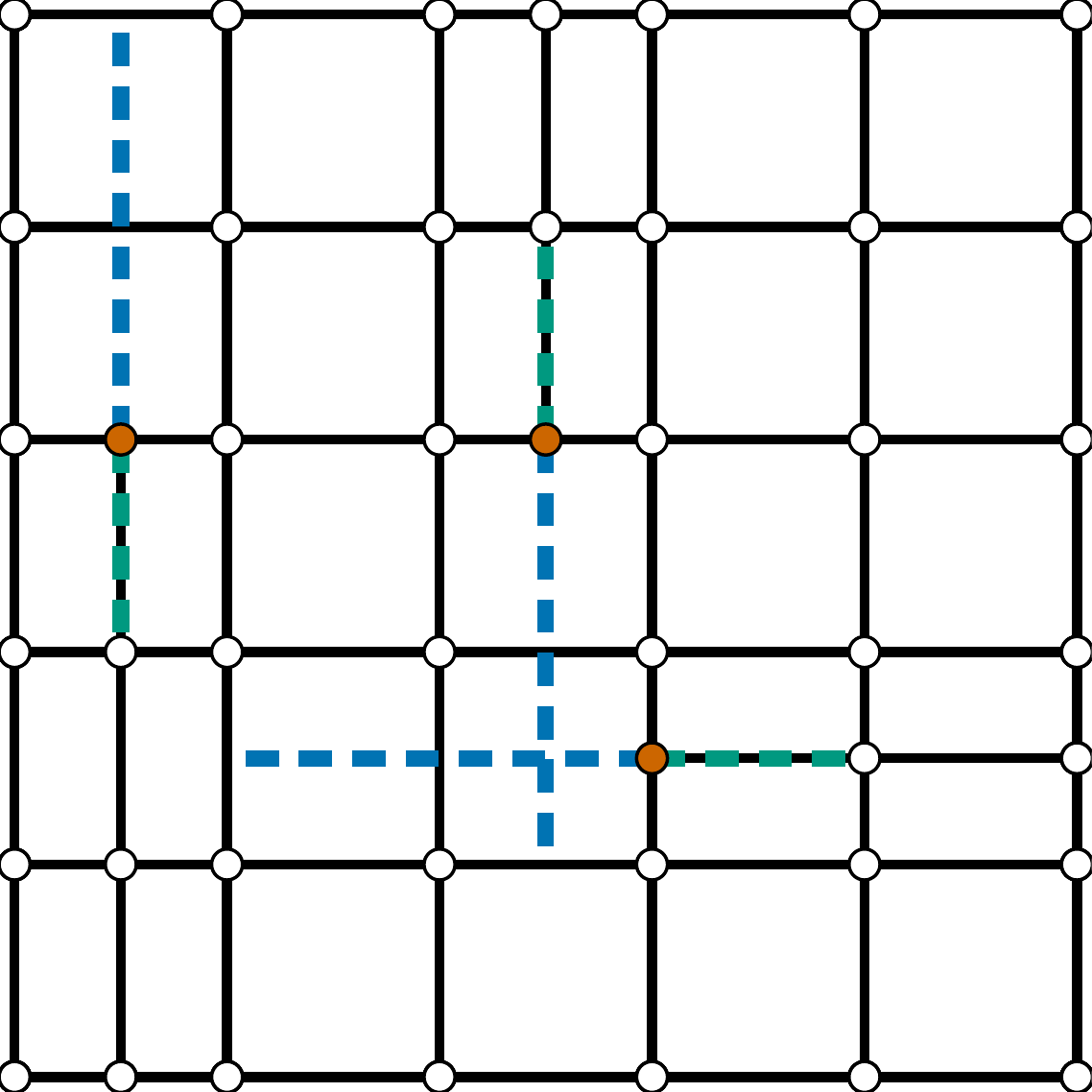}
&
\includegraphics[scale=0.5]{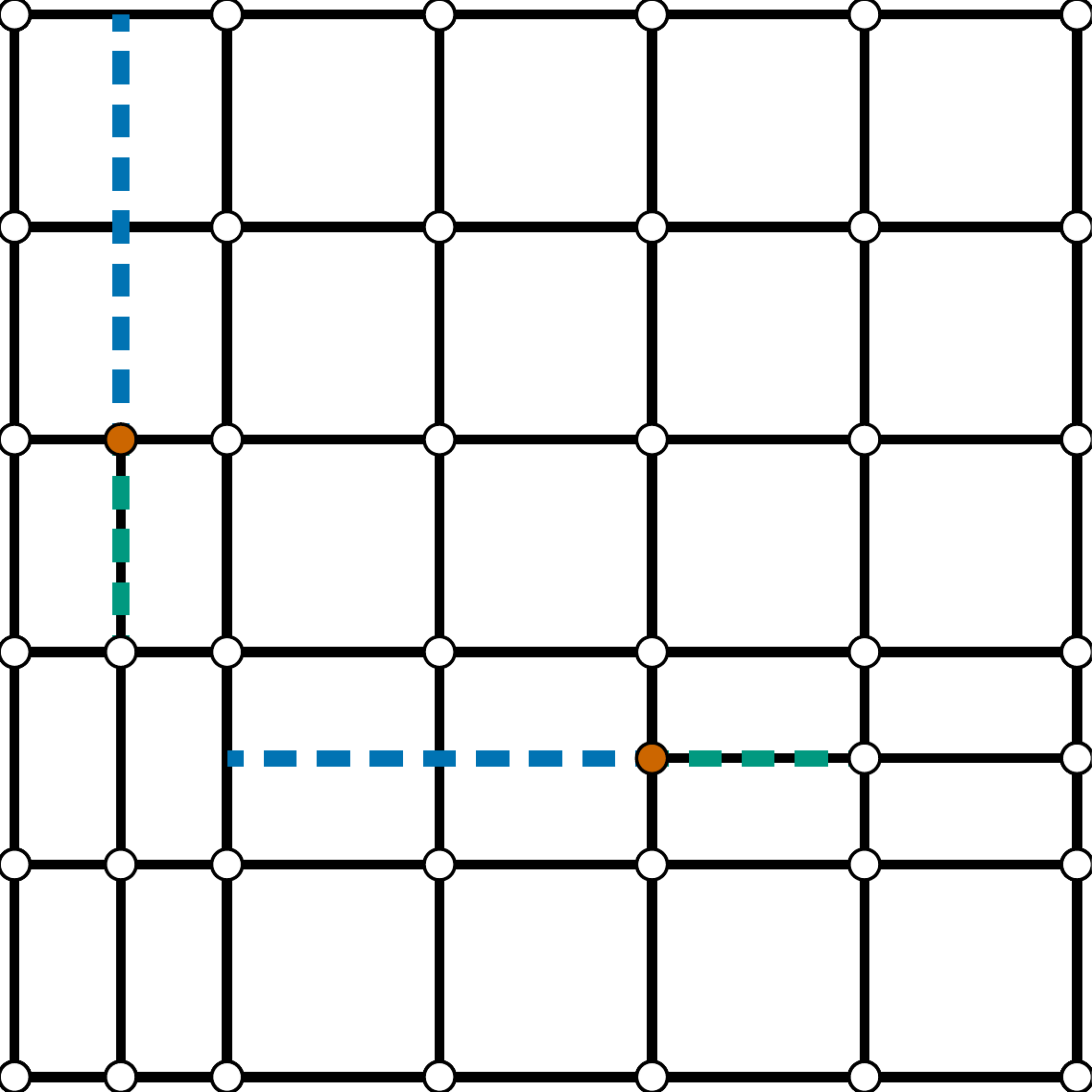}
\end{tabular}
\caption{T-junction extension in two dimensions for a bicubic T-spline.  Face extensions are shown in blue and edge extensions are shown in green.  The T-junctions are marked with orange circles.  The T-mesh on the left is not analysis-suitable while the T-mesh on the right is.}
\label{fig:extensions}
\end{figure}

The T-spline basis defined by an analysis-suitable T-mesh possesses many
important mathematical properties including the following theorems~\cite{EvScLiTh13,LiScSe12}:
\begin{theorem}\label{thm:ast-lin-indep}
  The basis functions of an analysis-suitable T-spline are locally linearly independent.
\end{theorem}
\begin{theorem}\label{thm:ast-complete}
  The basis functions of an analysis-suitable T-spline form a complete
  basis for the space of polynomials of degree $\vec{p}$.
\end{theorem}
\begin{theorem}
\label{thm:nested}
  The analysis-suitable T-spline spaces $\fspace{T}^a$ and
  $\fspace{T}^b$ are nested (i.e.,
  $\fspace{T}^a\subseteq\fspace{T}^b$) if
  $\mathsf{T}_{\mathrm{ext}}^a\subseteq\mathsf{T}_{\mathrm{ext}}^b$,
  that is, if $\mathsf{T}_{\mathrm{ext}}^b$ can be constructed by
  adding edges to $\mathsf{T}_{\mathrm{ext}}^a$.
\end{theorem}
We only consider analysis-suitable T-splines (ASTS) for the remainder of this work.
The spline space spanned by a T-spline basis defined by the T-mesh $\mathsf{T}^a$ is denoted by $\fspace{T}^a$.

\subsection{\Bezier extraction and spline reconstruction}
Although analysis-suitable T-splines possess the mathematical properties required by
analysis it is not immediately obvious how the basis can be integrated
into existing finite-element tools. 
One of the first issues that must be addressed is how a computational
mesh is obtained from the T-mesh. A simple and elegant solution to
this problem is based on \Bezier extraction introduced in Borden et
al.~\cite{Borden:2010nx} for NURBS and Scott et al.~\cite{ScBoHu10}
for T-splines. The \Bezier mesh $\mathsf{B}(\mathsf{T})$ is created by
adding the face extensions to $\mathsf{T}$ and then mapping the
resulting index mesh to the parametric domain. 
Note that for ASTS the edges in $\mathsf{B}(\mathsf{T})$ represent all
lines of reduced continuity in the T-spline basis. This makes the
\Bezier mesh the natural mesh for finite element analysis based on
T-splines since the basis is $C^{\infty}$ in the interior of each
\Bezier element. The image of the \Bezier mesh under the geometric
map~(\cref{eq:t-spl-geom-map}) generates the physical mesh. 
The \Bezier elements for a given mesh can be enumerated and so we refer to the elements by their index $e$.
The parametric domain of a \Bezier element $e$ is denoted by $\hat{\Omega}^e$
and the physical domain of a \Bezier element is denoted by $\Omega^e$.

\Bezier extraction generates the Bernstein-\Bezier representation of
the T-spline basis over each element. The resulting linear
relationship is encapsulated in the so-called \Bezier element
extraction operator denoted by $\mat{C}^e$.
Given the control values, $\vec{P}^e$, associated with
the spline basis functions which are nonzero over element $e$
the control values $\vec{Q}^e$ associated with the Bernstein basis
defined over the element are related to the spline control values using the
transpose of the element extraction operator
\begin{equation}
\label{eq:bez-ext-def}
\vec{Q}^e=\left( \mat{C}^e \right)^{\trans}\vec{P}^e.
\end{equation}
There is an element extraction operator associated with each element in the \Bezier mesh.
Algorithms for computing the element extraction operators for ASTS and
B-splines are given by \citet{Borden:2010nx} and
\citet{ScBoHu10}. \Bezier extraction is graphically demonstrated in~\cref{fig:bez-ext}.
The cubic B-spline curve and control points are shown in the upper
left and the corresponding B-spline basis in the lower left.
The basis is defined by the knot vector $[0,0,0,0,1,2,3,4,4,4,4]$.
The segment of the spline curve corresponding to the second element is
shown in the middle of the figure with its associated control points.
The spline basis supported by the second element is shown below. 
The \Bezier control points produced by the transpose of the element
extraction operator are shown in the upper right with the associated
Bernstein basis below.

\begin{figure}[htb]
  \centering
  \includegraphics[width=6in]{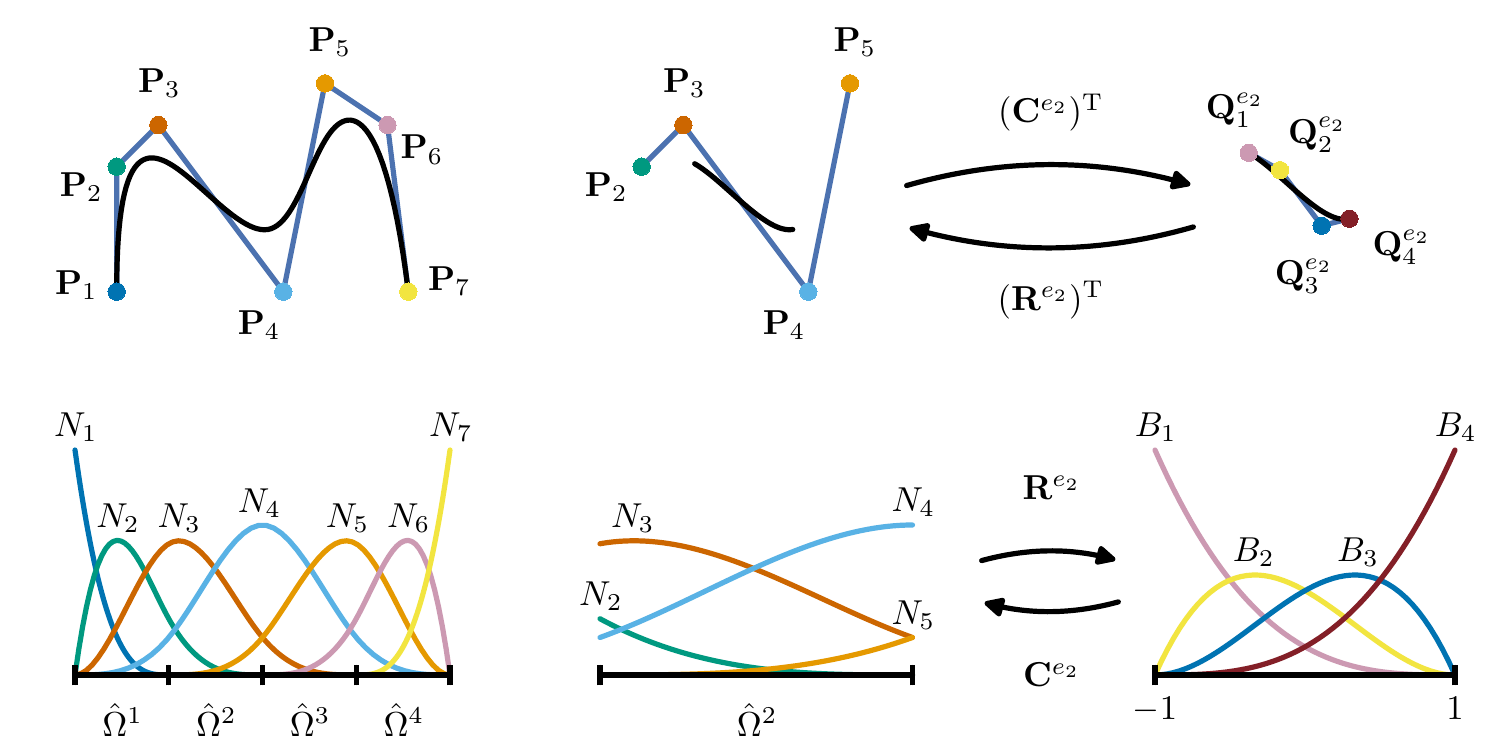}
  \caption{\label{fig:bez-ext}Illustration of the \Bezier element extraction
    operator $\mat{C}^e$ and the spline element reconstruction operator
    $\mat{R}^e$ for a spline of degree 3.  The Bernstein basis
    is shown over the biunit interval and so the spline basis segments
    are reconstructed by composing the Bernstein basis with the map
    from the biunit interval to the element.}
\end{figure}

\begin{lemma}\label{lemma:elem-ext-inv}
  The \Bezier element extraction operators for an analysis-suitable
  T-spline (and B-splines and NURBS) are invertible.
\end{lemma}
\begin{proof}
  The element extraction operator provides a map from the
  Bernstein polynomials to the T-spline basis functions over the
  element. Both sets are linearly independent and complete
  (\cref{lemma:bern-lin-indep,thm:ast-lin-indep}), therefore the
  element extraction operator is invertible. 
\end{proof}
The inverse of the \Bezier element extraction operator and its significance have not been considered previously.
It can be seen in \cref{fig:bez-ext} that the inverse transpose of the
\Bezier element extraction operator provides a means to convert the
\Bezier control points into spline control points. 
For this reason, we have termed the inverse of the \Bezier element
extraction operator the {\it spline element
  reconstruction operator}
\begin{equation}
\label{eq:rec-op-def}
\mat{R}^e\equiv(\mat{C}^{e})^{-1}
\end{equation}
or element reconstruction operator for short.
Whereas the \Bezier element extraction operator ``extracts'' \Bezier
coefficients the spline element reconstruction operator converts
\Bezier coefficients into spline coefficients, thus ``reconstructing''
the spline segment. 
Additionally, the element reconstruction operator can be used to
express the Bernstein basis in terms of the spline basis defined over
element $e$. The element reconstruction operator is a core component of
the \Bezier projection method developed in this paper.

\section{\Bezier projection}\label{sec:localized-projection}
\label{sec-2}
Given function spaces $\fspace{A}$ and $\fspace{B}$, we use
$\Pi[\fspace{A},\fspace{B}]$ to represent the projection from a
function in $\fspace{A}$ to a function in $\fspace{B}$. 
If $\fspace{A}\subseteq\fspace{B}$, then the projection is exact or injective.
Where the meaning is unambiguous, we use $\Pi[\fspace{B}]$ to
represent the projection onto $\fspace{B}$ or $\Pi$ to denote a
general projection. The definition of a projector requires that for
$f\in \fspace{A}$, $\Pi[\fspace{A}](f) = f$. 

We define our projection problem as follows: given a function $f$ in
some function space $\fspace{F}$ mapping a domain $\Omega\subset
\mathbb{R}^m$ to $\mathbb{R}^n$ and a discrete space of spline
functions, $\fspace{T}$, mapping $\hat{\Omega}\subset \mathbb{R}^m$ to
$\mathbb{R}$, find a set of coefficients or functionals (or vector of
functionals if $n>1$) $\{\lambda_i(f):\fspace{F}\rightarrow
\mathbb{R}^n\}$ such that the function given by
\begin{equation}
  \Pi[\fspace{F},\fspace{T}](f) = \sum_{A}\lambda_A(f)N_A,
\end{equation}
where $N_A$ are the basis functions of $\fspace{T}$ and
$\lambda_A(f)\in \mathbb{R}^n$ is the coefficient associated with the
$A$th basis function, approximates $f$ in some sense. 
The optimal projector returns the coefficients $\lambda_i(f)$ that
minimize the error with respect to the $L^k$ norm over the domain
$\Omega$ 
\begin{equation}
  \epsilon_k = \|f - \Pi(f)\|_k
\end{equation}
where $\|f\|_k = \left(\int_{\Omega}|f|^kd\Omega \right)^{1/k}$.
It is standard to use the $L^2$ norm.

If the domain $\Omega$ over which the function $f$ is defined does not
coincide with the domain $\hat{\Omega}$ over which the spline basis is
defined, a function mapping $\hat{\Omega}$ to $\Omega$,
$\psi:\hat{\Omega}\rightarrow\Omega$ must be introduced and the
composition $f\circ\psi$ is projected onto the spline basis. 
When projecting onto a spline basis over a geometry, this map is the
geometric map $\vec{x}(\vec{s})$ that defines the
geometry. For simplicity of exposition, in this section we assume that
the two domains coincide, that is $\Omega=\hat{\Omega}$.

In general, the functionals $\lambda_i$ of the global projection problem
require integration over the entire domain $\Omega$ and solution of a
linear system of the same size as the dimension of the spline space
$\fspace{T}$. A localized projection or quasi-interpolation \cite{sablonniere2005} is defined by choosing functionals $\lambda_i$
that can be determined from function values over a subdomain
$\Omega^{\prime}\subset \Omega$ and that do not require the solution
of a linear system of the same size as the spline space.

\subsection{Formulation of \Bezier projection}
We introduce \Bezier projection as a localized projection operation
that uses a linear combination of projections onto the element Bernstein
basis. Given a map
$\phi_e:[-1,1]\rightarrow\hat{\Omega}^e\subseteq\hat{\Omega}$ from the
biunit interval to the knot interval (or element) $e$, we define the
projector $\Pi^e[\fspace{B}^p]:\fspace{F}\rightarrow \fspace{B}^p$ of
a function $f\in\fspace{F}$ over the interval $e$ onto the Bernstein
basis over the biunit interval as
\begin{equation}
  \Pi^e[\fspace{B}^p](f) = \sum_{i=1}^{p+1}\beta_i^{p,e}(f) B_{i}^p
\end{equation}
where the functionals
$\beta_i^{p,e}:\fspace{F}\rightarrow\mathbb{R}^n$ are obtained from
the projection of $f\circ\phi_e$ onto the Bernstein basis of degree
$p$. When the meaning is unambiguous, we suppress the superscript $p$ for the polynomial degree.
Note that the original function is defined over the domain $\Omega = \hat{\Omega}$
while the definition of local projection $\Pi^e[\fspace{B}^p]$
presented here introduces a transformation so that the projected
segment of $f$ is represented over the biunit interval. 
%It is often convenient to define the vector of functionals or Bernstein coefficients $\vec{\beta}^e(f)$.
The functionals $\beta_i^{p,e}$ with respect to the $L^2$ norm are found from the linear system
\begin{equation}
\label{eq:bernstein-proj}
\sum_{j=1}^{p+1}(B_i^p,B_j^p)\beta_j^{p,e}(f)=(B_i^p,f\circ \phi_e)
\end{equation}
where the $L^2$ inner product is defined as
\begin{equation}
\label{eq:inner-prod}
(f,g)=\int_{-1}^1f(\xi)g(\xi)d\xi.
\end{equation} 
If $\Omega \neq \hat{\Omega}$ it would be necessary to introduce an
additional transformation $\psi:\hat{\Omega}\rightarrow\Omega$. In
this case, the second entry in
the inner product on the right-hand side of \cref{eq:bernstein-proj}
would be $f\circ\psi\circ\phi_e$.
The solution to \cref{eq:bernstein-proj} can be written as
\begin{equation}
\label{eq:bernstein-proj-sol-mat}
\vec{\beta}^{e}=\mat{G}^{-1}\vec{b}
\end{equation}
by using the inverse of the Gramian matrix and defining the vector of functionals (Bernstein coefficients)
$\vec{\beta}^e(f)=\left\{ \beta_i^{p,e}(f) \right\}_{i=1}^{p+1}$ and
the vector of basis function-function inner products $\vec{b}=\left\{
  (B_i^p,f\circ\phi_e) \right\}_{i=1}^{p+1}$.
The inverse of the matrix $\mat{G}$ is given in closed form by
\cref{eq:inv-gramian-expr} and so the solution can be computed
directly without a numerical solution step. 

\Bezier extraction can be interpreted as a projection of a spline
basis function onto the Bernstein basis:
\begin{equation}\label{eq:alt-ext-def}
  \Pi^e[\fspace{B}^p](N_i) = \sum_{i=1}^{p+1}c_{ij}^eB_{i}^p,
\end{equation}
where $c_{ij}^e$ are the entries of the element extraction operator $\mat{C}^e$.
Recall that the function $N_i$ is defined over the parametric domain
of the spline $\hat{\Omega}$ while the Bernstein representation on the
right-hand side of~\cref{eq:alt-ext-def} is defined over the biunit
interval. The coefficients of the spline basis functions, denoted by
$\lambda_A(f)$, over element $e$ 
are related to the \Bezier coefficients of the Bernstein basis by the
element reconstruction operator
\begin{equation}\label{eq:elem-lambda-def}
  \vec{\lambda}^e(f) = (\mat{R}^e)^{\trans}\vec{\beta}^e(f).
\end{equation}

\subsubsection{\Bezier projection weighting scheme}
In general, the control value for a given function produced by \Bezier
projection will be different for each element in the support of the
function. These values must be averaged, selected, or combined in some
way to generate a unique global control point. We choose to construct
the global set of coefficients from a weighted sum of the local
coefficients
\begin{equation}\label{eq:bez-proj-global-coeffs}
  \lambda_A(f) = \sum_{e\in \mathsf{E}_A} \omega_A^e \lambda_A^e(f)
\end{equation}
where $\mathsf{E}_A$ contains the elements in the support of the $A$th
basis function and $\lambda_A^e$ represents the coefficient of basis function $A$ on element $e$.
These results can be combined to express the full \Bezier projection as
\begin{equation}
\Pi _{B}[\fspace{F},\fspace{T}](f) = \sum_A\left[\sum_{e\in
    \mathsf{E}_A} \omega_A^e \lambda_A^e(f)\right]N_A. 
\end{equation}
To guide the development of our weighting scheme we recall that
\begin{lemma}\label{thm:same-coeffs}
  For a spline function $T\in \fspace{T}$ we have that $\lambda_A^e(T)
  = \lambda_A(T)$ for all $e\in \mathsf{E}_A$.
\end{lemma}
\begin{proof}
  This follows directly from the definition of Bézier extraction and the element
  extraction operator and its invertibility. 
\end{proof}
\begin{lemma}\label{thm:weight-sum}
If $\sum_{e = \mathsf{E}_A}\omega_A^e = 1$ then $\Pi _{B}$ is a projector.
\end{lemma}
\begin{proof}
  Given a function $T\in \fspace{T}$ if the weights do not sum
  to one then the weighted sum of the
  local coefficients cannot be equal to the global coefficient and $\Pi_B[\fspace{T}](T)
  \neq T$ so $\Pi_B$ is not a projector.
\end{proof}
We note that by choosing the weights as
$\omega_A^e=1/n^e_A$ where $n^e_A$ is the number of elements in the
support of the $A$th basis function, the local least-squares method of 
\citet{govindjee2012} is obtained; however, it will be seen that the weighting proposed here
provides significantly increased accuracy in the results. 

A particularly accurate choice of weights is
\begin{equation}
  \label{eq:weight-def}
  \omega_A^e = \frac{\int_{\Omega^e}N_Ad\Omega}{\sum_{e' = \mathsf{E}_A}\int_{\Omega^{e'}}N_A d\Omega}.
\end{equation}
When projecting a geometry onto a new basis, the spatial domain
$\Omega_A$ is not defined and the parametric domain $\hat{\Omega}_{A}$
is used instead. The weights obtained by this method for the basis functions defined by
the local knot vectors $[0,0,0,\sfrac{1}{3}]$,
$[0,0,\sfrac{1}{3},\sfrac{2}{3}]$, $[0,\sfrac{1}{3},\sfrac{2}{3},1]$,
$[\sfrac{1}{3},\sfrac{2}{3},1,1]$, and $[\sfrac{2}{3},1,1,1]$ are
shown in \cref{fig:weights}. 
\begin{figure}[htb]
  \centering
  \includegraphics[width=5in]{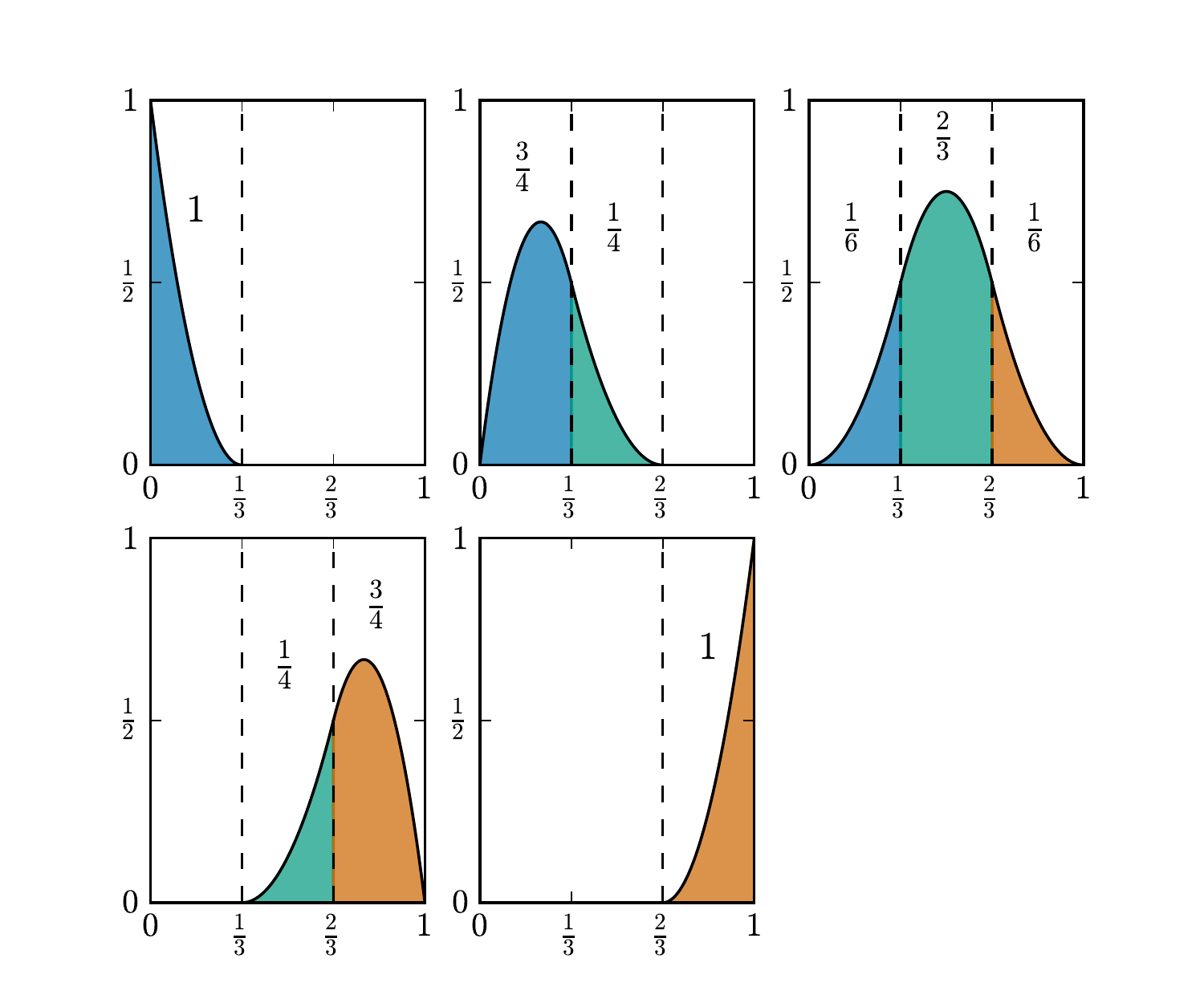}
  \caption{\label{fig:weights}Weights over each knot span associated
    with the basis function defined by the local knot vectors
    $[0,0,0,\sfrac{1}{3}]$, $[0,0,\sfrac{1}{3},\sfrac{2}{3}]$,
    $[0,\sfrac{1}{3},\sfrac{2}{3},1]$,
    $[\sfrac{1}{3},\sfrac{2}{3},1,1]$, and $[\sfrac{2}{3},1,1,1]$.} 
\end{figure}

The individual steps comprising the \Bezier projection algorithm are
illustrated in \cref{fig:loc-proj-example}. 
The curve defined by $\vec{f}(t)=\left( \frac{t}{3}
\right)^{3/2}\vec{e}_1+\frac{1}{10}\sin (\pi t )\,\vec{e}_2$,
$t\in[0,3]$ is projected onto the quadratic B-spline basis defined by
the knot vector $[0,0,0,\sfrac{1}{3},\sfrac{2}{3},1,1,1]$.
Because the domain of the curve parameter $t$ does not coincide with
the parametric domain of the spline, we introduce the affine map
$\psi:[0,1]\rightarrow[0,3]$ and project $\vec{f}\circ\psi$ onto the
spline basis using \Bezier projection. 
All of the basis functions are shown in \cref{fig:weights} along with
the weight associated with each function over each element. 
The first step is to perform a projection onto the Bernstein basis for
each element to obtain the local \Bezier coefficients that define an
approximation to the target function over the element. 
The \Bezier control points are indicated in part (1) of the figure by 
square markers that have been colored to match the corresponding
element. The local approximation to the target function is shown along
with the control points. 
Because the local \Bezier control points are interpolatory at the ends
of the segments, it can be seen from the placement of the \Bezier
control points associated with adjacent segments that the \Bezier
curve segments are discontinuous. 

Next, the element reconstruction operator is used to convert the
\Bezier control points into spline control points associated with the
basis function segments over each element. 
This operation does not change the discontinuous segments that
approximate the target function, but rather changes the basis used to
represent those segments from the Bernstein basis to the spline
basis. The new control points are marked with inverted triangles and
again colored to indicate the element with which the control point is
associated. The control points occur in clusters.
The clusters of control points represent the contributions from
multiple elements to a single spline basis function control point. 
The endpoints have contributions from a single element.
The points to the right and left of either endpoint have contributions from two elements.
It can be seen that the center control point contains contributions from each of the 3 elements of the mesh.
Again, the discontinuous nature of the approximation segments can be
discerned by observing that the control points are slightly
scattered.

Each cluster of control points must be combined (averaged) to obtain a
single control point associated with the respective basis function. 
A weighted average of the points in each cluster is computed using the
weighting given in \cref{eq:weight-def} and the resulting control
points are shown as circles with the relative contribution from each
element to each control point indicated by the colored fraction of the
control point marker. The weights are those from~\cref{fig:weights}.
The colors in \cref{fig:weights,fig:loc-proj-example} are coordinated
to illustrate where the averaging weights come from and their values.
To summarize, 
\begin{algorithm}\label{alg:local-proj}
  \Bezier projection onto a spline basis.
  \begin{enumerate}
  \item Compute the projection of the target function $f$ onto the
    Bernstein basis on each element to obtain the vector of local
    \Bezier control values (\cref{eq:bernstein-proj}). 
  \item Use the transpose of the element reconstruction operator
    to convert the local \Bezier values to spline control values for
    the element segments (\cref{eq:elem-lambda-def}). 
  \item Use the weighting defined in \cref{eq:weight-def} to compute
    the global spline coefficients from the local spline control
    values as given by \cref{eq:bez-proj-global-coeffs}. 
  \end{enumerate}
\end{algorithm}
Alternatively, steps 1 and 2 of \cref{alg:local-proj} can be combined
by projecting the target function $f$ directly onto the local segments
of the global spline basis. 

\begin{figure}[htb]
  \centering
  \begin{tabular}{c p{4in} p{2in}}
    (0)&\raisebox{-.5\height}{\includegraphics[width=4in]{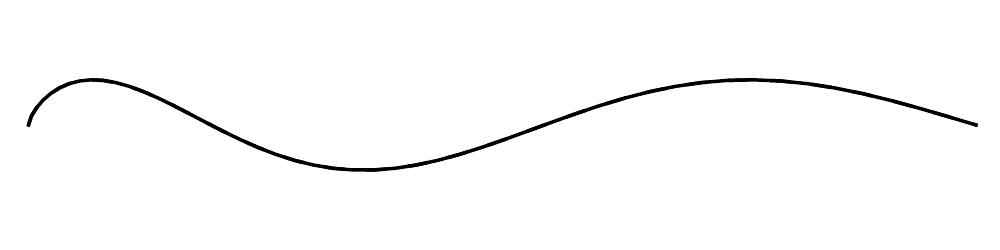}} & \begin{minipage}[t]{2in}Target function\end{minipage}\\
    (1)&\raisebox{-.5\height}{\includegraphics[width=4in]{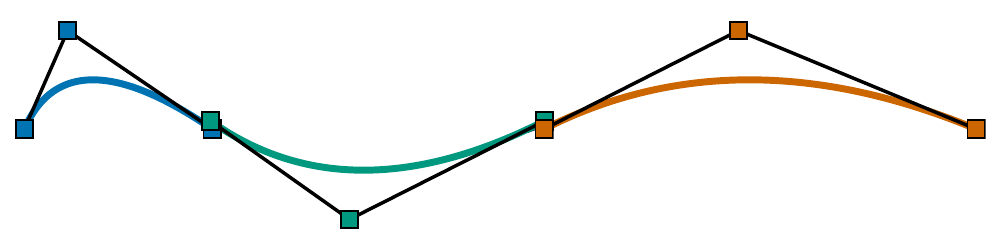}} & \begin{minipage}[t]{2in}Perform local projection to obtain Bézier control points (represented by squares, colored to match elements)\end{minipage}\\
    (2)&\raisebox{-.5\height}{\includegraphics[width=4in]{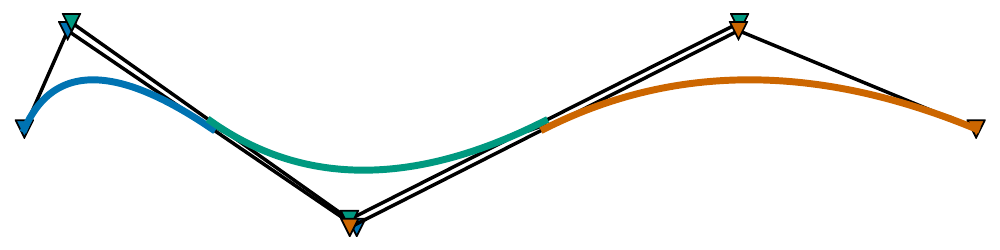}} & \begin{minipage}[t]{2in}Use element reconstruction operator to project \Bezier points to spline control points (represented by inverted triangles, colored to match elements)\end{minipage}\\
    (3)&\raisebox{-.5\height}{\includegraphics[width=4in]{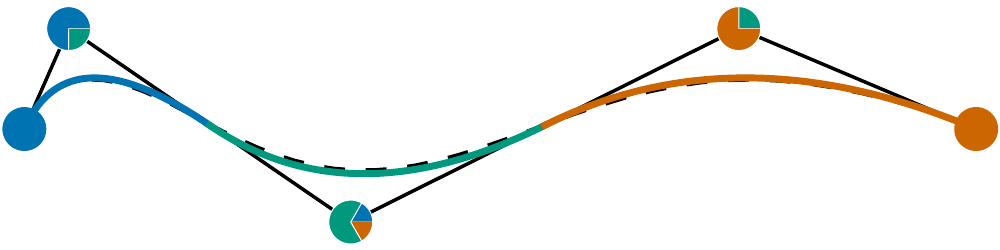}} & \begin{minipage}[t]{2in}Apply smoothing algorithm (contribution of each element to each control point shown by colored fraction)\end{minipage}\\
    (4)&\raisebox{-.5\height}{\includegraphics[width=4in]{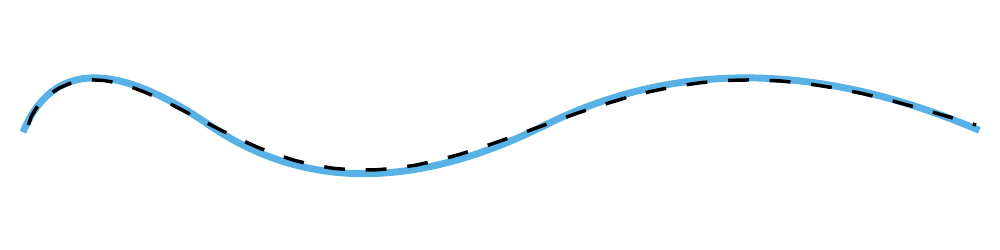}} & \begin{minipage}[t]{2in}Comparison of final function (light blue) and target function (dashed)\end{minipage}
  \end{tabular}
  \caption{\label{fig:loc-proj-example}Steps of \Bezier projection.}
\end{figure}

When projecting onto splines with parametric cells that are
proportionally similar to the physical elements it may not be
necessary for accuracy
to compute the physical integrals appearing in
\cref{eq:weight-def}. Recall that $\int_a^b B_i^p(\xi)
d\xi = (b-a)/(p+1)$. Therefore the weight associated with function
$A$ over element $e$ may be approximated by
\begin{equation}
  \label{eq:weight-ext-form}
  \omega_A^e =
  \frac{\textrm{vol}(\hat{\Omega}^e)\sum\limits_{i=1}^{p+1}c^e_{A,i}}{\sum\limits_{e^{\prime}
      =
      \mathsf{E}_A}\textrm{vol}(\hat{\Omega}^e)\sum\limits_{i=1}^{p+1}c^{e^{\prime}}_{A,i}} 
\end{equation}
where $\textrm{vol}(\hat{\Omega}^e)$ represents the volume of the
parametric domain associated with element $e$.

\subsection{Projection onto a rational basis}
When projecting a function $f$ onto a rational basis it is assumed
that a weight coefficient is given for each basis function. 
These weight coefficients define the weight function given in \cref{eq:weight-func-def}.
Rather than projecting directly onto the rational basis, we choose to
compute a set of homogeneous coefficients $\tilde{\lambda}_A(f)$. 
These coefficients are related to the coefficients of the rational basis functions by
\begin{equation}
\label{eq:homo-coeffs-trans}
\lambda_A(f)=\frac{\tilde{\lambda}_A}{w_{A}}.
\end{equation}
Leveraging homogeneous coefficients allows us to use \Bezier
projection as formulated for projection onto a polynomial spline basis with the exception
that the function that we project is
$w(\xi)f(\phi(\xi))$. The element extraction operator can be
used to convert the spline weight coefficients to \Bezier weight
coefficients. 
The \Bezier weight coefficients can then be used to compute the weight
function over element $e$.  In one dimension, the element weight
function is
\begin{equation}
\label{eq:bez-weight}
w^e(\xi)=\sum_{i=1}^{p+1}w_i^eB_i^p(\xi).
\end{equation}
The homogeneous \Bezier coefficients over the element $e$ are then given by
\begin{equation}
\label{eq:homo-coeff-lin}
\tilde{\vec{\beta}}^e=\mat{G}^{-1}\tilde{\vec{b}}
\end{equation}
where $\mat{G}^{-1}$ is given by \cref{eq:inv-gramian-expr} and the
entries in the vector $\tilde{\vec{b}}$ are given in one dimension by 
\begin{equation}
\label{eq:bern-weight-int-def}
\tilde{b}_i=\int_{-1}^1B_i^p(\xi)w^e(\xi)f(\phi_e(\xi))d\xi.
\end{equation}
The \Bezier projection process then proceeds as outlined from
\cref{eq:elem-lambda-def} onward with all $\beta$ and $\lambda$
variables replaced by their homogeneous counterparts $\tilde{\beta}$
and $\tilde{\lambda}$. 

\subsection{Convergence}\label{sec:convergence}
We illustrate the convergence of the \Bezier projection method with
the weighting defined in \cref{eq:weight-ext-form} by comparing the
error in the global $L^2$ projection to the error in \Bezier projection.
We use a single cycle of of a sinusoid over the unit interval given by $f(x)=\sin(2\pi x)$ to test the convergence.
The $L^2$ error of the projection of the sine function onto a spline
basis using \Bezier projection is compared to the global $L^2$
projection in \cref{fig:convergence} for splines with polynomials
degree ranging from 2 to 5. 
It can be seen that the \Bezier projection preserves the optimal
convergence rates of the underlying basis for all degrees. 
Furthermore, the results of \Bezier projection converge rapidly to the global projection results with only weak dependence on the polynomial degree of the basis.
% We hypothesize that the weak dependence on the polynomial degree of the spline space is due to the increased support of the spline basis functions and that it may be possible to obtain improved agreement between the localized and global projections by selecting a weighting scheme that more strongly weights the elements near the center of the function.
% We also note the presence of an increased numerical ``floor'' for high-degree splines.
% This limit can be attributed to two sources.
% Both are related to the conditioning of the basis.

\begin{figure}[htb]
  \centering
  \includegraphics[width=5in]{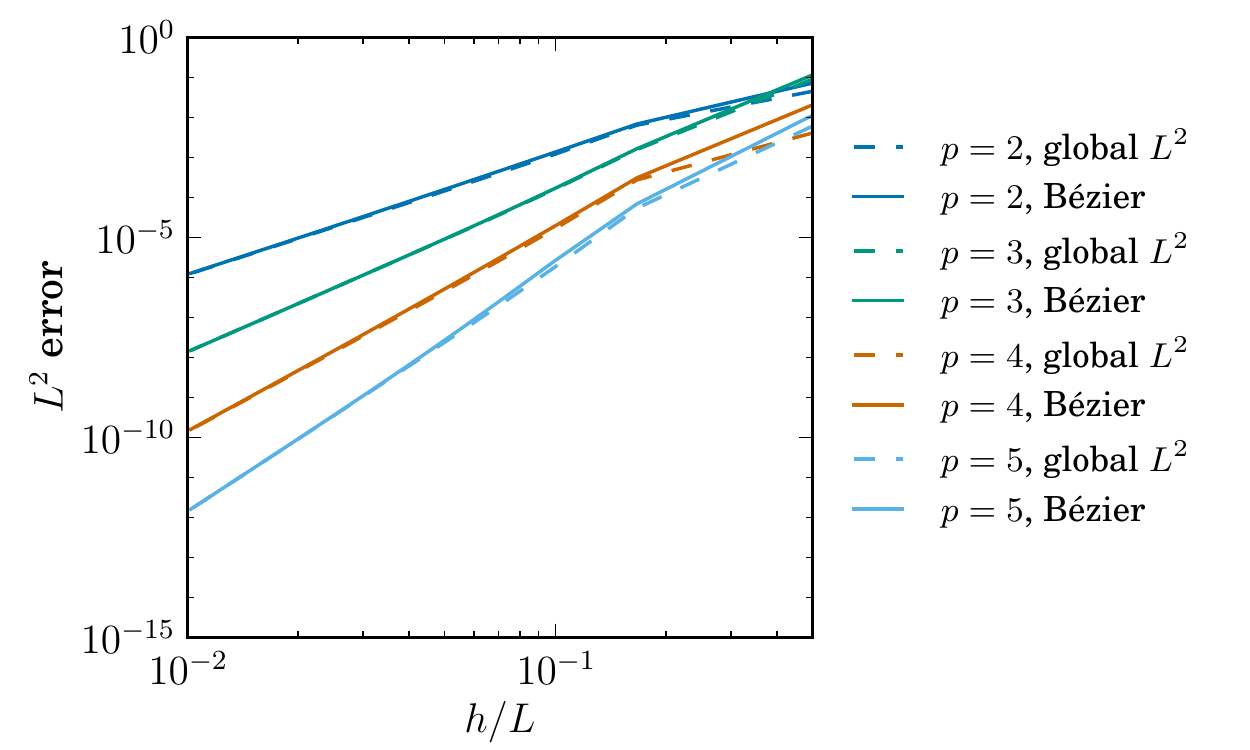}
  \caption{\label{fig:convergence}Convergence of the \Bezier projection method for the projection of a sine function onto a uniform B-spline basis.}
\end{figure}

We now consider the benchmark problem proposed by \citet{govindjee2012} in which the function
\begin{equation}
\label{eq:gov-func}
f(x,y)=\sin \left( \frac{3\pi x}{\sqrt{2}R}\right)\sin \left(\frac{ 2\pi y}{L} \right)
\end{equation}
is projected onto the rational spline basis that defines a quarter cylinder of length $L$ and radius $R$ that is positioned so that one flat edge of the shell lies on the $y$ axis with the lower corner at the origin and the other lies in the $z=0$ plane along the line $x=\sqrt{2}R$.
The geometry is illustrated in part (a) of \cref{fig:govindjee-surf} and the evaluation of the function $f$ on the surface is shown in part (b) of the same figure.

The convergence of the \Bezier projection is compared to global $L^2$ projection in \cref{fig:govindjee-benchmark}.
It can be seen that the \Bezier projection converges optimally.
There is an apparent floor for the convergence of the $p=5$ case that can be attributed to the conditioning of the Bernstein basis and the element extraction operators.
Addressing this issue is beyond the scope of this paper but will be treated in a future paper.
\begin{figure}[htb]
  \centering
  \begin{tabular}{ *{2}{c}}
    \includegraphics[width=2in]{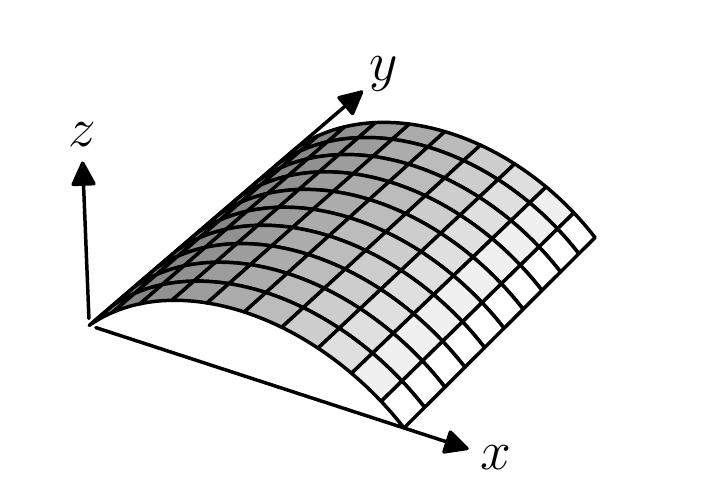}&\includegraphics[width=2in]{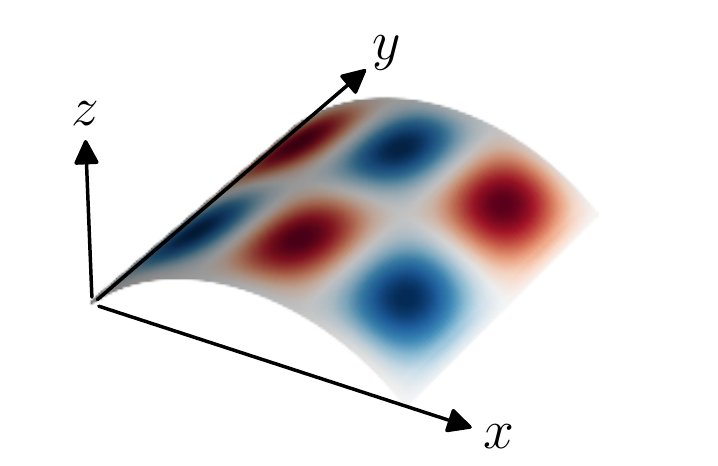}\\
    (a)&(b)
  \end{tabular}
  \caption{\label{fig:govindjee-surf} Geometry for the Govindjee
    benchmark problem.}
\end{figure}

\begin{figure}[htb]
  \centering
  \includegraphics[width=5in]{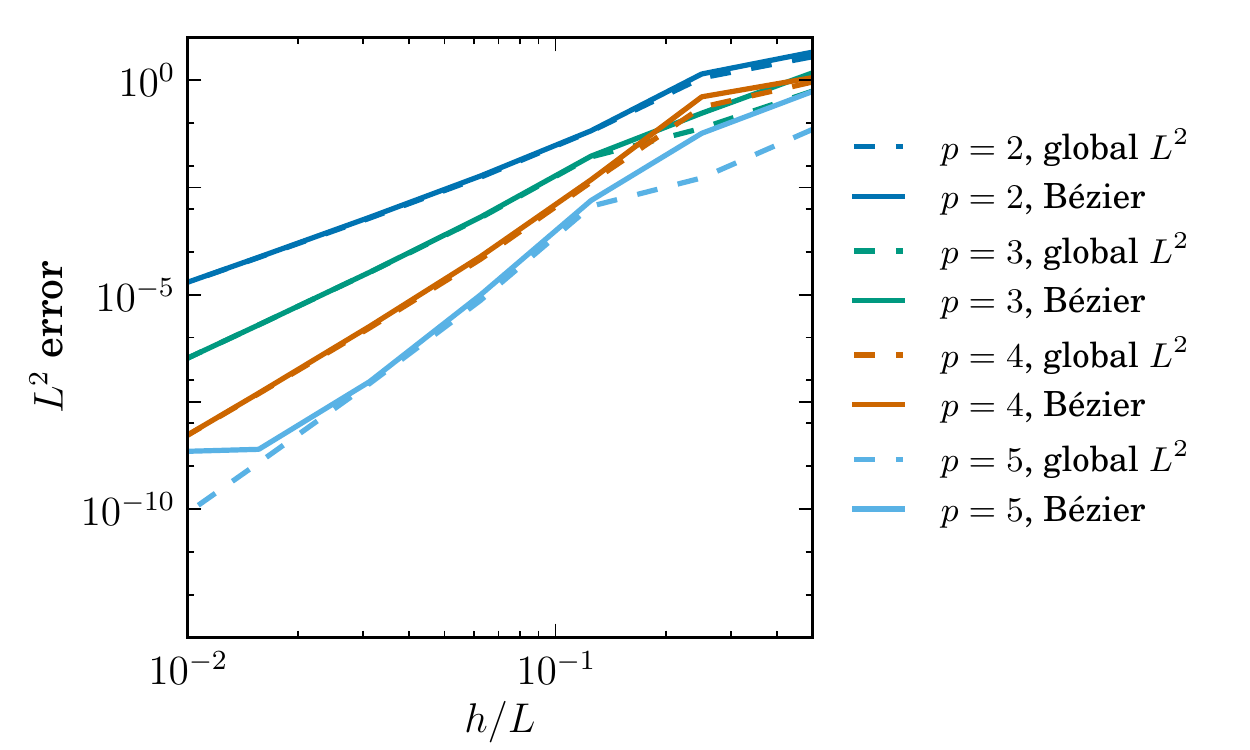}
  \caption{\label{fig:govindjee-benchmark}Convergence plots for the Govindjee benchmark problem.}
\end{figure}

The optimal convergence rates observed in \cref{fig:govindjee-benchmark,fig:convergence} support the following theorem.
\begin{theorem}
  The \Bezier projector $\Pi_B:\fspace{F}\rightarrow\fspace{T}$ exhibits optimal convergence rates.
\end{theorem}
\begin{proof}
\noindent See \cref{sec:proof-cont-proj}.
\end{proof}
\subsection{Applications and examples}
\subsubsection{Lifting of a surface normal field to spline control vectors}
For many operations involving surfaces it is advantageous to have an
accurate but approximate spline representation of the normal field.
The normal field $\vec{n}$ can be approximated by finding a set of ``control vectors'' that define a vector field over the surface:
\begin{equation}
\label{eq:spline-vec}
\hat{\vec{n}} = \sum_A\vec{V}_AN_A.
\end{equation}
The control vectors can be thought of as vectors anchored to the control points that define the surface.
We refer to the process of calculating control vectors from a vector
field defined over a surface as ``lifting'' the field off the surface onto the control points.
The control vectors which approximate a normal field can be easily
computed with \Bezier projection.
\begin{figure}
  \centering
  \includegraphics[width=4in]{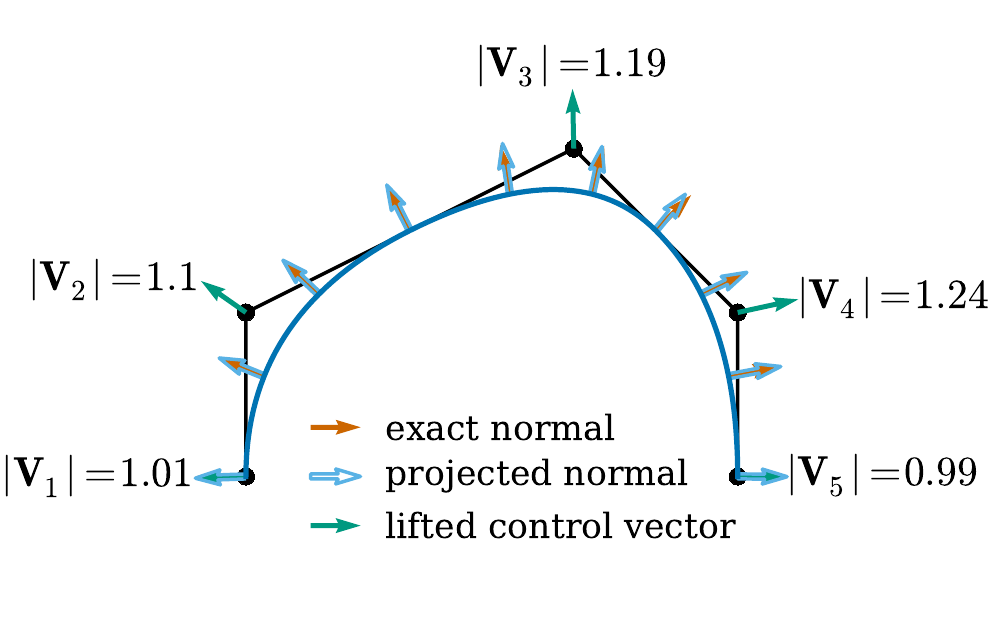}
  \caption{Lifting of normals} 
  \label{fig:normal-project}
\end{figure}
The process is illustrated in \cref{fig:normal-project}.
The normal field $\vec{n}$ for the curve is depicted with orange arrows.
The control vectors $\vec{V}_A$ obtained from \Bezier projection of the normal field onto the spline basis are shown in green on the associated control points.
The magnitute of the control vectors is also shown.
It is interesting to note that the control vectors for the approximate normal field have magnitude greater than 1.
The approximate projected normal field $\hat{\vec{n}}$ given by \cref{eq:spline-vec} is represented by the blue empty arrows.
The approximate normal field represents the normal field well except in regions of increased curvature.
Because the \Bezier projection method enjoys optimal convergence rates,
the accuracy can be improved by increasing the polynomial degree of
the spline basis used to represent the normal field or by subdividing
knot intervals.

A more complex T-spline normal lifting example is shown
in~\cref{fig:ts-smooth-hull}. A smooth containership hull is modeled
using bicubic T-splines. The 
Autodesk T-spline plugin for Rhino is used to model the surface \cite{TSManual12} and
the \Bezier extraction of the surface is then automatically exported for further
processing. Note that once the \Bezier extraction is computed, no further
information from the original CAD model is required. The T-splines
in~\cref{fig:ts-smooth-hull}a and \cref{fig:ts-smooth-hull}c are
composed of 36 \Bezier elements and 75 control 
points. The globally-refined T-splines in \cref{fig:ts-smooth-hull}b
and \cref{fig:ts-smooth-hull}d are
composed of 156 \Bezier elements and 221
control points. In~\cref{fig:ts-smooth-hull}a the magnitude of
the error in the projected normal field is shown. Notice the expected
concentration of error in regions of high
curvature. To improve the accuracy of the projected normal
field the coarse containership hull is globally refined to produce the
T-spline in~\cref{fig:ts-smooth-hull}b. Notice the
dramatic improvement in the accuracy of the 
projected normal field. \cref{fig:ts-smooth-hull}c and \cref{fig:ts-smooth-hull}d compare the
exact normals (blue arrows) to the projected normals (red arrows) at the corners of each
\Bezier element for the coarse and fine T-spline.

The projected normal field can be used to automatically generate a 
thickened shell geometry as shown in~\cref{fig:ts-smooth-hull-thick}. We
feel that this approach has the potential to provide a rigorous
geometric foundation for structural mechanics
applications based on plate and shell models.

\begin{figure}[htb]
  \centering
  \begin{tabular}{ *{1}{c} }
    \includegraphics[width=4in]{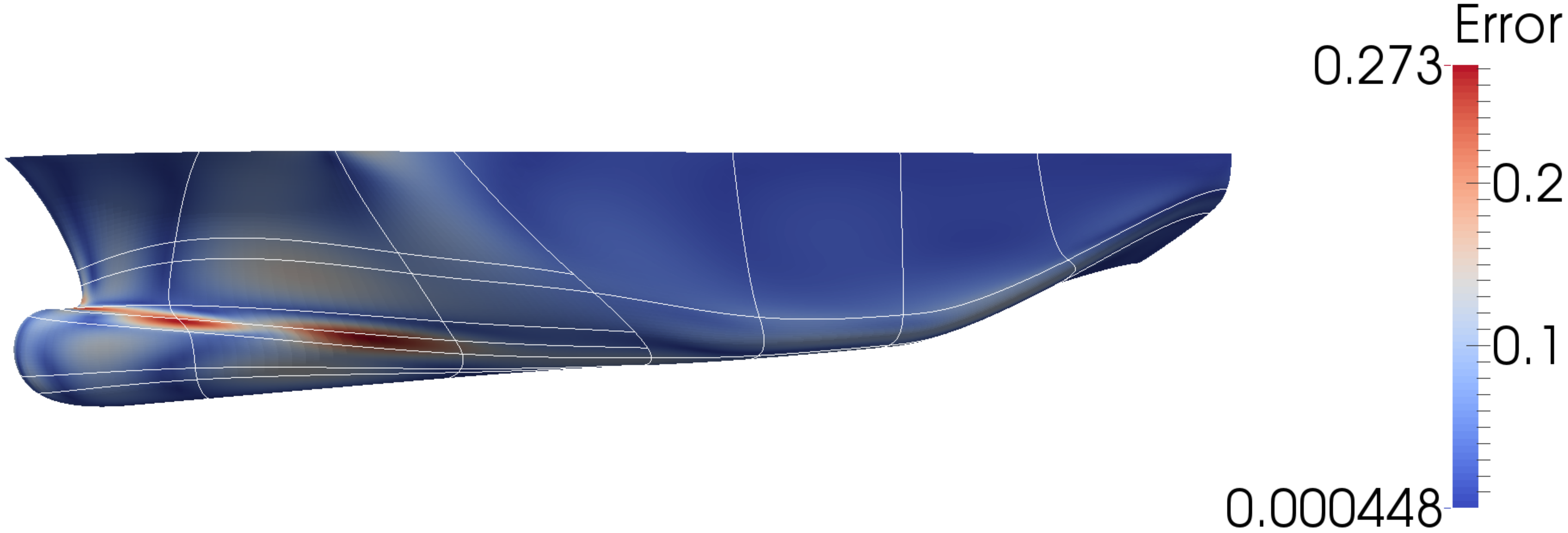}\\ \vspace{10pt}
    (a) Magnitude of error in projected normals for coarse T-mesh \\\vspace{10pt}
    \includegraphics[width=4in]{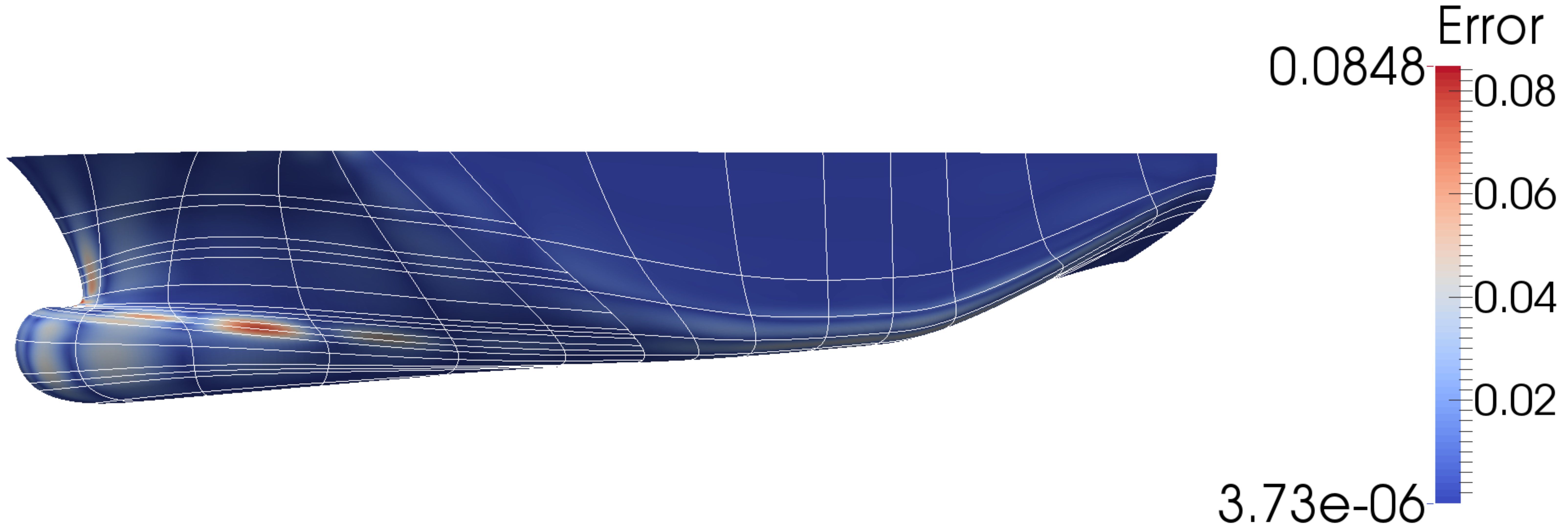}\\\vspace{10pt}
    (b) Magnitude of error in projected normals for fine T-mesh \\\vspace{10pt}
  \includegraphics[width=4in]{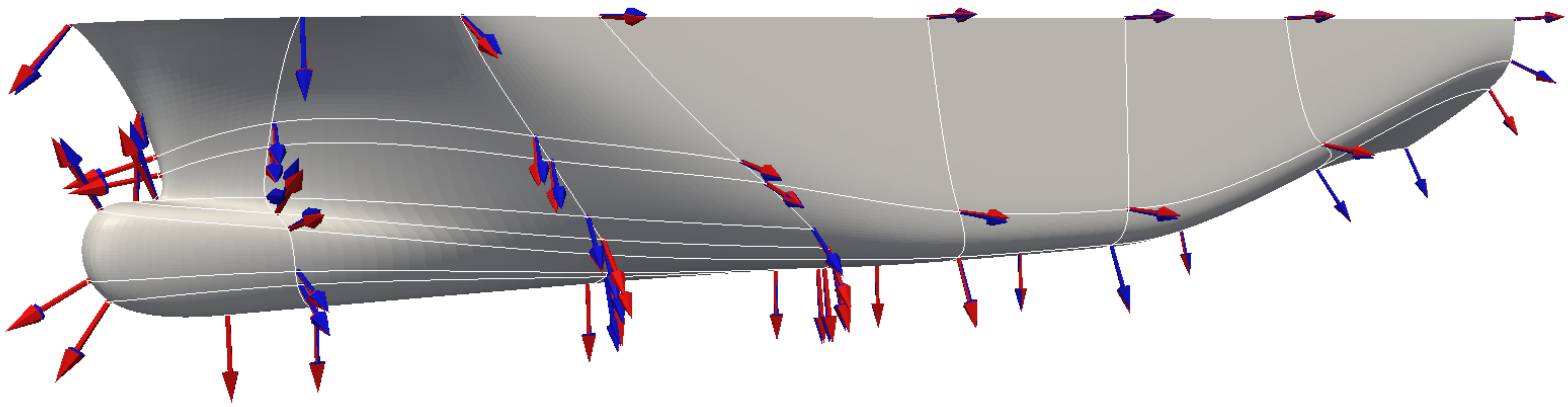}\\\vspace{10pt}
    (c) Exact normals (blue arrows) compared to projected normals (red
    arrows) for coarse T-mesh \\\vspace{10pt}
    \includegraphics[width=4in]{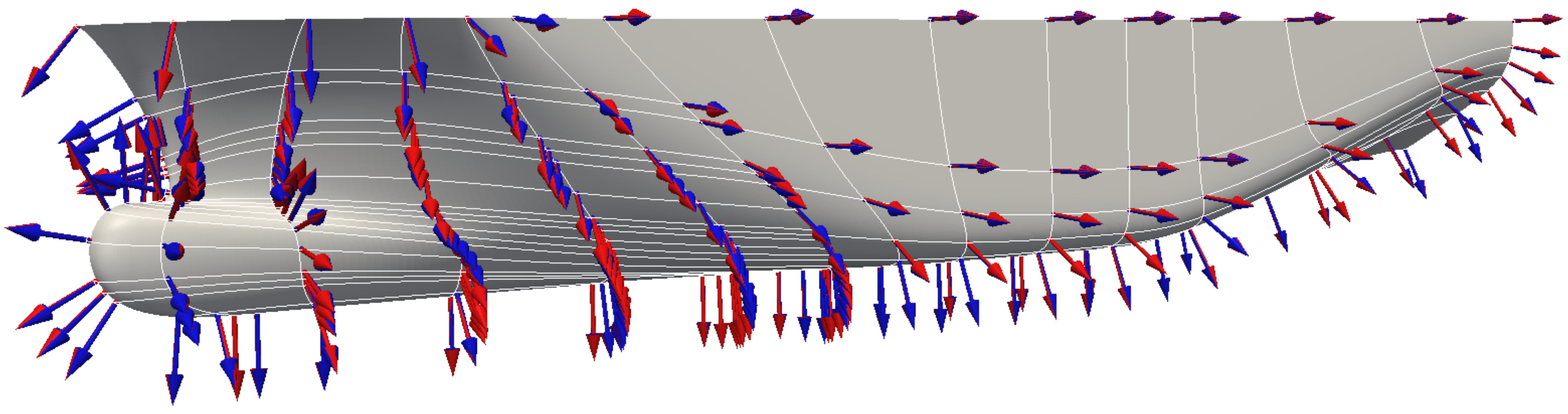}\\\vspace{10pt}
    (d) Exact normals (blue arrows) compared to projected normals (red
    arrows) for fine T-mesh.
  \end{tabular}
  \caption{A bicubic T-spline containership hull. The magnitude of the error
    in the projected normal field for a coarse (a) and fine (b) T-mesh. The exact normals (blue
    arrows) are compared to the projected normals (red arrows) at the
    corners of each \Bezier element for a coarse (c) and fine (d)
    T-mesh.\label{fig:ts-smooth-hull}}
\end{figure}
\clearpage

\begin{figure}[htb]
  \centering
  \begin{tabular}{ *{1}{c} }
    \includegraphics[width=4in]{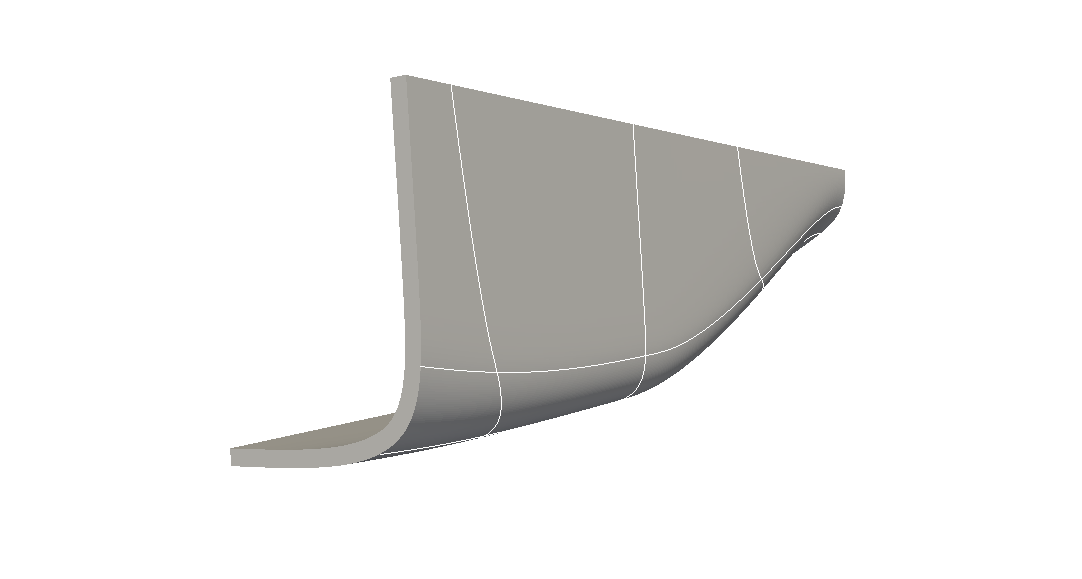}
  \end{tabular}
  \caption{A thickened bicubic T-spline containership hull. This
    T-spline surface in \cref{fig:ts-smooth-hull} has automatically
    been thickened using \Bezier projection.
\label{fig:ts-smooth-hull-thick}}
\end{figure}

\subsubsection{Projection between advected meshes}
Another application that benefits significantly from \Bezier
projection are isogeometric methods that rely on moving or advected meshes.
For example, in problems involving flow over moving boundaries or
large deformations it is often necessary to move the mesh with the
material and/or remeshing of the computational domain.
This is a fundamental component of arbitrary Lagrangian-Eulerian (ALE)
type methods~\cite{DGH82,JohTez94}.

As a simple example of the behavior of \Bezier projection for these
problems, we consider a benchmark problem of pure advection, namely the advection of a sine bump temperature feature by a rotating flow.
The problem statement is given in \cref{fig:rot-cone-setup}.
Rather than solve the advection problem directly, we employ a moving mesh approach.
The mesh and field defined over it are advected with the flow and then periodically projected back onto the the original spatial grid.
This procedure is illustrated in \cref{fig:spinning-mesh}.
The rotation operator $\vec{\rho}$ is used to rotate the mesh and field and then a \Bezier projection $\Pi_B$ is used to project the field from the rotated mesh onto a spatially aligned grid.
\Bezier projection is especially advantageous here because only the elements in the rotated mesh that overlap a given element in the spatial mesh are required to perform the integration on each element.
After the local integrations have been completed, the result is smoothed using \cref{eq:weight-def}.
The process is then repeated until the simulation is complete.
The spinning mesh was carried out on a 30 by 30 mesh on a square domain $[-1.5,1.5]\times[-1.5,1.5]$ using biquadratic B-splines.
The mesh was rotated by $\theta=\pi/10$ at each time step so that 20 steps were required to complete a full rotation.
A total of 8 rotations of the cone about the origin were computed using 160 steps.
The first and last steps of the simulation are shown in \cref{fig:spinning-mesh-start-finish}.
The 160 steps required to advect the sine bump through 8 rotations have resulted in an increase in over- and undershoot of approximately one percent.
A slice through the solution along the $x$ axis is shown in \cref{fig:spinning-mesh-slice}.
The repeated projections between non-aligned meshes have not produced any observable changes between the initial and final states beyond the slight increase in over- and undershoot.
\begin{figure}
  \centering
  \includegraphics[width=2in]{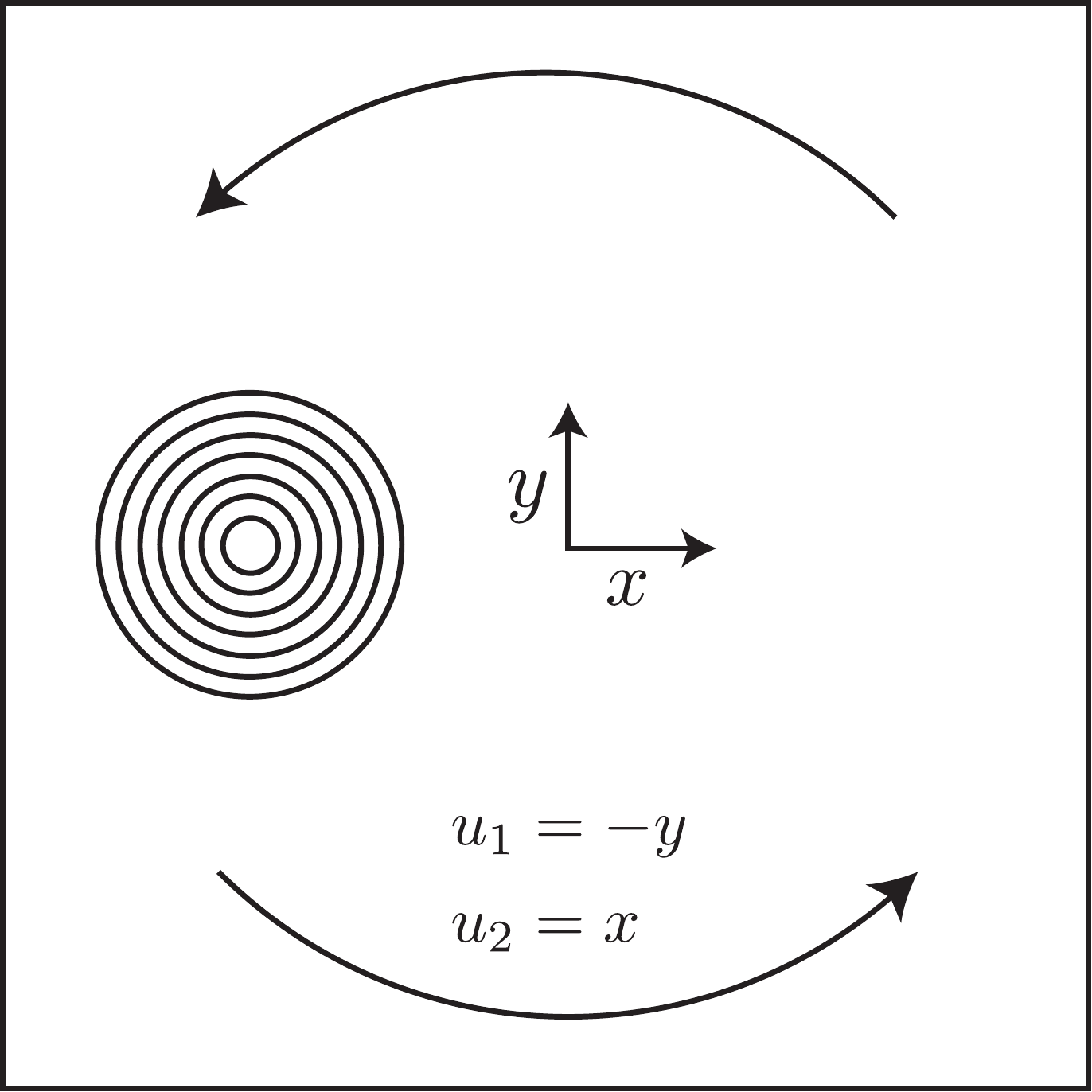}
  \caption{The rotating cone in a square problem statement.} 
  \label{fig:rot-cone-setup}
\end{figure}
\begin{figure}
  \centering
  \includegraphics[width=2in]{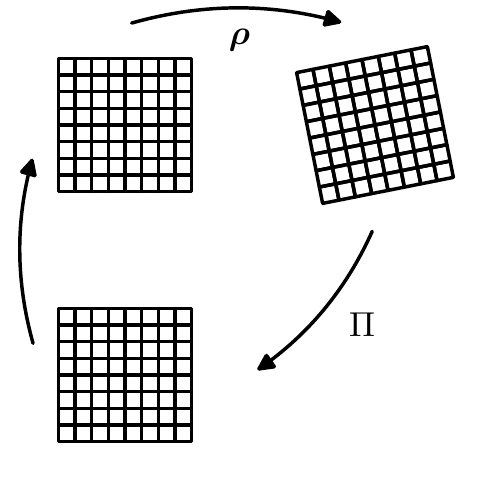}
  \caption{Spinning mesh procedure used for the rotating cone in a square problem.} 
  \label{fig:spinning-mesh}
\end{figure}
\begin{figure}[htb]
  \centering
  \begin{tabular}{*{1}{c}}
    \includegraphics[width=3in]{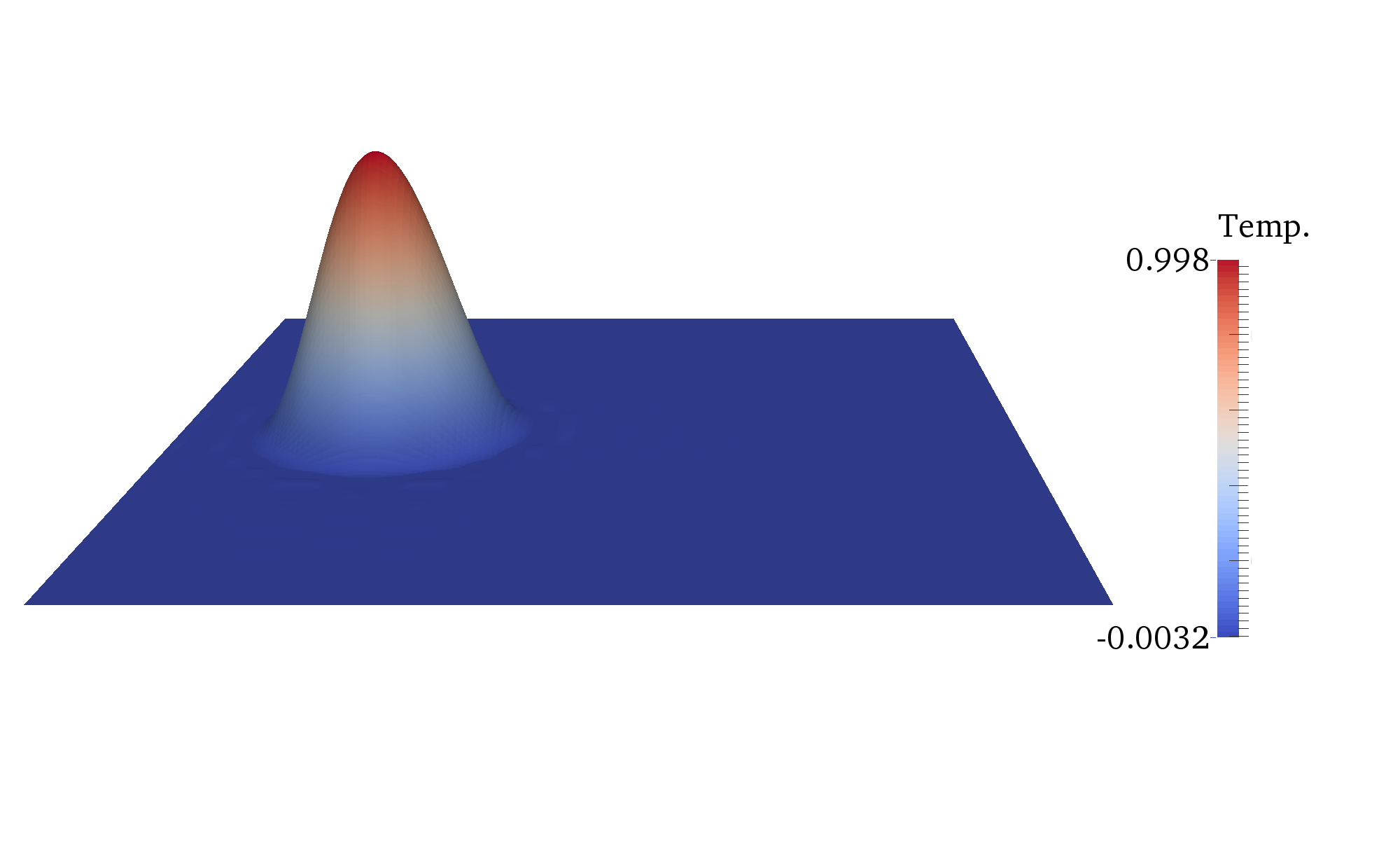}\\
    \includegraphics[width=3in]{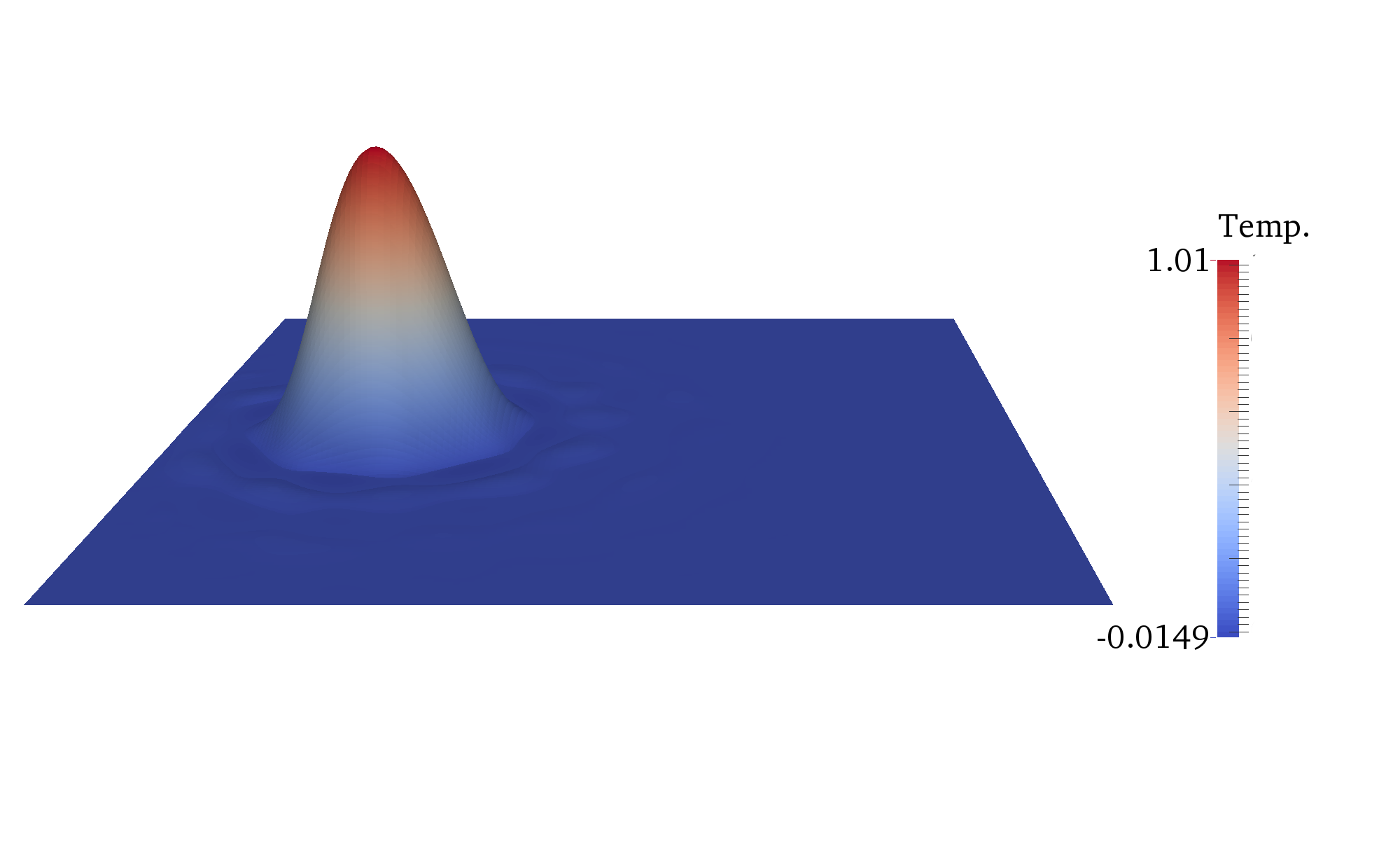}
  \end{tabular}
  \caption{\label{fig:spinning-mesh-start-finish}Initial (top) and
    final (bottom) steps in the solution of the pure advection of a sine hill by means of an advected mesh.}
\end{figure}
\begin{figure}[htb]
  \centering
    \includegraphics[width=4in]{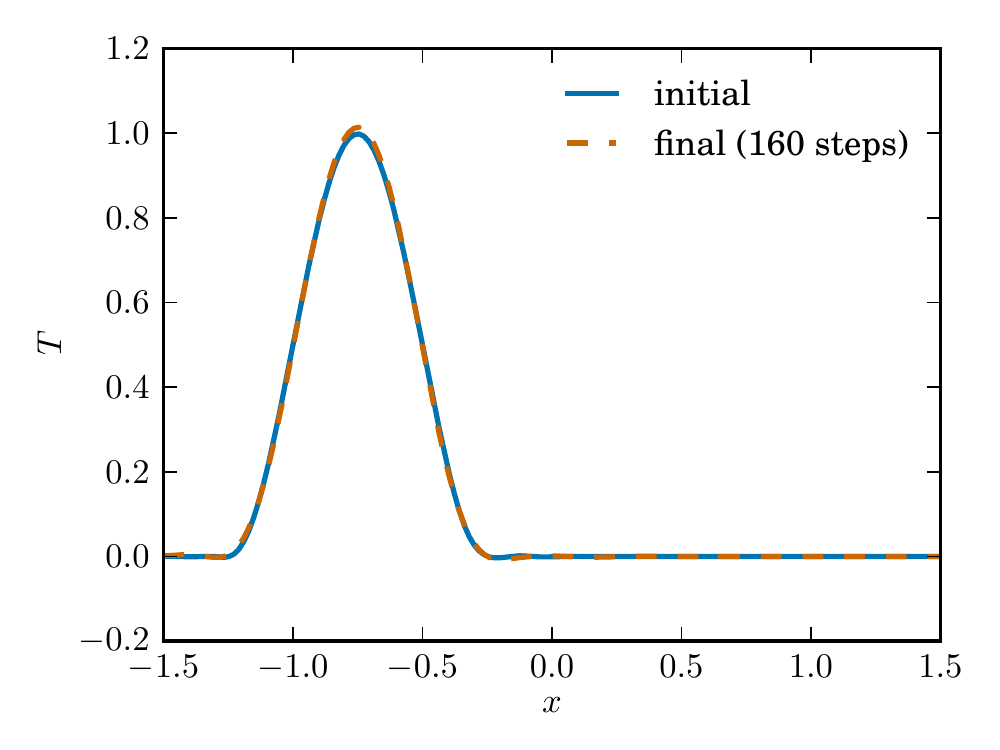}
  \caption{\label{fig:spinning-mesh-slice}Slice of the final result along the line $y=0$ for the rotating cone advection problem.}
\end{figure}

\section{Projection between spline spaces and operations on splines}\label{sec:proj-betw-spline}
We now consider \Bezier projection between splines spaces. We refer to
the spaces $\fspace{T}^a$ and $\fspace{T}^b$ as the source space and
target space, respectively.
While it is possible to use \Bezier projection as defined in
\cref{sec:localized-projection} to project between spline spaces, we
will show that it is possible to
define a \textit{quadrature-free} \Bezier projection approach for projection between spline spaces.
This approach can then be used to obtain algorithms for knot insertion
and removal, degree elevation, and reparameterization that can be
applied to any spline that can be represented using \Bezier
extraction (B-splines, NURBS, T-splines, LR-splines, etc.).

In all cases, the algorithms construct a new spline space and then project the original
spline representation onto the new space.  
All of these operations except reparameterization consist of either
projection from one space onto a superspace, which we refer to as
refinement, or projection from a space onto a subspace, which we refer
to as coarsening. It is also possible to define non-nested refinement operations that
increase the number of basis functions or degrees of freedom without
projecting onto a superspace although this possibility is not
considered in great depth here beyond reparameterization. 

One of the intriguing new features of the isogeometric paradigm is the
potential for $k$-refinement \cite{HuCoBa04,CoReBaHu05}, or in
other words, adaptively modifying the local smoothness of the
basis to improve the accuracy of the solution. 
\citet{HuCoBa04} and \citet{CoReBaHu05} originally used the term
$k$-refinement to refer to the process of generating a sequence of
smoother and smoother bases. 
We prefer to use $k$-refinement to denote the process of basis
roughening, or reducing the smoothness of a basis through knot
insertion, and $k$-coarsening to indicate the smoothing of the basis
functions through knot removal. The reason for this convention is twofold.
First, we prefer to use the word refinement to indicate the
transformation of the solution into a space that contains the
unrefined solution. This is not the case for a smoothed basis.
A function represented in terms of a spline basis that is $C^1$ at
each knot cannot be represented by a basis that has
higher continuity at the knots. 
Second, the word refinement suggests increased resolution or the
capability to represent finer detail or additional features. 
The process of basis smoothing reduces the dimension of the
space and so we refer to basis smoothing as $k$-coarsening. 
The space of spline functions defined by basis roughening contains the
original space and provides additional degrees of freedom and so we
feel that it is most natural to associate basis roughening with
$k$-refinement.

To simplify later developments we adopt the following naming convention:
\begin{enumerate}
\item (Cell) Subdivision is $h$-refinement.
\item (Cell) Merging is $h$-coarsening.
\item (Degree) Elevation is $p$-refinement.
\item (Degree) Reduction is $p$-coarsening.
\item (Basis) Roughening is $k$-refinement.
\item (Basis) Smoothing is $k$-coarsening.
\item Reparameterization is $r$-adaptivity.
\end{enumerate}

\Bezier projection between spline spaces reduces to a highly localized
projection between two different Bernstein bases. It is possible to express this in matrix form as
\begin{equation}\label{eq:generic-spline-spline-proj}
  \vec{P}^{e^{\prime},b} = (\mat{R}^{e^{\prime},b})^{\trans}(\mat{M}^{a,b})^\trans(\mat{C}^{e,a})^\trans\vec{P}^{e,a}
\end{equation}
where the element extraction operator on the source mesh
$\vec{C}^{e,a}$ converts the spline coefficients $\vec{P}^{e,a}$ to
\Bezier form, the matrix $\vec{M}^{a,b}$ converts the \Bezier
coefficients of the source Bernstein basis into coefficients of the
target Bernstein basis, and the element reconstruction operator on the
target mesh is used to convert the new \Bezier coefficients into
spline coefficients for the target basis. The weighted average or smoothing
algorithm given in \cref{eq:bez-proj-global-coeffs} can then be
applied, if necessary, to
obtain a set of global coefficients. 
If the target space is a superspace of the source space,
$\fspace{T}^a\subseteq \fspace{T}^b$, then the smoothing algorithm is
not required and the projection is exact. 
Otherwise, the projection is approximate. The form of $\vec{M}$ will
depend on the particular type of projection being performed (i.e, $h$,
$p$, or $k$).

Leveraging the tensor-product structure of the element extraction and reconstruction operators and the multivariate Bernstein basis,
\cref{eq:generic-spline-spline-proj} can also be written as
\begin{align}
\label{eq:generic-spline-spline-proj-kron}
\vec{P}^{e^{\prime},b} &=
\left\{\left[(\mat{R}_{{d_p}}^{e^{\prime},b})^{\trans}(\mat{M}^{a,b}_{{d_p}})^\trans(\mat{C}_{{d_p}}^{e,a})^\trans\right]\otimes\cdots\otimes\left[(\mat{R}_{1}^{e^{\prime},b})^{\trans}(\mat{M}_{1}^{a,b})^\trans(\mat{C}_{1}^{e,a})^\trans\right]\right\}\vec{P}^{e,a}\nonumber\\  
  &= \left[\bigodot_{i=1}^{d_p}
    (\mat{R}_{i}^{e^{\prime},b})^{\trans}(\mat{M}_{i}^{a,b})^\trans(\mat{C}_{i}^{e,a})^\trans\right]\vec{P}^{e,a} 
\end{align}
where $d_p$ denotes the number of parametric dimensions and the
reversed Kronecker product is denoted by 
\begin{equation}
  \bigodot_{i=1}^N \mat{C}_i = \mat{C}_N\otimes\cdots\otimes\mat{C}_1.
\end{equation}
This follows from standard properties of the Kronecker product.
Thus, most operations can be carried out by multiplying relatively
small matrices for each parametric dimension and then computing the
full Kronecker product of the result. 
This approach has the added benefit that there is no need to store a
large matrix. Instead, the operators for each dimension may be
computed and then used to compute any needed entries in the large
matrix.

\subsection{Projection operations between multiple elements}
As a preliminary tool we consider the projection between multiple
\Bezier elements. We will consider projections to and from a large element $\bar{e}$ and
$n$ subelements $\{e_i\}$.
\subsubsection{Projection from a single element onto multiple subelements}
We first consider the \Bezier projection from a large element $\bar{e}$ onto
$n$ subelements $\{e_i\}$. We require that
$\hat{\Omega}(e_i)\cap \hat{\Omega}(\bar{e}) =
\hat{\Omega}(e_i)$ for all $e_i$. 
\begin{algorithm}\label{alg:large-to-small-proj}
\Bezier projection from a large element $\bar{e}$ to a subelement $e_i$.
  \begin{enumerate}
  \item Convert the spline control values to \Bezier form using the element extraction operator for the element $\bar{e}$
\begin{equation}
\vec{Q}^{\bar{e}} = (\mat{C}^{\bar{e}})^{\trans}\vec{P}^{\bar{e}}.
\end{equation}
  \item Compute the transformation matrix $\mat{A}_i$ between the
    Bernstein basis over the large element and the Bernstein basis
    over the small element using \cref{eq:interval-op} by converting
    the upper and lower bounds of the small element $e_i$ to
    the local coordinates of the large element and using the result
    for $\tilde{a}$ (lower bound) and $\tilde{b}$ (upper bound) in
    \cref{eq:interval-op}. For multivariate elements, the process is
    carried out in each parametric dimension and the matrix $\mat{A}_i$
    is given by the Kronecker product
    \begin{equation}\label{eq:big-A-def}
      \mat{A}_i = \bigodot_{j=1}^{d_p}\mat{A}_{j}.
    \end{equation}
\item Apply the transformation matrix to the \Bezier control values on
  the large element to calculate \Bezier control values on the small
  element 
\begin{equation}
\vec{Q}^{e_i} = \mat{A}_i\vec{Q}^{\bar{e}}.
\end{equation}
\item Use the element reconstruction operator on the small
  element to convert the \Bezier control values to spline control
  values 
\begin{equation}
\vec{P}^{e_i} = (\mat{R}^{e_i})^{\trans}\vec{Q}^{e_i}.
\end{equation}
  \end{enumerate}
The matrix expression for these steps is
\begin{equation}\label{eq:large-to-small}
\vec{P}^{e_i} =
(\mat{R}^{e_i})^{\trans}\mat{A}_i(\mat{C}^{\bar{e}})^{\trans}\vec{P}^{\bar{e}}.
\end{equation}
Because $e_i$ is completely covered by $\bar{e}$ and the basis
functions over each element have the same polynomial degree, a
function expressed in terms of the Bernstein basis over the large
element can be exactly represented by the basis over the small
element. 
\end{algorithm}

\subsubsection{Projection from multiple subelements onto a single element}
We now consider \Bezier projection from a set of $n$ subelements $\{e_i\}$
onto a large element $\bar{e}$. 

\begin{remark} 
In many cases, like $r$-adaptivity, it may be
that $\hat{\Omega}(e_i)\cap \hat{\Omega}(\bar{e}) \neq
\hat{\Omega}(e_i)$ for some $e_i$. In that case,
apply~\cref{alg:large-to-small-proj} first to
trim the element so that $\hat{\Omega}(e_i)\cap \hat{\Omega}(\bar{e}) =
\hat{\Omega}(e_i)$.
\end{remark}

The operation to convert the spline form defined over elements $\{e_i\}$
to Bernstein-\Bezier form on element $\bar{e}$ is 
\begin{equation}
\label{eq:coarsen-spl-local-proj}
\mat{N}^{\trans}\vec{P}=\bar{\vec{B}}^{\trans}\vec{Q}^{\bar{e}}+\epsilon
\end{equation}
where $\vec{N}$ denotes the vector of spline basis functions defined
over the elements $\left\{ e_i \right\}$, $\vec{P}$ represents the
associated control values, $\bar{\vec{B}}$ is the vector of Bernstein
basis functions defined over $\bar{e}$, and $\vec{Q}^{\bar{e}}$ is the
vector of control values that we seek. 
Because the spaces are not nested, there is some error $\epsilon$ associated with the projection.
We can perform an $L^2$ projection of the spline function onto the
Bernstein basis of $\bar{e}$ to obtain the coefficients $\vec{Q}^{\bar{e}}$ by
multiplying both sides by the Bernstein basis $\bar{\vec{B}}$ and
integrating over the domain of $\bar{e}$
\begin{equation}
\label{eq:coarsen-spl-bern-l2-proj}
\int_{\hat{\Omega}(\bar{e})}\bar{\vec{B}}\vec{N}^{\trans}\vec{P}d\hat{\Omega}=\int_{\hat{\Omega}(\bar{e})}\bar{\vec{B}}\bar{\vec{B}}^{\trans}\vec{Q}^{\bar{e}}d\hat{\Omega}. 
\end{equation}
The error in this approximation is orthogonal to the basis $\bar{\vec{B}}$.
Note that this is a matrix equation; each integral is assumed to be carried out over each entry in the matrix.
Now convert the left-hand side from an integral over the domain of
$\bar{e}$ to a sum of integrals over the domains of the elements
$\{e_i\}$ and use the element extraction operators to convert
from spline coefficients to \Bezier coefficients over each element 
\begin{equation}
\label{eq:coarsen-integral-ei}
\sum_{i=1}^n\int_{\hat{\Omega}(e_i)}\bar{\vec{B}}\vec{B}_i^{\trans}\vec{Q}^{e_i}d\hat{\Omega}=\int_{\hat{\Omega}(\bar{e})}\bar{\vec{B}}\bar{\vec{B}}^{\trans}\vec{Q}^{\bar{e}}d\hat{\Omega}.
\end{equation}
We now exploit the relationship between the Bernstein bases given by
\cref{eq:inv-interval-op-basis} to write the Bernstein basis over
$\bar{e}$ in terms of the Bernstein basis over the elements $\{e_i\}$
\begin{equation}
\label{eq:Bbar-A-Bi}
\bar{\vec{B}}=\mat{A}_i^{-\trans}\vec{B}_i.
\end{equation}
This relationship permits \cref{eq:coarsen-integral-ei} to be written as
\begin{equation}
\label{eq:coarsen-integral-Ai}
\sum_{i=1}^n\int_{\hat{\Omega}(e_i)}\mat{A}_i^{-\trans}\vec{B}_i(\vec{B}_i)^{\trans}\vec{Q}^{e_i}d\hat{\Omega}=\int_{\hat{\Omega}(\bar{e})}\bar{\vec{B}}\bar{\vec{B}}^{\trans}\vec{Q}^{\bar{e}}d\hat{\Omega}.
\end{equation}
The integrals generate the Gramian or inner product matrices for the basis functions $\bar{\vec{B}}$ and $\vec{B}_i$.
The Gramian matrix for a Bernstein basis defined over the biunit box of dimension $d_p$ is given by
\begin{equation}
\label{eq:gram-def}
\mat{G}=\int_{[-1,1]^{d_p}}\vec{B}(\vec{\xi})\left[ \vec{B}(\vec{\xi}) \right]^{\trans}d\hat{\Omega}.
\end{equation}
The Gramian matrix for a Bernstein basis defined by a tensor product
over any other box of the same dimension can be related by a constant
scaling related to the volumes of the two boxes. As a result, the
Gramian matrix for the Bernstein bases $\bar{\vec{B}}$ and $\vec{B}_i$
are given by 
\begin{align}
\label{eq:B-bar-scale}
\bar{\mat{G}}&=\frac{\mathrm{vol}\,\hat{\Omega}(\bar{e})}{2^{d_p}}\mat{G}\\
\label{eq:Bi-scale}
\mat{G}_i&=\frac{\mathrm{vol}\,\hat{\Omega}(e_i)}{2^{d_p}}\mat{G}.
\end{align}
An expression for the Gramian matrix is given in \cref{eq:gramian-entries}.
With these relationships
\cref{eq:coarsen-integral-ei} can be rewritten without quadrature as
\begin{equation}
\label{eq:quad-free-coarsen}
\sum_{i=1}^n\vec{A}_i^{-\trans}\vec{G}_i\vec{Q}^{e_i}=\bar{\vec{G}}\vec{Q}^{\bar{e}}
\end{equation}
and $\vec{Q}^{\bar{e}}$ is given in terms of the Gramian $\mat{G}$ of the Bernstein basis over the biunit interval by
\begin{equation}
\label{eq:quad-free-coarsen-G}
\vec{Q}^{\bar{e}}=\sum_{i=1}^n\phi_i\vec{G}^{-1}\vec{A}_i^{-\trans}\vec{G}\vec{Q}^{e_i}
\end{equation}
where $\phi_i=\textrm{vol}\,\hat{\Omega}(e_i)/\textrm{vol}\,\hat{\Omega}(\bar{e})$.

\begin{algorithm}\label{alg:multi-to-one-proj}
Projection of control values from $n$ elements $\{e_i\}$ onto control values for a single element $\bar{e}$.
\begin{enumerate}
\item Use the element extraction operators for the elements $\{e_i\}$ to
  convert the control values on each element to \Bezier form 
  \begin{equation}
    \vec{Q}^{e_i} = (\mat{C}^{e_i})^{\trans}\vec{P}^{e_i}.
  \end{equation}
  \item The vector of \Bezier control values on $\bar{e}$ is given by \cref{eq:quad-free-coarsen-G}
\begin{equation}
\vec{Q}^{\bar{e}}=\sum_{i=1}^n\phi_i\vec{G}^{-1}\vec{A}_i^{-\trans}\vec{G}\vec{Q}^{e_i}.
\end{equation}
\item Use the element reconstruction operator for $\bar{e}$ to
  convert the \Bezier control values to spline control values 
\begin{equation}
\vec{P}^{\bar{e}} = (\mat{R}^{\bar{e}})^{\trans}\vec{Q}^{\bar{e}}. 
\end{equation}
  \item Use the weights defined in \cref{eq:weight-def} to compute the
    new global control values for the $A$th function on the target
    mesh from the new local control values
\begin{equation}
  \vec{P}_A = \sum_{\bar{e}\in \mathsf{E}_A} \omega_A^{\bar{e}} \vec{P}_A^{\bar{e}}.
\end{equation}
\end{enumerate}
\end{algorithm}
In many cases, the element extraction operators, the Bernstein
transformation matrices $\mat{A}_i$, and the Gramian matrix $\mat{G}$
are all formed from Kronecker products, and so it is possible to rewrite steps 1-3 of \cref{alg:multi-to-one-proj} as a more efficient Kronecker product matrix expression
\begin{align}
  \vec{P}^{\bar{e}} &= (\mat{R}^{\bar{e}})^{\trans}\left[\sum_{i=1}^n \phi_i\mat{G}^{-1}\mat{A}_i^{-\trans}\mat{G}(\mat{C}^{e_i})^\trans\vec{P}^{e_i}\right]\label{eq:small-to-large}\\
  &=\sum_{i=1}^n\phi_i\left[\bigodot_{j = 1}^{d_p}(\mat{R}_{s_j}^{\bar{e}})^{\trans}\mat{G}^{-1}_{s_j}\mat{A}_{i,s_j}^{-\trans}\mat{G}_{s_j}(\mat{C}^{e_i}_{s_j})^\trans\right]\vec{P}^{e_i}.
\end{align}
In two dimensions,
\begin{align}
  \left[\sum_{i=1}^n\phi_i\bigodot_{j = 1}^{2}(\mat{R}_{s_j}^{\bar{e}})^{\trans}\mat{G}_{s_j}^{-1}\mat{A}_{i,s_j}^{-1}\mat{G}_{s_j}(\mat{C}^{e_i}_{s_j})^\trans\right]\vec{P}^{e_i}&=\biggl\{\sum_{i=1}^n \phi_i\left[( \mat{R}_t^{\bar{e}} )^{\trans}\mat{G}_t^{-1} \mat{A}_{i,t}^{-1}\mat{G}_t  ( \mat{C}^{e_i}_t )^\trans\right]\nonumber\\
&\qquad\otimes\left[( \mat{R}_s^{\bar{e}} )^{\trans}\mat{G}_s^{-1}\mat{A}_{i,s}^{-1}\mat{G}_s( \mat{C}^{e_i}_s )^\trans\right]\biggr\}\vec{P}^{e_i}.
\end{align}
Step 3 in the above algorithm can also be replaced by the following approximate process:
Compute a weighted average of the projection of the \Bezier control values on the source elements onto control values on the target using the ratio of the parametric volume of $e_i$ to the parametric volume of $\bar{e}$
  % \begin{equation}\label{eq:overlap-frac-def}
  %   \phi_i = \frac{\mathrm{vol}[\hat{\Omega}(e_i)]}{\mathrm{vol}[\hat{\Omega}(\bar{e})]}
  % \end{equation}
 as the weight.
For each source element, there is a transformation operator $\mat{A}_i$ given by \cref{eq:interval-op} that can be used to relate coefficients of the basis on the source element to coefficients of the basis on the target element.
The matrix $\mat{A}_i$ is calculated by converting the upper and lower parametric bounds of each element $e_i$ to the local coordinates of the element $\bar{e}$ and using the lower bound as $\tilde{a}$ and the upper bound as $\tilde{b}$ in \cref{eq:interval-op} with ${a}=-1$ and ${b}=1$.
For multivariate elements, the process is carried out for each parametric dimension and the matrix $\mat{A}_i$ is given by
\begin{equation}
\mat{A}_i = \bigodot_{j=1}^{d_p}\mat{A}_{i,s_j}.
\end{equation}
The weighted average of the transformed control values is
\begin{equation}
\vec{Q}^{\bar{e}} = \sum_{i=1}^n\phi_i\mat{A}_i\vec{Q}^{e_i}.
\end{equation}

\subsection{$p$-adaptivity of B-splines and NURBS}\label{sec:degr-elev-reduct}
Degree elevation or $p$-refinement can be viewed as \Bezier projection from one spline space to
another spline space of higher polynomial degree in at least one
dimension. The projection is exact. 
For tensor-product splines $p$-refinement is accomplished by
incrementing the multiplicity of each knot and the polynomial degree in any dimension that is to
be elevated. Next, the \Bezier element extraction
operators for the source mesh, the element
reconstruction operators for the target mesh, and the
transformation matrix between the Bernstein bases on the source and
target meshes are computed. We call the Bernstein transformation in
the case of $p$-refinement the degree elevation matrix. It can be obtained
using standard approaches \cite{farouki2012}.

The degree elevation matrix $\mat{E}^{p,p+1}$ used to elevate a
one-dimensional Bernstein polynomial of degree $p$ by one is a
$p+1\times p+2$ matrix with entries given by: 
\begin{align}
E^{p,p+1}_{1,1} &= 1\nonumber\\
E^{p,p+1}_{i,i+1} &= \frac{i}{p+1},\, \text{for}\,i=1,2,\dots,p+1\nonumber\\
E^{p,p+1}_{i+1,i+1} &= 1-\frac{i}{p+1},\, \text{for}\,i=1,2,\dots,p+1\nonumber\\
E^{p,p+1}_{p+1,p+2} &= 1,\nonumber\\
E^{p,p+1}_{i,j} &= 0,\, \text{otherwise}.\label{eq:elev-mat-df}
\end{align}
This is the matrix that, given the vector of Bernstein polynomials of
degree $p$ and the vector of Bernstein polynomials of degree $p+1$,
satisfies 
\[
\vec{B}^p = \mat{E}^{p,p+1}\vec{B}^{p+1}.
\]
In other words, it provides a representation of the Bernstein basis of degree
$p$ in terms of the basis of degree $p+1$. Degree elevation of Bernstein
polynomials by more than one degree can be achieved by repeated
application of degree elevation matrices or the use of optimized algorithms~\cite{szafnicki2005}.
For multivariate tensor-product splines, the degree elevation matrix
is constructed from the Kronecker product of one-dimensional
degree elevation matrices. If the original degree in each dimension is given by the degree vector $\vec{p}=\{p_1,p_2,\dots,p_d\}$ and the final degree in each dimension is given by $\vec{p}^{\prime}=\{p^\prime_1,p^\prime_2,\dots,p^\prime_d\}$ subject to the constraint that $p^{\prime}_i \geq p_i$ then
\begin{equation}\label{eq:elev-mat-kp}
\mat{E}^{\vec{p},\vec{p}^{\prime}} = \mat{E}^{p_d,p_d^{\prime}}\otimes\cdots\otimes\mat{E}^{p_1,p_1^{\prime}}.
\end{equation}
The case where some dimensions are elevated and others are not can be
accommodated by requiring that $\mat{E}^{p_i,p_i}$ be equal to the
identity matrix of dimension $p_i+1$.

\begin{algorithm}\label{alg:spline-elev}
  Degree elevation of a B-spline or NURBS ($p$-refinement)
  \begin{enumerate}
  \item Create the target mesh by incrementing
    the degree and knot multiplicity in each parametric direction that is to be
    elevated. 
  \item Perform the \Bezier projection
    \begin{equation}\label{eq:bez-proj-elev}
    \vec{P}^{e,b} = (\mat{R}^{e,b})^{\trans}(\mat{E}^{\vec{p},\vec{p}^{\prime}})^{\trans}(\mat{C}^{e,a})^{\trans}\vec{P}^{e,a}     
    \end{equation}
  \end{enumerate}
Because the spline spaces are nested
the weighted averaging step is not required. 
\end{algorithm}

An example of a single degree elevation in one dimension is shown in \cref{fig:elevate-example}.
The original curve is shown in the center.
The original basis is defined by the knot vector
\begin{equation}
  \label{eq:deg-example-source}
\left\{ 0,0,0,0,\sfrac{1}{3},\sfrac{1}{3},\sfrac{2}{3},\sfrac{2}{3},1,1,1,1 \right\}.
\end{equation}
The elevated curve is shown on the right and the elevated basis is defined by the knot vector 
\begin{equation}
\label{eq:deg-elev-example-target}
\left\{ 0,0,0,0,0,\sfrac{1}{3},\sfrac{1}{3},\sfrac{1}{3},\sfrac{2}{3},\sfrac{2}{3},\sfrac{2}{3},1,1,1,1,1 \right\}.
\end{equation}
It can be seen that the elevated knot vector is obtained by increasing the multiplicity of each knot in the original knot vector by one.
The basis functions generated by these knot vectors are shown above the curves and the basis functions are colored to match the associated control points.
The extraction operator on the second element $e_2$ in the source mesh is
\begin{equation}
\label{eq:elev-example-source-ext}
\mat{C}^{e_2,a}=
\begin{bmatrix}\sfrac{1}{2} & 0 & 0 & 0\\\sfrac{1}{2} & 1 & 0 & 0\\0 & 0 & 1 & \sfrac{1}{2}\\0 & 0 & 0 & \sfrac{1}{2}\end{bmatrix}
\end{equation}
The Bernstein degree elevation matrix to elevate from degree 3 to degree 4 is given by \cref{eq:elev-mat-df} as
\begin{equation}
\label{eq:deg-elev-mat-3-4}
\mat{E}^{3,4}=\begin{bmatrix}1 & \sfrac{1}{4} & 0 & 0 & 0\\0 & \sfrac{3}{4} & \sfrac{1}{2} & 0 & 0\\0 & 0 & \sfrac{1}{2} & \sfrac{3}{4} & 0\\0 & 0 & 0 & \sfrac{1}{4} & 1\end{bmatrix}.
\end{equation}
The extraction operator on second element in the target mesh defined by \cref{eq:deg-elev-example-target} is 
\begin{equation}
\label{eq:deg-elev-example-target-ext-op}
\mat{C}^{e_2,b}=\begin{bmatrix}\sfrac{1}{2} & 0 & 0 & 0 & 0\\\sfrac{1}{2} & 1 & 0 & 0 & 0\\0 & 0 & 1 & 0 & 0\\0 & 0 & 0 & 1 & \sfrac{1}{2}\\0 & 0 & 0 & 0 & \sfrac{1}{2}\end{bmatrix}
\end{equation}
and so the associated reconstruction operator is
\begin{equation}
\label{eq:deg-elev-example-target-rec-op}
\mat{R}^{e_2,b}=\begin{bmatrix}2 & 0 & 0 & 0 & 0\\-1 & 1 & 0 & 0 & 0\\0 & 0 & 1 & 0 & 0\\0 & 0 & 0 & 1 & -1\\0 & 0 & 0 & 0 & 2\end{bmatrix}.
\end{equation}
These are the matrices required for \cref{eq:bez-proj-elev}.
The result of using these matrices and the appropriate matrices for the other elements in the mesh to elevate the original curve in the lower center of \cref{fig:elevate-example} is shown on the lower right of the same figure.
The original curve and control points are shown in gray on the right for reference.
It can be seen that the elevated curve exactly represents the original curve and that the new control points lie on the lines connecting the original control points.
\begin{figure}[htb]
  \centering
  \begin{tabular}{ *{3}{c} }
    degree reduced&original curve&degree elevated\\
    \includegraphics[width=2in]{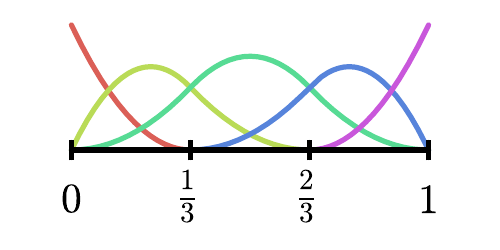}&\includegraphics[width=2in]{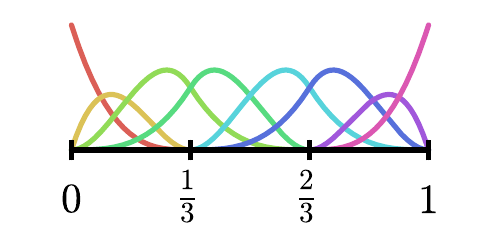}&\includegraphics[width=2in]{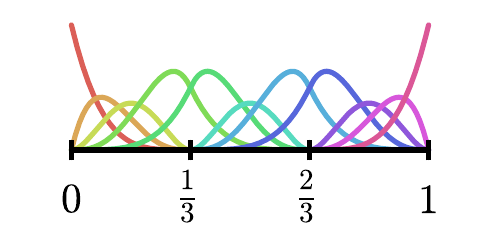}\\
    \includegraphics[width=2in]{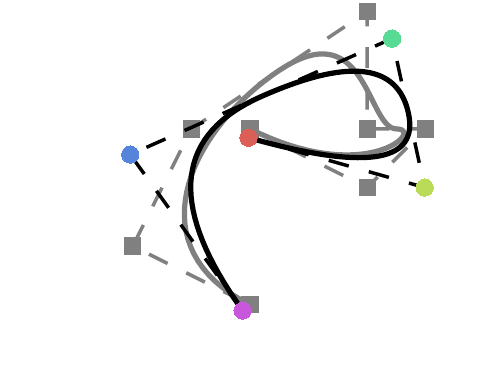}&\includegraphics[width=2in]{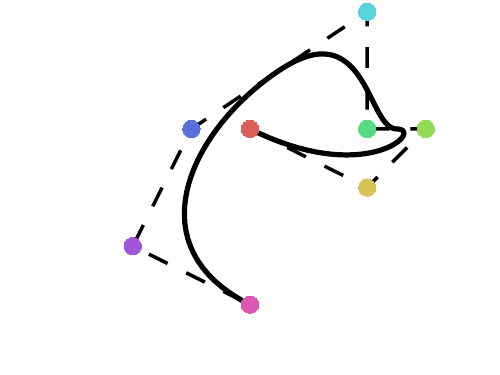}&\includegraphics[width=2in]{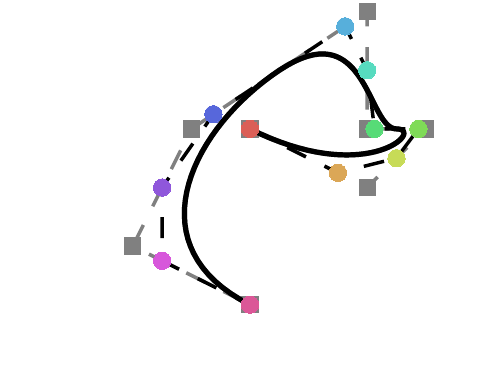}
  \end{tabular}
  \caption{\label{fig:elevate-example}B-spline elevation and reduction by \Bezier projection.}
\end{figure}

Degree reduction or $p$-coarsening can be viewed as \Bezier projection from one spline space to
another spline space of lower polynomial degree in at least one
dimension. The projection is approximate. For tensor-product splines $p$-coarsening is accomplished by
decrementing the multiplicity of each knot and the polynomial degree in any dimension that is to
be reduced. Next, the \Bezier element extraction
operators for the source mesh, the element
reconstruction operators for the target mesh, and the
transformation matrix between the Bernstein bases on the source and
target meshes are computed.

The transformation matrix that provides the best $L^2$ approximation
of the Bernstein basis of degree $p$ by the basis of degree $p-1$ is
given by the right pseudoinverse of the matrix $\mat{E}^{p-1,p}$ defined by \cref{eq:elev-mat-df} 
\begin{equation}\label{eq:deg-red-mat}
\mat{D}^{p,p-1} = (\mat{E}^{p-1,p})^\trans\left[ \mat{E}^{p-1,p}(\mat{E}^{p-1,p})^\trans \right]^{-1}.
\end{equation}
This is due to the fact that the best $L^2$ projection for polynomial
degree reduction is given by the best Euclidean approximation of the
the \Bezier coefficients \cite{lutterkort1999,peters2000}. 
In other words, for degree reduction of Bernstein polynomials, $L^2$
projection of the functions reduces to $\ell^2$ projection of the
function coefficients. The multivariate transformation matrix is given
by the Kronecker product of the one-dimensional matrices 
\begin{equation}
\label{eq:deg-red-mat-kp}
\mat{D}^{\vec{p},\vec{p}^{\prime}} = \mat{D}^{p_d,p_d^{\prime}}\otimes\cdots\otimes\mat{D}^{p_1,p_1^{\prime}}
\end{equation}
where $p_i^{\prime}\leq p_i$.
\begin{algorithm}\label{alg:spline-red}
  Degree reduction of a B-spline or NURBS ($p$-coarsening)
  \begin{enumerate}
  \item Create the target mesh by decrementing the
    degree and knot multiplicity in each parametric direction that is to be
    reduced. 
  \item Perform the \Bezier projection
    \begin{equation}
    \vec{P}^{e,b} = (\mat{R}^{e,b})^{\trans}(\mat{D}^{\vec{p},\vec{p}^{\prime}})^{\trans}(\mat{C}^{e,a})^{\trans}\vec{P}^{e,a}     
    \end{equation}
  \item Smooth the result
\begin{equation}
  \vec{P}^b_A = \sum_{e\in \mathsf{E}_A} \omega_A^e \vec{P}_A^{e,b}.
\end{equation}
  \end{enumerate}
\end{algorithm}

An example of this process is given in \cref{fig:elevate-example}.
The  degree-reduced basis is
defined by the knot vector \[\left\{
  0,0,0,\sfrac{1}{3},\sfrac{2}{3},1,1,1 \right\}.\]
The original basis is again defined by \cref{eq:deg-example-source} and so the exraction operator on the second element of the source mesh is given by \cref{eq:elev-example-source-ext}.
The extraction operator for the second element in the target mesh $e_2$ is
\begin{equation}
\label{eq:red-ext-op-example}
\mat{C}^{e_2,b}=
\begin{bmatrix}\sfrac{1}{2} & 0 & 0\\\sfrac{1}{2} & 1 & \sfrac{1}{2}\\0 & 0 & \sfrac{1}{2}\end{bmatrix}
\end{equation}
and so the associated reconstruction operator is
\begin{equation}
\label{eq:red-reconst-op-example}
\mat{R}^{e_2,b}=
\begin{bmatrix}
2 & 0 & 0\\
-1 & 1 & -1\\
0 & 0 & 2
\end{bmatrix}.
\end{equation}
The degree reduction operator from degree 3 to degree 2 is
\begin{equation}
\label{eq:red-mat-example}
\mat{D}^{3,2}=\frac{1}{20}\begin{bmatrix}
19 & - 5 & 1\\
3 & 15 & - 3\\
- 3 & 15 & 3\\
1 & - 5 & 19
\end{bmatrix}.
\end{equation}
The averaging weights for the target basis are the same as for the example illustrated in \cref{fig:weights}.
The result of applying \cref{alg:spline-red} to degree reduce the curve in the lower center of \cref{fig:elevate-example} is shown on the lower left of the same figure and the original curve is shown in gray for reference; it is
apparent that the degree reduced curve is only an approximation to the
original curve. This is in contrast to the elevated case in which the curve was
preserved exactly. 

\subsection{$p$-adaptivity of T-splines}
Degree elevation of a T-spline is achieved by increasing the degree
and multiplicity of each edge in each elevation direction 
by one and propagating T-junctions through any new repeated edges. 
To ensure nestedness of degree elevated T-splines requires
that the extended source and target T-meshes
are nested and analysis-suitable. This is a mild
generalization of~\cref{thm:nested}.

Consider the even-to-odd degree elevation example shown in
\cref{fig:t-spline-elevate-even-odd}. 
A representative source mesh is shown in \cref{fig:t-spline-elevate-even-odd}a.
The degree and multiplicity of each edge in each elevation direction is increased
by one and T-junctions are propagated through any new repeated edges. 
This process produces the target mesh shown in
\cref{fig:t-spline-elevate-even-odd}b. It can be easily verified that
the extended T-meshes are nested. This ensures that 
the quadratic T-spline basis function, anchored at the orange diamond
in \cref{fig:t-spline-elevate-even-odd}a,
can be exactly represented in the degree elevated T-mesh by the four
cubic functions anchored at the blue diamonds in \cref{fig:t-spline-elevate-even-odd}b.

Now consider the odd-to-even degree elevation example shown in
\cref{fig:t-spline-elevate-odd-even}.
A representative source mesh is shown in \cref{fig:t-spline-elevate-odd-even}a.
Again, the degree and multiplicity of each edge in each elevation direction is
increased by one and T-junctions are propagated through any new
repeated edges.
However, in this case the extended T-meshes are not nested due to the
repeated vertical edges at $\frac{3}{5}$. Inspection of the mesh shown in
\cref{fig:t-spline-elevate-odd-even}b reveals that the nine quadratic
functions, anchored at the blue diamonds in \cref{fig:t-spline-elevate-odd-even}b, needed to represent the
linear function anchored at the orange diamond in \cref{fig:t-spline-elevate-odd-even}a, are not reproduced by
the refined T-mesh in \cref{fig:t-spline-elevate-odd-even}b. In this
case, each T-junction must be extended across a single parametric
element and through the repeated edges on the opposite side as shown
in \cref{fig:t-spline-elevate-odd-even}c.

\begin{figure}[htb]
  \centering
  \begin{tabular}{ *{2}{c}}
    \includegraphics[width=2in]{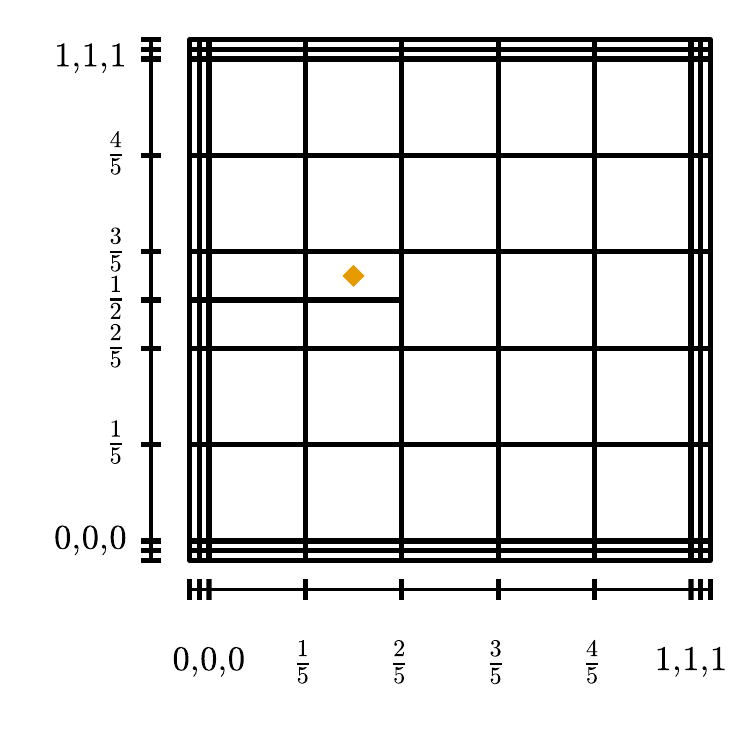}&\includegraphics[width=2in]{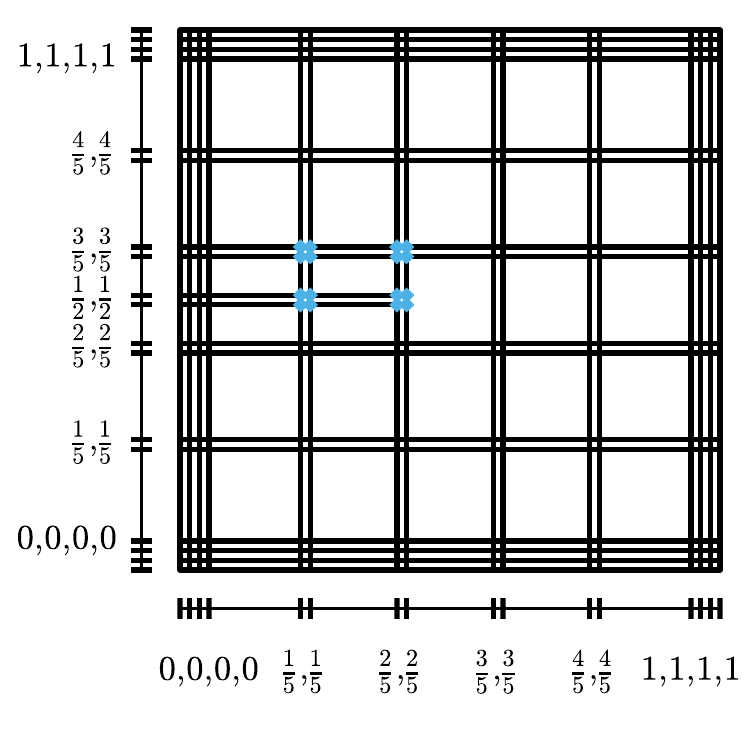}\\
   (a) original mesh ($p=2$)& (b) elevated mesh ($p=3$)
  \end{tabular}
  \caption{\label{fig:t-spline-elevate-even-odd}T-spline elevation even to odd.}
\end{figure}

\begin{figure}[htb]
  \centering
  \begin{tabular}{ *{2}{c}}
    \includegraphics[width=2in]{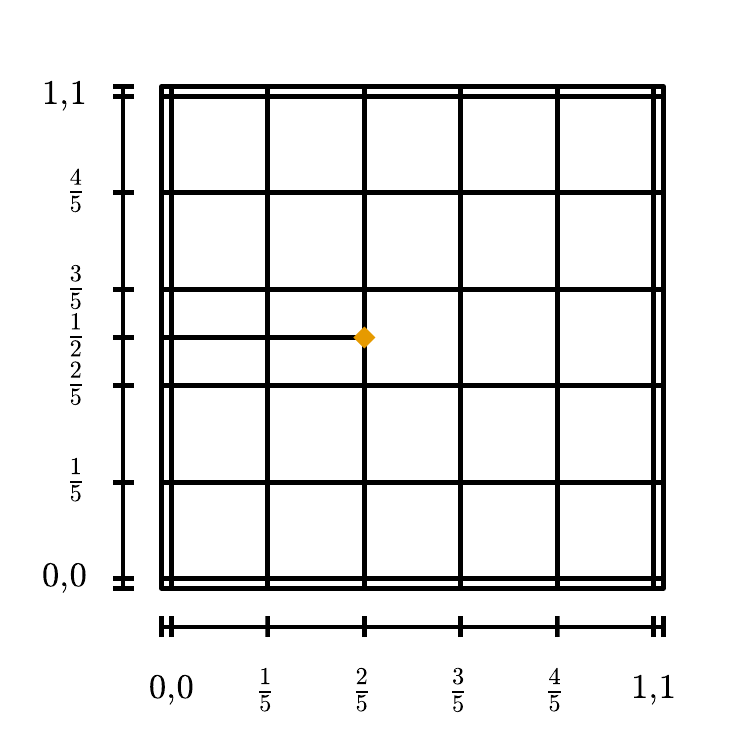}&\includegraphics[width=2in]{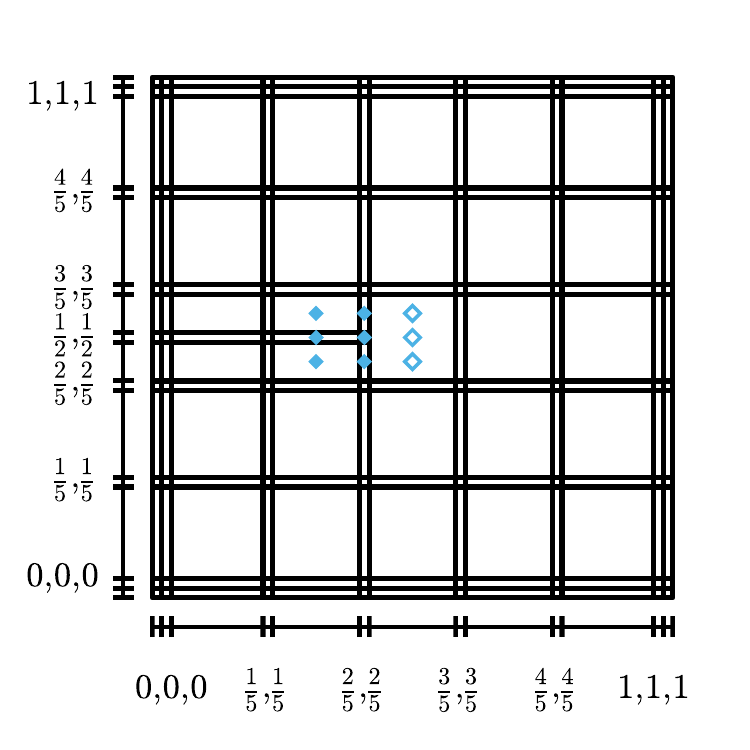}\\ 
    (a) original mesh ($p=1$)& (b) elevated mesh ($p=2$)\\
    \includegraphics[width=2in]{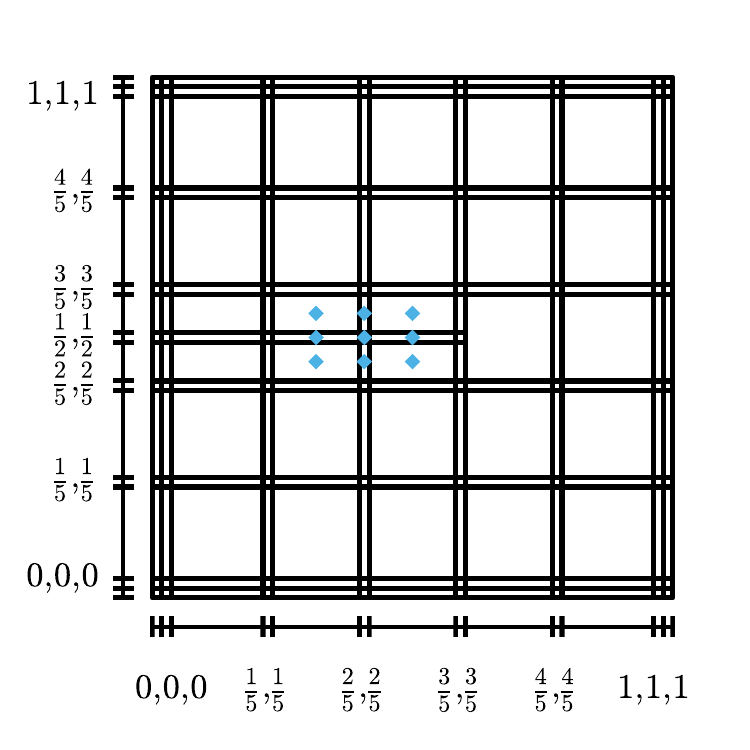}&\\
    (c) extended elevated mesh ($p=2$)&
  \end{tabular}
  \caption{\label{fig:t-spline-elevate-odd-even}T-spline elevation odd to even.}
\end{figure}
\clearpage

\begin{algorithm}\label{alg:t-spline-elev}
Degree elevation of an analysis-suitable T-spline ($p$-refinement)
\begin{enumerate}
\item Create the target mesh by incrementing the degree and
  multiplicity of all T-mesh edges in each parametric direction that is to be elevated.
\item Extend all T-junctions through any repeated edges.
\item Modify T-junctions so that
  $\mathsf{T}^a_{\mathrm{ext}}\subseteq \mathsf{T}^b_{\mathrm{ext}}$
  and $\mathsf{T}^b_{\mathrm{ext}}$ is analysis-suitable.
\item Continue from step 2 of \cref{alg:spline-elev}. 
\end{enumerate}
\end{algorithm}

An example of T-spline degree elevation is shown in \cref{fig:ts-project-p-refine}.
The original T-mesh for a $C^1$ T-spline of degree 3 in each direction is shown in the upper center of the figure.
Repeated knots are indicated by closely spaced edges.
A random surface is generated by assigning a uniformly distributed random value between $0$ and $\sfrac{1}{2}$ to the elevation values of the control points that generate a linearly parameterized surface.
This random surface is shown in the lower center.
The degree elevation algorithm for ASTS is applied to the original mesh to obtain the degree-elevated mesh shown in the upper right.
The result of applying the \Bezier projection algorithm to compute the elevated surface shown in the lower right.
The original surface is shown as a wireframe for comparison while the elevated surface is blue.
It can be seen that the elevated surface coincides exactly with the
original surface.
\begin{figure}[htb]
  \centering
  \begin{tabular}{ *{3}{c} }
    degree-reduced T-mesh&original T-mesh&degree-elevated T-mesh\\
    \includegraphics[width=2in]{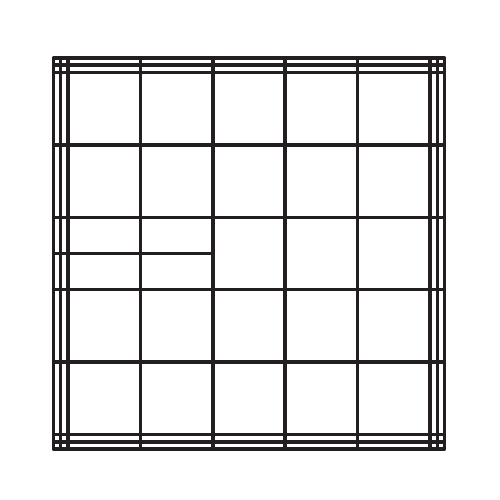}&\includegraphics[width=2in]{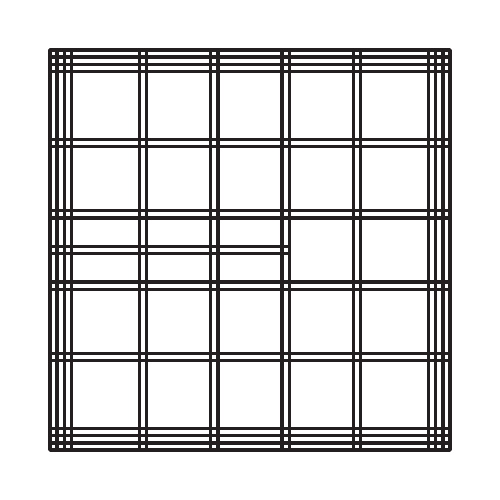}&\includegraphics[width=2in]{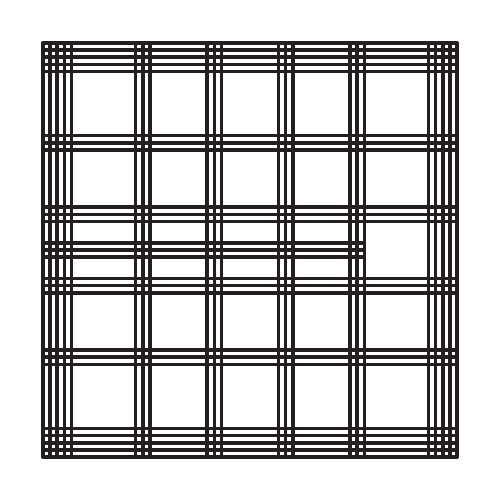}\\
    \includegraphics[width=2in]{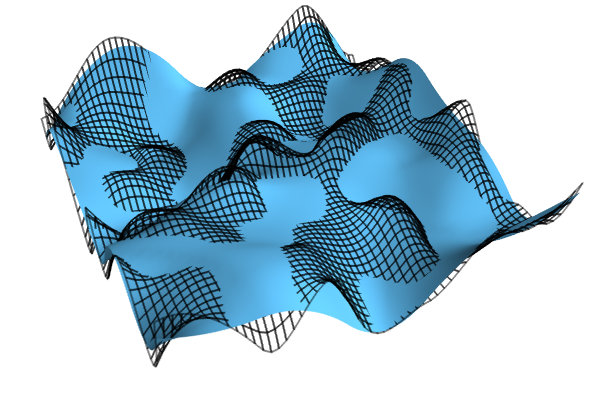}&\includegraphics[width=2in]{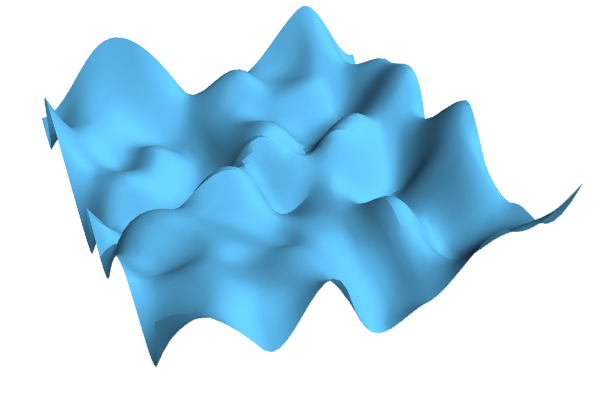}&\includegraphics[width=2in]{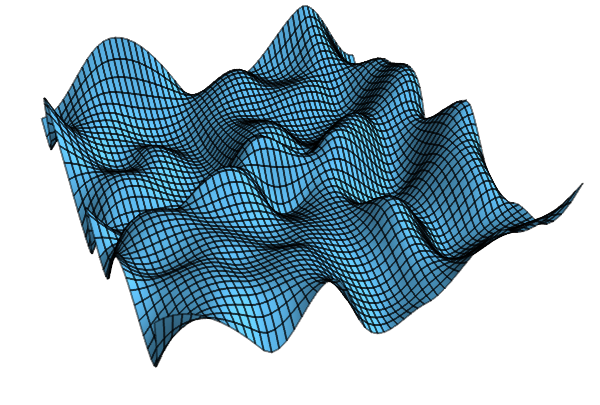}\\
    degree-reduced surface&original surface&degree-elevated surface
  \end{tabular}
  \caption{\label{fig:ts-project-p-refine}
    T-spline reduction and elevation by \Bezier projection.
}
\end{figure}

\begin{algorithm}\label{alg:t-spline-red}
Degree reduction of an analysis-suitable T-spline ($p$-coarsening)
\begin{enumerate}
\item Create the target mesh by decrementing the degree and
  multiplicity of all T-mesh edges in each parametric direction that is to be reduced.
\item Modify T-junctions so that
  $\mathsf{T}^a_{\mathrm{ext}}\supseteq \mathsf{T}^b_{\mathrm{ext}}$
  and $\mathsf{T}^b_{\mathrm{ext}}$ is analysis-suitable.
\item Continue from step 2 of \cref{alg:spline-red}.
\end{enumerate}
\end{algorithm}

An example of T-spline degree reduction is shown in \cref{fig:ts-project-p-refine}.
The degree reduction algorithm is applied to the original T-mesh in
the upper center to obtain the reduced T-mesh shown on the upper left.
The result of applying the \Bezier projection algorithm to compute the
reduced surface is shown in the lower left.
The original surface is shown as a wireframe for comparison while the reduced surface is blue.
It can be seen that while the reduced surface has an overall shape
similar to the original surface, as expected, the two surfaces do not
match.

\begin{remark}
Due to T-junction extension in degree elevation and T-junction
retraction in degree reduction the number of \Bezier elements defined
by the source and target T-meshes may be different. In that case,
\cref{alg:large-to-small-proj} may be required to compute the projection.
\end{remark}

\subsection{$k$-adaptivity of B-splines and NURBS}
\label{sec:knot-insert-remov}
Basis roughening or $k$-refinement is achieved by increasing the
multiplicity of some or all of the knots. Basis smoothing or
$k$-coarsening is achieved by decreasing the multiplicity of some or all
of the knots.
\begin{algorithm}\label{alg:knot-insertion-spline}
Roughening of a NURBS or B-spline ($k$-refinement)
\begin{enumerate}
\item Create the target mesh by incrementing the knot multiplicity of
  some of the knots in each parametric direction that is to be
  roughened.
\item Perform the \Bezier projection
\begin{equation}
\label{eq:rough-smooth-local}
\vec{P}^{e,b} = (\mat{R}^{e,b})^{\trans}(\mat{C}^{e,a})^{\trans}\vec{P}^{e,a}.
\end{equation}
\end{enumerate}
Because the spline spaces are nested
the weighted averaging step is not required. 
\end{algorithm}

\begin{algorithm}\label{alg:knot-removal-spline}
Smoothing of a NURBS or B-spline ($k$-coarsening)
\begin{enumerate}
\item Create the target mesh be decrementing the knot multiplicity of
  some of the knots in each parametric direction that is to be
  smoothed.
\item Perform the \Bezier projection using \cref{eq:rough-smooth-local}.
\item Smooth the result
\begin{equation}
  \vec{P}^b_A = \sum_{e\in \mathsf{E}_A} \omega_A^e \vec{P}_A^{e,b}.
\end{equation}
\end{enumerate} 
\end{algorithm}
\begin{figure}[htb]
  \centering
  \begin{tabular}{ *{3}{c}}
    smoothed&original curve&roughened\\
    \includegraphics[width=2in]{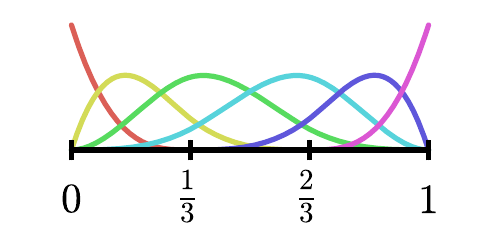}&\includegraphics[width=2in]{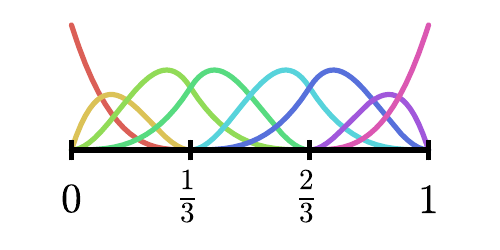}&\includegraphics[width=2in]{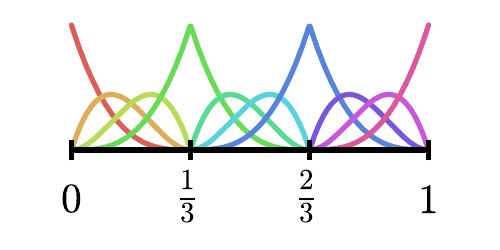}\\ 
    \includegraphics[width=2in]{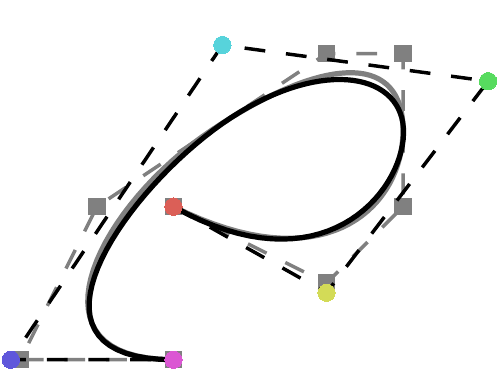}&\includegraphics[width=2in]{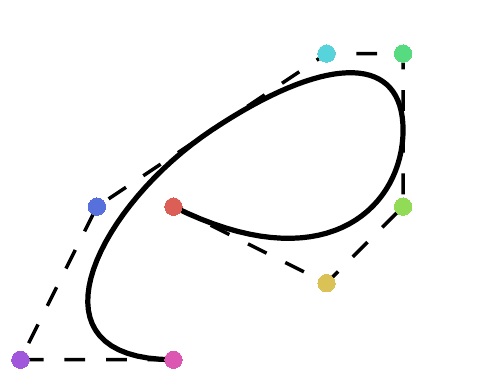}&\includegraphics[width=2in]{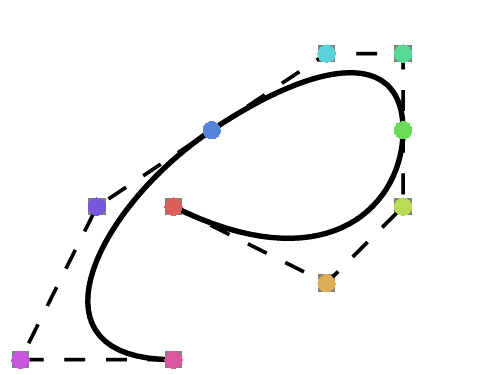} 
  \end{tabular}
  \caption{\label{fig:roughen-example}B-spline roughening and smoothing by \Bezier projection.}
\end{figure}
Basis roughening and coarsening for a one-dimensional B-spline curve
are illustrated in \cref{fig:roughen-example}. 
We begin with a cubic B-spline basis defined by the knot vector
\begin{equation}
\label{eq:k-ref-source-mesh}
\left\{ 0,0,0,0,\sfrac{1}{3},\sfrac{1}{3},\sfrac{2}{3},\sfrac{2}{3},1,1,1,1 \right\}.
\end{equation}
The basis is shown in the upper center of \cref{fig:roughen-example}.
The extraction operator on the second element $e_2$ defined by this knot vector is the same as for the degree elvation example given previously in \cref{eq:elev-example-source-ext}
A set of control points is chosen to define the curve shown in the lower center of the figure.
A refined or roughened basis is defined by the knot vector
\begin{equation}
\label{eq:k-ref-target-mesh}
\left\{ 0,0,0,0,\sfrac{1}{3},\sfrac{1}{3},\sfrac{1}{3},\sfrac{2}{3},\sfrac{2}{3},\sfrac{2}{3},1,1,1,1 \right\}.
\end{equation}
Because the knot multiplicity is equal to the polynomial degree, the extraction and reconstruction operators defined by this knot vector are identity matrices of size $p+1$.

A smoothed basis is defined by the knot vector
\begin{equation}
\label{eq:k-coarse-target-mesh}
\left\{ 0,0,0,0,\sfrac{1}{3},\sfrac{2}{3},1,1,1,1 \right\}.
\end{equation}
The element extraction and reconstruction operators for the second element in the smoothed mesh are
\begin{equation}
\label{eq:k-coarse-ext-op}
\mat{C}^{e_2,b}=\begin{bmatrix}\sfrac{1}{4} & 0 & 0 & 0\\\sfrac{7}{12} & \sfrac{2}{3} & \sfrac{1}{3} & \sfrac{1}{6}\\\sfrac{1}{6} & \sfrac{1}{3} & \sfrac{2}{3} & \sfrac{7}{12}\\0 & 0 & 0 & \sfrac{1}{4}\end{bmatrix}
\end{equation}
and 
\begin{equation}
\label{eq:k-coarse-rec-op}
\mat{R}^{e_2,b}\begin{bmatrix}4 & 0 & 0 & 0\\-4 & 2 & -1 & 1\\1 & -1 & 2 & -4\\0 & 0 & 0 & 4\end{bmatrix}.
\end{equation}

The roughened and smoothed bases are shown in the upper right and left
of \cref{fig:roughen-example} and the projection of the original curve
onto the new basis is shown below each. 
It is apparent that the refined basis exactly represents the original
curve (shown in gray for comparison) and that the number of control
points has been increased (indicating the increased dimension of the
roughened spline space). 
In contrast, the coarsened basis cannot fully represent the original
curve because of the increased continuity and the reduced size of the
basis. 

\subsection{$k$-adaptivity of T-splines}
Roughening or $k$-refinement of a T-spline is achieved by increasing the multiplicity of
each edge in each roughening direction by one and propagating
T-junctions through any new repeated edges.
To ensure nestedness of roughened or $k$-refined T-splines requires
that the extended source and target T-meshes are nested and
analysis-suitable as described in~\cref{thm:nested}.
\begin{algorithm}\label{alg:knot-insertion-t-spline}
Roughening of a T-spline ($k$-refinement)
\begin{enumerate}
\item Create the target mesh by incrementing the multiplicity of any
  edges that are to be roughened.
\item Extend all T-junctions through any repeated edges.
\item Modify T-junctions so that
  $\mathsf{T}_{\mathrm{ext}}^a\subseteq\mathsf{T}^b_{\mathrm{ext}}$
  and $\mathsf{T}^b$ is analysis-suitable.
\item Continue from step 2 of \cref{alg:knot-insertion-spline}
\end{enumerate} 
\end{algorithm}
Smoothing of a T-spline is the opposite of roughening.
\begin{algorithm}\label{alg:knot-removal-t-spline}
Smoothing of a T-spline ($k$-coarsening)
\begin{enumerate}
\item Create the target mesh by decrementing the multiplicity of any
  edges that are to be smoothed.
\item Modify T-junctions so that
  $\mathsf{T}_{\mathrm{ext}}^b\subseteq\mathsf{T}^a_{\mathrm{ext}}$
  and $\mathsf{T}^b$ is analysis-suitable.
\item Continue from step 2 of \cref{alg:knot-removal-spline}.
\end{enumerate} 
\end{algorithm}
Examples of T-spline roughening and smoothing are
shown in \cref{fig:ts-project-k-refine}. 
The same random surface used for T-spline elevation and reduction is used.
\begin{figure}[htb]
  \centering
  \begin{tabular}{ *{3}{c} }
    smoothed&original surface&roughened\\
    \includegraphics[width=2in]{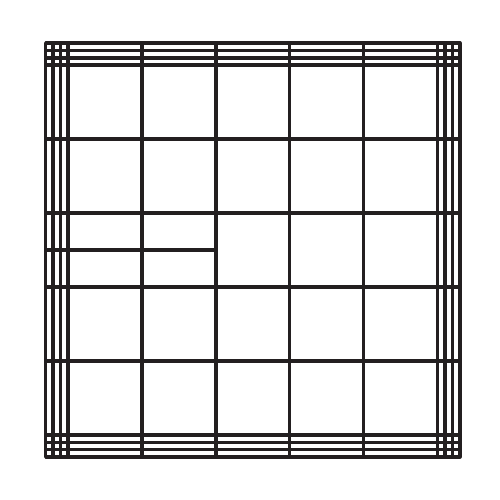}&\includegraphics[width=2in]{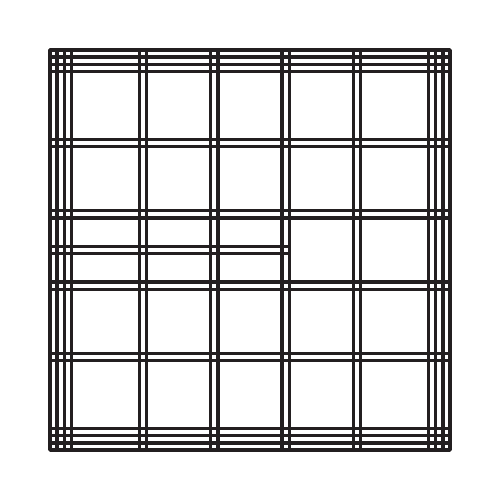}&\includegraphics[width=2in]{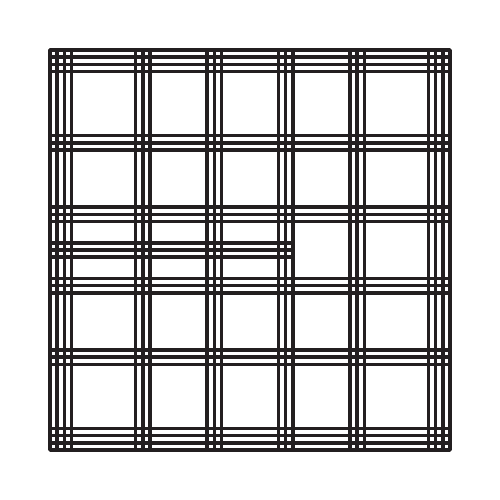}\\
    \includegraphics[width=2in]{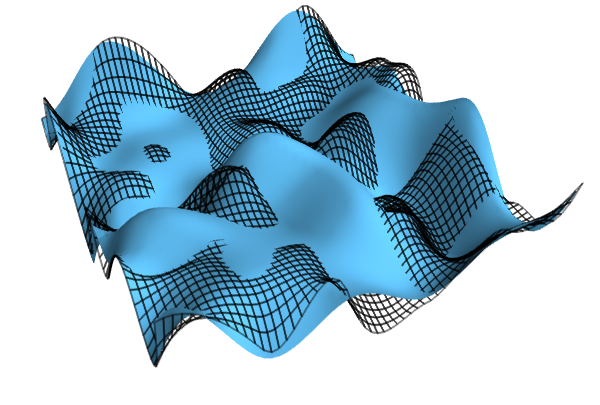}&\includegraphics[width=2in]{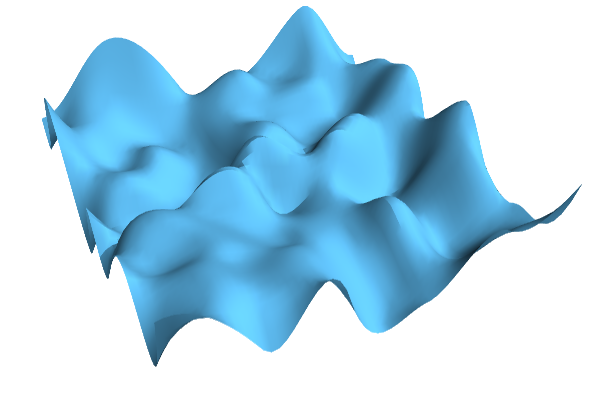}&\includegraphics[width=2in]{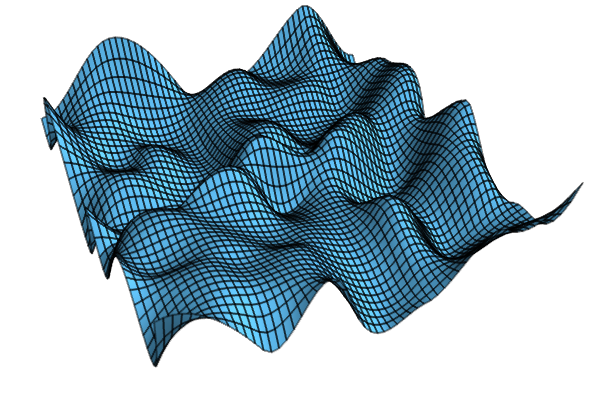}
  \end{tabular}
  \caption{\label{fig:ts-project-k-refine} T-spline roughening and smoothing by \Bezier projection.}
\end{figure}

\subsection{$h$-adaptivity of B-splines and NURBS}
\label{sec:knot-insert-remov-1}
Subdivision or $h$-refinement is achieved by subdividing a set of knot
intervals or elements with non-zero parametric length (area/volume).
\begin{algorithm}\label{alg:h-refine-spline}
  Subdivision for a B-spline or NURBS ($h$-refinement)
  \begin{enumerate}
  \item Create the target mesh by subdividing knot intevals with non-zero parametric length.
  \item Perform the \Bezier projection in \cref{alg:large-to-small-proj}.
  \end{enumerate}
  Because the spline spaces are nested
  the weighted averaging step is not required. 
\end{algorithm}
To illustrate \cref{alg:h-refine-spline} and in particular the application of \cref{alg:large-to-small-proj} we derive global $h/2$
refinement of a univariate B-spline using \Bezier projection. The knot vector is
\begin{equation}
\label{eq:h-refine-example-orig}
\{ 0,0,0,\sfrac{1}{4},\sfrac{1}{2},\sfrac{3}{4},1,1,1 \}.
\end{equation}
The quadratic basis is shown in the upper center of \cref{fig:refine-example}.
A new knot vector is constructed by dividing every nonzero knot
interval in half to obtain 
\begin{equation}
\label{eq:h-refine-example}
\{ 0,0,0,\sfrac{1}{8}, \sfrac{1}{4}, \sfrac{3}{8}, \sfrac{1}{2},
\sfrac{5}{8},\sfrac{3}{4},\sfrac{7}{8},1,1,1 \}. 
\end{equation}
The refined basis is shown on the upper right of \cref{fig:refine-example}.
The element extraction operators for both the source and target meshes can be
be computed using the methods given by \citet{Borden:2010nx} and
\citet{ScBoHu10}. All that remains in order to use
\cref{alg:large-to-small-proj} is the computation of $\mat{A}$. 
As stated in \cref{alg:large-to-small-proj}, the transformation
operator is found by converting the bounds of the small (target)
element to the local coordinates of the large (source) element and
then using \cref{eq:interval-op}. 
We follow the standard finite-element convention that the local
coordinate system of an element is the domain $[-1,1]$. 
For each element in the source mesh, we must project onto two new
elements, one bounded by $-1$ and $0$ in the local coordinate system
of the source element and the other bounded by $0$ and $1$. 
There is a distinct transformation operator associated with each of
the new elements which we denote $\mat{A}_l$ for the left subelement
(the subdomain $[-1,0]$) and $\mat{A}_r$ for the right subelement (the
subdomain $[0,1]$). 
The boundaries of each subelement are used for $\tilde{a}$ and
$\tilde{b}$ while $-1$ and $1$ are used for $a$ and $b$ in
\cref{eq:interval-op} to obtain 
\begin{equation}
\label{eq:lr-interval-ops}
\mat{A}_l =
\begin{bmatrix}
  1 & 0 & 0 \\
\sfrac{1}{2} & \sfrac{1}{2} & 0\\
\sfrac{1}{4} & \sfrac{1}{2} & \sfrac{1}{4}
\end{bmatrix},\qquad
\mat{A}_r =
\begin{bmatrix}
\sfrac{1}{4} & \sfrac{1}{2} & \sfrac{1}{4}\\
0 & \sfrac{1}{2} & \sfrac{1}{2} \\
  0 & 0 & 1
\end{bmatrix}.
\end{equation}
The extraction operator for the first element in the source mesh
(corresponding to the interval $[0,\sfrac{1}{4}]$) is 
\begin{align}
\mat{C}^{e_1}&=\begin{bmatrix}  1 &  0 &  0\\ 0 &  1 &  \sfrac{1}{2}\\ 0 &  0 &  \sfrac{1}{2}\end{bmatrix}.\label{eq:local-ext-ops-h-refine-source}\\
\end{align}
The element extraction operators for the first two elements of the
target mesh (corresponding to the intervals $[0,\sfrac{1}{8}]$ and
$[\frac{1}{8},\sfrac{1}{4}]$) 
are
\begin{align}
\mat{C}^{e^{\prime}_1}&=\begin{bmatrix}  1 &  0 &  0\\ 0 &  1 &
  \sfrac{1}{2}\\ 0 &  0 &
  \sfrac{1}{2}\end{bmatrix},\label{eq:local-ext-ops-h-refine-target-1}\\ 
\mat{C}^{e^{\prime}_2}&=\begin{bmatrix}  \sfrac{1}{2} &   0 &   0\\
  \sfrac{1}{2} &   1 &   \sfrac{1}{2}\\  0 &   0 &
  \sfrac{1}{2}\end{bmatrix}.\label{eq:local-ext-ops-h-refine-target-2} 
\end{align}
The reconstruction operators for these elements are
\begin{align}
\mat{R}^{e^{\prime}_1}&=\begin{bmatrix}1 & 0 & 0\\0 & 1 & -1\\0 & 0 & 2\end{bmatrix},\label{eq:rec-ops-h-refine-target-1}\\
\mat{R}^{e^{\prime}_2}&=\begin{bmatrix}2 & 0 & 0\\-1 & 1 & -1\\0 & 0 & 2\end{bmatrix}.\label{eq:rec-ops-h-refine-target-2}
\end{align}

The control points associated with the first interval in the target
mesh are given by 
\begin{equation}
\vec{P}^{e^{\prime}_1} =
(\mat{R}^{e^{\prime}_1})^{\trans}\mat{A}_l(\mat{C}^{e_1})^{\trans}\vec{P}^{e_1} 
\end{equation}
where $\vec{P}^{e_1} = \left\{ \vec{P}_1, \vec{P}_2, \vec{P}_3
\right\}^{\trans}$ and $\vec{P}^{e_1^{\prime}} = \left\{
  \vec{P}_1^{\prime}, \vec{P}_2^{\prime}, \vec{P}_3^{\prime}
\right\}^{\trans}$ 
and the control points for the second interval are given by
\begin{equation}
\vec{P}^{e^{\prime}_2} =
(\mat{R}^{e^{\prime}_2})^{\trans}\mat{A}_r(\mat{C}^{e_1})^{\trans}\vec{P}^{e_1} 
\end{equation}
where $\vec{P}^{e_2^{\prime}} = \left\{ \vec{P}_2^{\prime}, \vec{P}_3^{\prime}, \vec{P}_4^{\prime} \right\}^{\trans}$.
Because the source space is a subspace of the target space, the values
computed for $\vec{P}_2^{\prime}$ and $\vec{P}_3^{\prime}$ will be the
same for both of the target elements. 
The results of carrying the projection process out for the control
points that define the curve in the lower center of
\cref{fig:refine-example} are shown on the right of the same figure. 
\begin{figure}[htb]
  \centering
  \begin{tabular}{ *{3}{c} }
    coarsened&original curve&refined\\
    \includegraphics[width=2in]{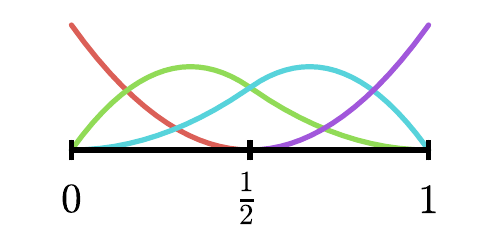}&\includegraphics[width=2in]{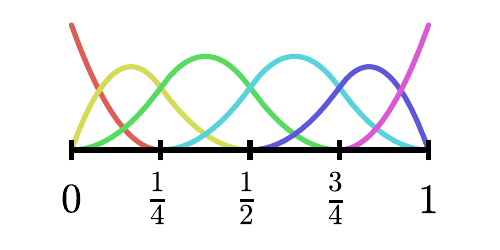}&\includegraphics[width=2in]{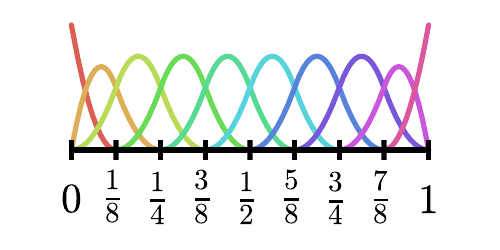}\\ 
    \includegraphics[width=2in]{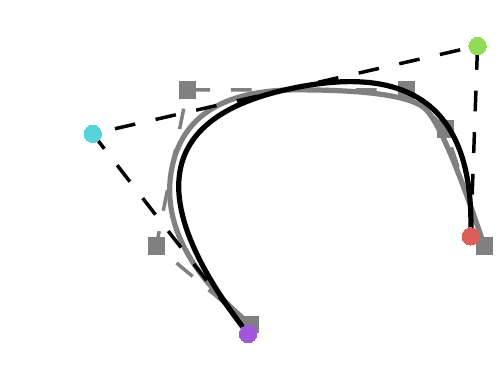}&\includegraphics[width=2in]{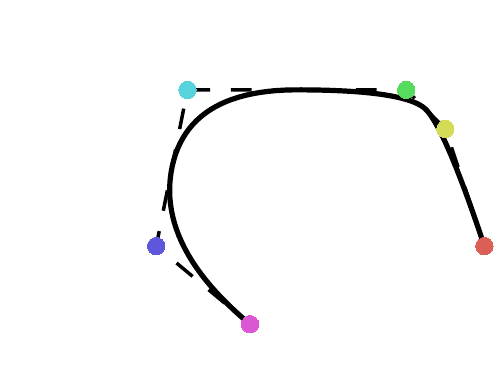}&\includegraphics[width=2in]{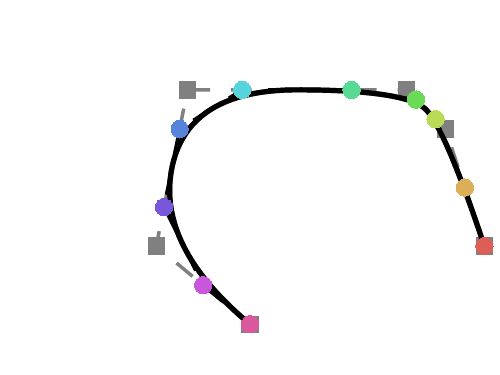}
  \end{tabular}
  \caption{\label{fig:refine-example}Refinement and coarsening of a spline curve by \Bezier projection.}
\end{figure}

Merging or $h$-coarsening is achieved by removing unique knots to
combine two or more knot intervals into a single interval.
\begin{algorithm}\label{alg:h-coarsen-spline}
  Merging for a B-spline or NURBS ($h$-coarsening)
  \begin{enumerate}
  \item Create the target mesh by removing knots.
  \item Perform the \Bezier projection in \cref{alg:multi-to-one-proj}.
  \item Smooth the result
    \begin{equation}
      \vec{P}_A^b = \sum_{e \in \mathsf{E}_A} \omega_A^{e} \vec{P}_A^{e,b}.
    \end{equation}
  \end{enumerate}
\end{algorithm}

To illustrate \cref{alg:h-coarsen-spline} we remove knots
$\sfrac{1}{4}$ and $\sfrac{3}{4}$ from $\{0,0,0,\sfrac{1}{4},
\sfrac{1}{2}, \sfrac{3}{4},1,1,1\}$ to obtain
\begin{equation}
\label{eq:h-coarsen-example-target}
\{ 0,0,0,\sfrac{1}{2},1,1,1 \}.
\end{equation}
The basis defined by this knot vector is shown in the upper left of
\cref{fig:refine-example}. 
The element extraction operators for the first two elements in the source mesh
(knot intervals $[0,\sfrac{1}{4}]$ and $[\sfrac{1}{4},\sfrac{1}{2}]$) are 
\begin{align}
\mat{C}^{e_1}&=\begin{bmatrix}  1 &  0 &  0\\ 0 &  1 &  \sfrac{1}{2}\\
  0 &  0 &
  \sfrac{1}{2}\end{bmatrix},\label{eq:local-ext-ops-h-coarsen-source-1}\\ 
\mat{C}^{e_2}&=\begin{bmatrix}  \sfrac{1}{2} &   0 &   0\\
  \sfrac{1}{2} &   1 &   \sfrac{1}{2}\\  0 &   0 &
  \sfrac{1}{2}\end{bmatrix}\label{eq:local-ext-ops-h-coarsen-source-2} 
\end{align}
and the element extraction and reconstruction operators for the first element of the target mesh
(knot interval $[0,\sfrac{1}{2}]$) that the two source elements are
projected onto are
\begin{align}
\mat{C}^{e_1^{\prime}}&=\begin{bmatrix}  1 &  0 &  0\\ 0 &  1 &
  \sfrac{1}{2}\\ 0 &  0 &  \sfrac{1}{2}\end{bmatrix},\label{eq:local-ext-ops-h-coarsen-target}\\
\mat{R}^{e_1^{\prime}} &=\begin{bmatrix}1 & 0 & 0\\0 & 1 & -1\\0 & 0 & 2\end{bmatrix}.\label{eq:local-rec-ops-h-coarsen-target}
\end{align}
The matrices used to transform the \Bezier coefficients of the source
elements onto the target element are computed by converting the lower
and upper bounds of the target element into the local coordinates of
the source elements. 
For this case, the bounds of the target element map to $-1$ and $3$ on
the first source element $e_1$ and $-3$ and $1$ on the second source
element. 
These values for the upper and lower bounds are used as values for $a$
and $b$ in \cref{eq:inv-interval-op} while $\tilde{a}=-1$ and
$\tilde{b} = 1$ to obtain 
\begin{align}
\mat{A}_1^{-1} &=
\begin{bmatrix}
  1 & 0 & 0 \\
  -1 & 2 & 0 \\
  1 & - 4 & 4
\end{bmatrix},\\
\mat{A}_2^{-1} &=
\begin{bmatrix}
  1 & - 4 & 4\\
  0 & 2 & -1 \\
  0 & 0 & 1
\end{bmatrix}.
\end{align}
Note that $\mat{A}_1^{-1}$ is in fact the inverse of the matrix
$\mat{A}_l$ defined in the previous example of $h$-refinement and
$\mat{A}_2^{-1}$ is the inverse of $\mat{A}_r$. 
Each of the source elements covers half of the target element and so
the weight for the relative contribution of each element to the final
control points is $\phi_1=\phi_2=\frac{1}{2}$. 
Finally, the Gramian matrix for the Bernstein basis of degree 2 is
\begin{equation}
\label{eq:gramian-bern-quad}
\mat{G}=\begin{bmatrix}
  \sfrac{2}{5} & \sfrac{1}{5} & \sfrac{1}{15}\\
\sfrac{1}{5} & \sfrac{4}{15} & \sfrac{1}{5}\\
\sfrac{1}{15} & \sfrac{1}{5} & \sfrac{2}{5}
\end{bmatrix}.
\end{equation}
The expression for the points points on $e_1^{\prime}$ given by \cref{alg:multi-to-one-proj} for this problem is
\begin{equation}
\vec{P}^{e_1^{\prime}} = (\mat{R}^{e^{\prime}_1})^{\trans}\left[
  \frac{1}{2}
  \mat{G}^{-1}\mat{A}_1^{-\trans}\mat{G}(\mat{C}^{e_1})^{\trans}\vec{P}^{e_1}
  + \frac{1}{2}\mat{G}^{-1}
  \mat{A}_2^{-\trans}\mat{G}(\mat{C}^{e_2})^{\trans}\vec{P}^{e_2}
\right] 
\end{equation}
where $\vec{P}^{e_1} = \left\{ \vec{P}_1, \vec{P}_2, \vec{P}_3
\right\}^{\trans}$ and $\vec{P}^{e_2} = \left\{ \vec{P}_2, \vec{P}_3,
  \vec{P}_4 \right\}^{\trans}$. 
The result obtained by completing this process for the remaining
elements and computing the weighted average of the resulting control
points on each element in the target mesh is shown in the lower left
of \cref{fig:refine-example}. 
It can be seen that the coarsened basis cannot fully capture the
original curve (shown in gray for comparison).

\subsection{$h$-adaptivity of T-splines}
Subdivision or $h$-refinement and merging or $h$-coarsening of a
T-spline is achieved by subdividing or merging \Bezier elements in the
T-mesh. 
\begin{algorithm}\label{alg:h-refine-t-spline}
  Subdivision for a T-spline ($h$-refinement)
  \begin{enumerate}
  \item Create the target mesh by adding edges to the T-mesh.
  \item Extend T-junctions so that
    $\mathsf{T}_{\mathrm{ext}}^a\subseteq\mathsf{T}^b_{\mathrm{ext}}$
    and $\mathsf{T}^b$ is analysis-suitable.
  \item Perform the \Bezier projection in \cref{alg:large-to-small-proj}.
  \end{enumerate}
  Because the spline spaces are nested
  the weighted averaging step is not required. 
\end{algorithm}

\begin{algorithm}\label{alg:h-coarsen-t-spline}
  Merging for a T-spline ($h$-coarsening)
  \begin{enumerate}
  \item Create the target mesh by removing edges from the T-mesh.
  \item Modify T-junctions so that
  $\mathsf{T}_{\mathrm{ext}}^b\subseteq\mathsf{T}^a_{\mathrm{ext}}$
  and $\mathsf{T}^b$ is analysis-suitable.
  \item Continue from step 2 of \cref{alg:h-coarsen-spline}.
  \end{enumerate}
\end{algorithm}

We present simple examples of T-spline subdivision and merging
in \cref{fig:ts-project-h-refine}. The original T-mesh is shown in the upper center of the figure.
A randomly generated surface is shown in the lower center of \cref{fig:ts-project-h-refine}.
The refined T-mesh is produced by adding edges in the
upper left and lower right corners of the original T-mesh to obtain the T-mesh shown in the
upper right of the figure. New control point positions are computed
by \Bezier projection.
The original and refined surfaces are compared in the lower right of
the figure using the same convention to distinguish between the two as
in the previous two T-spline examples
(\cref{fig:ts-project-p-refine,fig:ts-project-k-refine}). 
It can be seen that the refined surface coincides exactly with the original surface.
The coarsened mesh is shown in the upper left of the figure.
The coarsened mesh is obtained from the original T-mesh by removing
edges from 3 cells in the upper and lower left corners of the
mesh. The new control point positions are computed by \Bezier projection.
The original and coarsened surfaces are compared in the lower left of \cref{fig:ts-project-h-refine}.
It can be seen that the coarsened surface only approximates the original surface.
\begin{figure}[htb]
  \centering
  \begin{tabular}{ *{3}{c} }
    coarsened T-mesh&original T-mesh&refined T-mesh\\
    \includegraphics[width=2in]{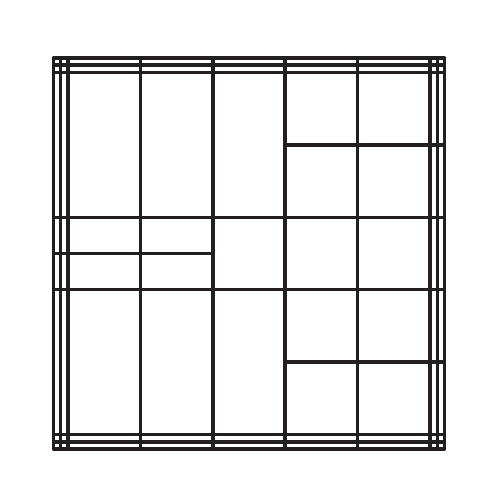}&\includegraphics[width=2in]{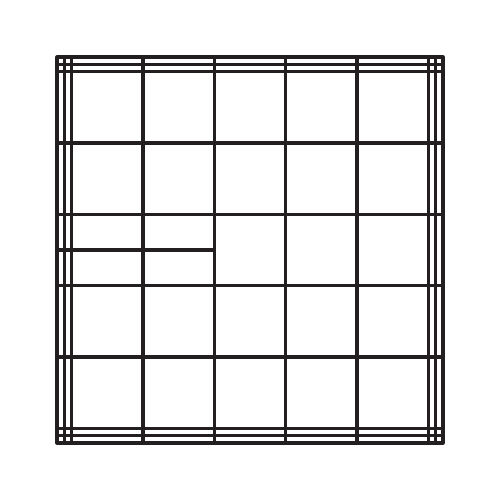}&\includegraphics[width=2in]{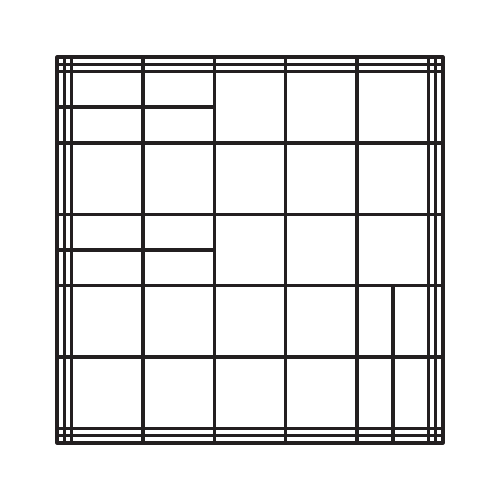}\\ 
    \includegraphics[width=2in]{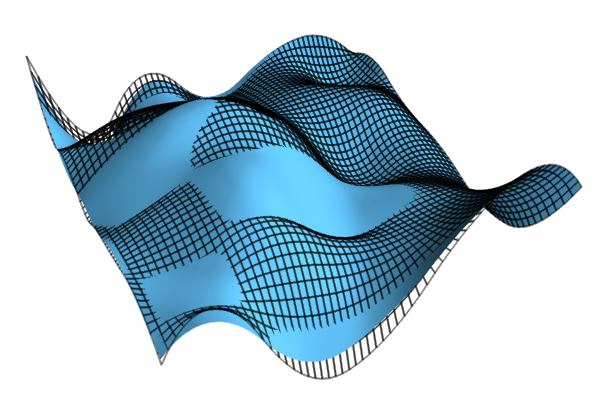}&\includegraphics[width=2in]{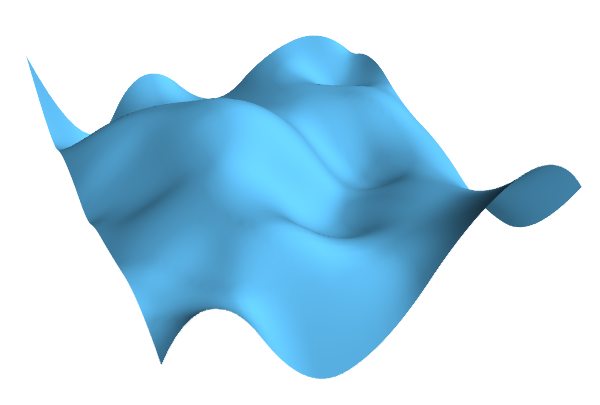}&\includegraphics[width=2in]{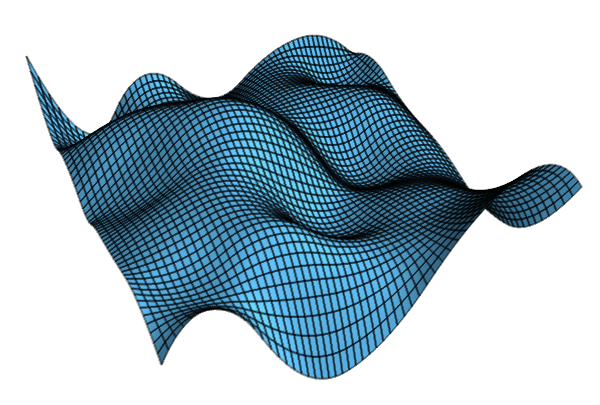}\\
    coarsened surface&original surface&refined surface
  \end{tabular}
  \caption{\label{fig:ts-project-h-refine} T-spline subdivision and
    merging by \Bezier projection.}
\end{figure}
\subsection{$r$-adaptivity of B-splines, NURBS, and T-splines}\label{sec:repar-r-refin}
Reparameterization or $r$-adaptivity in this paper refers to the
process of moving knots or T-mesh edges. 
This can be used to adjust the relative size of elements and hence the
resolution of the mesh while leaving the number of degrees of freedom
unchanged. Reparametrization does not generally
produce nested spaces. Reparameterization requires combined use of
\cref{alg:large-to-small-proj,alg:multi-to-one-proj}. 

\begin{algorithm}\label{alg:reparam}
Reparameterization of B-splines, NURBS, and T-splines ($r$-adaptivity)
  \begin{enumerate}
  \item Create the target mesh by repositioning knots or edges.
  \item Compute a target-to-source element map by enumerating all
    source elements that must be projected onto each target element. 
  \item For each element in the target mesh if the number of elements
    that are to be projected onto the element is greater than one, use
    \cref{alg:multi-to-one-proj}, otherwise use
    \cref{alg:large-to-small-proj}. 
  \item Smooth the result
\begin{equation}
  \vec{P}_A^{b} = \sum_{e \in \mathsf{E}_A} \omega_A^{e} \vec{P}_A^{e,b}.
\end{equation}
  \end{enumerate}
\end{algorithm}

An example of reparameterization is given in
\cref{fig:reparam-example} for a B-spline curve.
The knot vector for the source mesh is 
\begin{equation}
\label{eq:reparam-example-orig}
\{ 0,0,0,\sfrac{1}{2},1,1,1 \}
\end{equation}
and the spline basis defined by this knot vector is shown in the upper center of \cref{fig:reparam-example}.
The reparameterized knot vector is chosen as
\begin{equation}
\label{eq:reparam-example-refine}
\{ 0,0,0,\sfrac{7}{10},1,1,1 \}
\end{equation}
and the associated basis is shown in the upper right of \cref{fig:reparam-example}.
The target-to-source element map is
\begin{equation}
\label{eq:target-to-source-map}
\begin{bmatrix}
  [1,2]\\
[2]
\end{bmatrix}.
\end{equation}
This means that elements 1 and 2 in the source mesh must be projected
onto element 1 in the target mesh and element 2 in the source mesh
must be projected onto element 2 in the target mesh. 
The element extraction operators defined by the source mesh are
\begin{align}
\mat{C}^{e_1}&= \begin{bmatrix}1 & 0 & 0\\0 & 1 & \sfrac{1}{2}\\0 & 0 & \sfrac{1}{2}\end{bmatrix},\\
\mat{C}^{e_2}&= \begin{bmatrix}\sfrac{1}{2} & 0 & 0\\\sfrac{1}{2} & 1 & 0\\0 & 0 & 1\end{bmatrix}.
\end{align}
The element extraction operators defined by the target mesh are
\begin{align}
\mat{C}^{e_1^{\prime}}&= \begin{bmatrix}1 & 0 & 0\\0 & 1 & \sfrac{3}{10}\\0 & 0 & \sfrac{7}{10}\end{bmatrix},\\
\mat{C}^{e_2^{\prime}}&= \begin{bmatrix}\sfrac{3}{10} & 0 & 0\\\sfrac{7}{10} & 1 & 0\\0 & 0 & 1\end{bmatrix}.
\end{align}
The associated reconstruction operators are
\begin{align}
\mat{R}^{e_1^{\prime}}&=\begin{bmatrix}1 & 0 & 0\\0 & 1 & - \frac{3}{7}\\0 & 0 & \frac{10}{7}\end{bmatrix},\\
\mat{R}^{e_2^{\prime}}&=\begin{bmatrix}\frac{10}{3} & 0 & 0\\- \frac{7}{3} & 1 & 0\\0 & 0 & 1\end{bmatrix}.
\end{align}

\begin{figure}[htb]
  \centering
  \begin{tabular}{ *{3}{c}}
    left reparameterization&original curve&right reparameterization\\
    \includegraphics[width=2in]{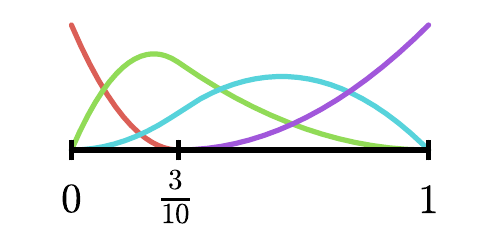}&\includegraphics[width=2in]{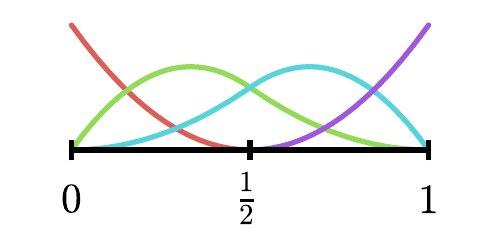}&\includegraphics[width=2in]{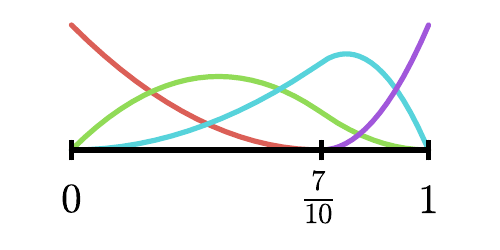}\\
    \includegraphics[width=2in]{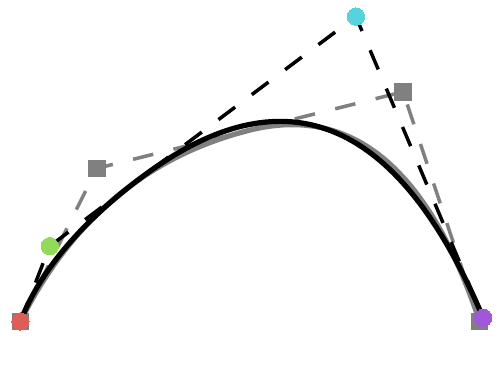}&\includegraphics[width=2in]{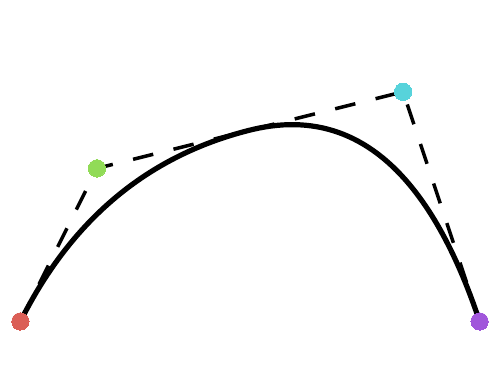}&\includegraphics[width=2in]{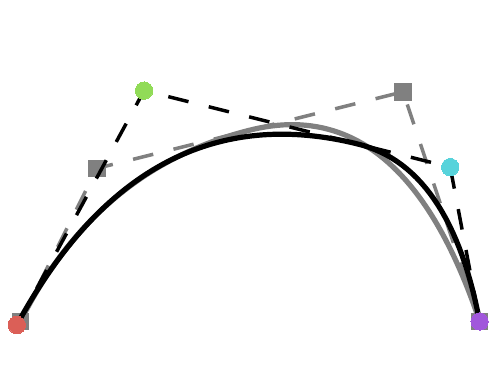}
  \end{tabular}
  \caption{\label{fig:reparam-example}B-spline reparameterization by \Bezier projection.}
\end{figure}

We begin with the projection of the control points associated with
$e_1$ and $e_2$ on the source mesh onto the element $e^{\prime}_1$ on
the target mesh. 
\cref{alg:multi-to-one-proj} must be used because two elements are
being projected onto one element. 
The transformation matrix $\mat{A}_1^{-1}$ that connects the % \Bezier control points for the
Bernstein basis on $e_1$ % into control points for 
to the basis on $e_1^{\prime}$ is found by converting the lower and
upper bounds of $e_1$ to the local coordinates of $e_1^{\prime}$ to
obtain $a = -1$ and $b = \sfrac{3}{7}$. 
Using these values with $\tilde{a}=-1$ and $\tilde{b}=1$ in
\cref{eq:inv-interval-op}, the transformation matrix is 
\begin{equation}
  \mat{A}_1^{-1} = 
  \frac{1}{49}
  \begin{bmatrix}
    1 & 0 & 0 \\
14 & 35 & 0 \\
4 & 20 & 5
  \end{bmatrix}.
\end{equation}
Because the spline segments over $e_1$ and $e_2$ must be projected
onto $e_1^{\prime}$ it is necessary to find an intermediate
representation of the segment from $e_2$ so that
\cref{alg:multi-to-one-proj} can be applied. 
The representation of the segment over the element defined by the
intersection of $e_2$ and $e_1^{\prime}$ is found by multiplying the
\Bezier coefficients on $e_2$ by 
\begin{equation}
\mat{A}^{\prime}=
\begin{bmatrix}
  1 & 0 & 0 \\
  \sfrac{4}{5}  & \sfrac{1}{5} & 0 \\
  \sfrac{16}{25} & \sfrac{8}{25} & \sfrac{1}{25}
\end{bmatrix}.
\end{equation}
This matrix is given by \cref{eq:interval-op} with $a=-1$, $b=1$, $\tilde{a}=-1$, $\tilde{b}=-\sfrac{3}{5}$.
The transformation matrix to convert from the representation over
$e_2\cap e_1^{\prime}$ to $e_1^{\prime}$ is given by
\cref{eq:inv-interval-op} with $a = \sfrac{3}{7}$ and $b = 1$  
\begin{equation}
  \mat{A}_2^{-1} = 
  \frac{1}{4}
  \begin{bmatrix}
 49 & -70 & 25 \\
 0 & 14 & -10 \\
 0 & 0 & 1
  \end{bmatrix}.
\end{equation}
The Bernstein basis is again of degree 2 and so the Gramian is given
by \cref{eq:gramian-bern-quad}. 
The element fractions are $\phi_1 = \sfrac{5}{7}$ and $\phi_2 =
\sfrac{2}{7}$. 
The set of control points associated with element $e_1$ is
$\vec{P}^{e_1} = \{\vec{P}_1, \vec{P}_2, \vec{P}_3\}^{\trans}$ and the
set of control points associated with element $e_2$ is $\vec{P}^{e_2}
= \{\vec{P}_2, \vec{P}_3, \vec{P}_4\}^{\trans}$. 
With this information, the control points on the target element $e_1^{\prime}$ are given by
\begin{equation}
\vec{P}^{e_1^{\prime}} = (\mat{R}^{e_1^{\prime}})^{\trans} \mat{G}^{-1}\left[ \frac{5}{7}\mat{A}_1^{-\trans}\mat{G}(\mat{C}^{e_1})^{\trans}\vec{P}^{e_1} + \frac{2}{7} \mat{A}_2^{-\trans}\mat{G}\mat{A}^{\prime}(\mat{C}^{e_2})^{\trans}\vec{P}^{e_2} \right].
\end{equation}

The control points on element $e^{\prime}_2$ are found by projecting the control points from $e_2$ and so \cref{alg:large-to-small-proj} is appropriate.
The transformation matrix $\mat{A}$ is given by \cref{eq:interval-op} with $a = -1$, $b = 1$, and the element boundaries of $e_2^{\prime}$ are converted to the local coordinates of $e_2$ to obtain $\tilde{a} = -\sfrac{2}{3}$ and $\tilde{b} = 1$:
\begin{equation}
\label{eq:7}
\mat{A} =  
\frac{1}{36}
\begin{bmatrix}
 25 & 10 & 1\\
 0 & 30 & 6\\
 0 & 0 & 36
\end{bmatrix}.
\end{equation}
The control points on $e_2^{\prime}$ are given by
\begin{equation}
  \vec{P}^{e_2^{\prime}} = (\mat{R}^{e_2^{\prime}})^{\trans}\mat{A}(\mat{C}^{e_2})^{\trans}\vec{P}^{e_2}.
\end{equation}
The weighted average of the control points on each element to obtain the global control points is now computed.
The weights for the basis functions shown in the upper right of \cref{fig:reparam-example} for each element are
\begin{align}
\omega_1^{e_1^{\prime}} &= 1, \quad
\omega_2^{e_1^{\prime}} = \frac{13}{16},\quad
\omega_3^{e_1^{\prime}} = \frac{7}{24} \\
\omega_2^{e_2^{\prime}} &= \frac{3}{16},\quad
\omega_3^{e_2^{\prime}} = \frac{17}{24},\quad
\omega_4^{e_2^{\prime}} = 1.
\end{align}
The global control points are thus given by
\begin{align*}
\vec{P}^{\prime}_1 = \vec{P}_1^{e_1^{\prime}},\quad \vec{P}_2^{\prime} = \frac{13}{16}\vec{P}_2^{e_1^{\prime}} + \frac{3}{16}\vec{P}_1^{e_2^{\prime}},\quad
\vec{P}^{\prime}_3 = \frac{7}{24}\vec{P}_3^{e^{\prime}_1} + \frac{17}{24}\vec{P}_2^{e_2^{\prime}},\quad \vec{P}_4^{\prime} = \vec{P}_3^{e_2^{\prime}}.
\end{align*}

The result of this process is illustrated for a simple curve in \cref{fig:reparam-example}.
The original curve and control points are shown in the lower center
and the new curve, produced by the projection of the original control
points onto the reparameterized basis, and control points
are shown in the lower right. Since the spaces are not nested the two curves are
not equal (the source curve is shown in gray on the lower right for
comparison). The reparameterization produced by shifting the center knot to the
left instead of to the right is shown on the left of
\cref{fig:reparam-example} for comparison. 
Because the slope of the curve at the new knot location is better
represented in the left-shifted basis than in the right-shifted basis,
the new curve produced by projection onto the left-shifted basis more
closely approximates the original curve than the right-shifted basis. 

Two simple examples of reparameterization applied to a T-spline are
shown in \cref{fig:ts-project-r-refine}.
\begin{figure}[htb]
  \centering
  \begin{tabular}{ *{3}{c} }
   left shifted T-mesh&original T-mesh&right shifted T-mesh\\
    \includegraphics[width=2in]{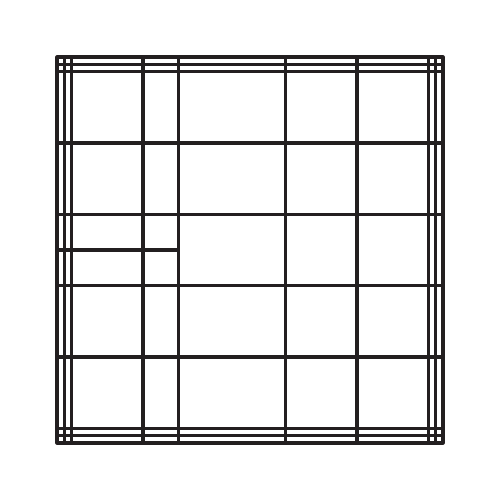}&\includegraphics[width=2in]{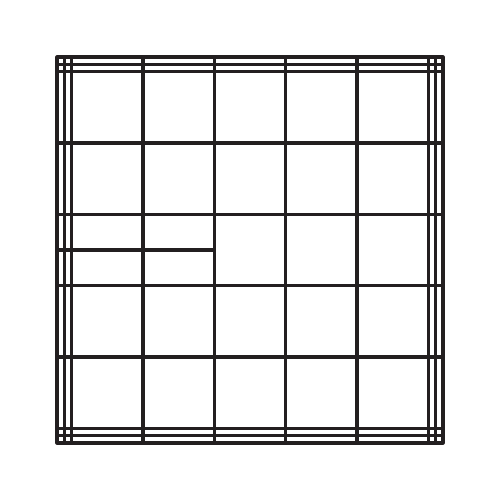}&\includegraphics[width=2in]{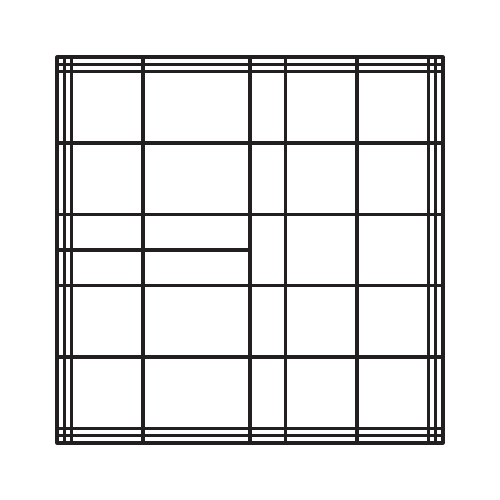}\\
    \includegraphics[width=2in]{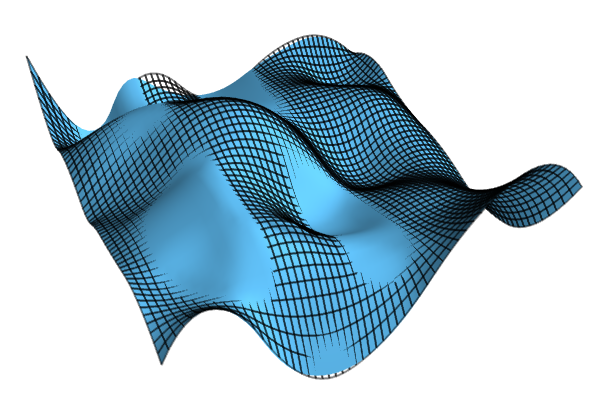}&\includegraphics[width=2in]{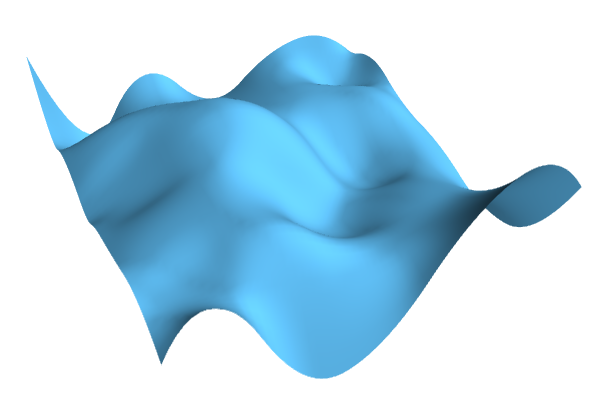}&\includegraphics[width=2in]{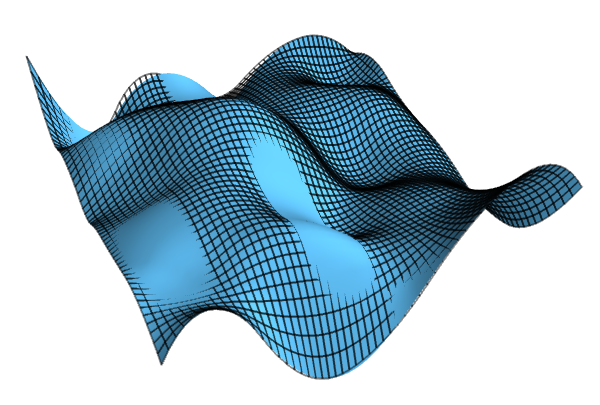}\\
   left shifted surface&original surface&right shifted surface
  \end{tabular}
  \caption{\label{fig:ts-project-r-refine} T-spline reparameterization
    by \Bezier projection.}
\end{figure}
\subsection{Combining $h$-, $p$-, $k$-, and $r$-refinement and coarsening}
Refinement and coarsening by \Bezier projection lends itself well to
the successive application of different refinement and coarsening
procedures. Once the control values have been converted to \Bezier form each
operation is accomplished by the application of the appropriate
matrix. When applying multiple operations it is most efficient to convert to
\Bezier form, apply all necessary transformation matrices to the
Bernstein coefficients, and then convert to spline form and, if
necessary, apply the weighted averaging algorithm as the final
processing step. If any of the operations are not exact then this approch
avoids the accumulation of error associated with each weighted
averaging step. The result will be identical if all of the operations are exact.
\section{Summary and conclusions}
We have presented Bezier projection as a unified approach to local projection, refinement, and coarsening of NURBS and T-splines.
The approach employs a simple three-step procedure, namely, projection onto a local basis, conversion to a global basis, and smoothing using a weighting scheme.
Moreover, the approach relies on the fundamental concept of Bezier extraction and the associated spline reconstruction developed here, resulting in an element-based formulation that may be easily implemented in existing finite element codes.  Optimal convergence rates are proven, and a novel weighting scheme is presented that leads to dramatic improvements in accuracy over previous  approaches.  In fact, by comparing our \cref{fig:govindjee-benchmark} to Fig. 5 in \citet{govindjee2012}, it can be seen that the methodology proposed here produces significantly better results.  Several exemplary applications of Bezier projection are presented to illustrate the accuracy, robustness, and flexibility of the method.

% We have presented \Bezier projection as an accurate, efficient, and
% unified approach to local projection, refinement, and coarsening of
% NURBS and T-splines. Optimal rates of convergence are proven and a novel weighting
% scheme is presented that dramatically improves the accuracy of the
% projection over previous work. Several exemplary applications of the method are presented.
% The fundamental concept of using the \Bezier element extraction operator and
% the spline element reconstruction operator to convert to and from a Bernstein
% representation is fundamental to the method. 
% It can be shown that the local least-squares method of \citet{govindjee2012} is a special case of \Bezier projection.
% By comparing our \cref{fig:convergence} to Fig.\ 5 in
% \citet{govindjee2012}, it can be seen that the weighted averaging proposed here
% provides significantly better results. 
% This difference occurs because simple averaging of the local projection results
% does not respect the relative contribution of a given function over
% different elements and also neglects the potential difference in
% element sizes. 

In the event that data fitting or interpolation is desired, \Bezier
projection is easily modified to accomodate these cases. 
There are efficient and accurate methods for determining the Bernstein
representation of a polynomial interpolating function
\cite{marco2007,marco2010}. 
Once the Bernstein representation has been computed, the spline
segment representation over each element can be computed using the
element reconstruction operator and the weighted average is computed
as before. 

This procedure of computing a Bernstein representation, converting to
a local spline representation, and then smoothing the result to obtain
a smooth global spline representation or approximation provides a
general tool for adapting methods developed for Bernstein polynomials
into methods for splines. 
Indeed, the whole host of methods and techniques developed for
Bernstein polynomials can be applied locally, converted to spline form
and then a smoothed global spline approximation can be computed using
the weighted averaging developed here. 
A retrospective overview of work on Bernstein polynomials is given by
\citet{farouki2012} and many interesting possibilities can be found there and in
the references contained therein.

A unified framework was
developed for quadrature-free degree elevation, degree reduction, knot insertion, knot removal, and
reparameterization of B-splines, NURBS, and T-splines that requires
only element-level information. We feel
that \Bezier projection provides the fundamental building blocks required for
$hpkr$-adaptivity in isogeometric analysis.

\begin{appendices}
\numberwithin{equation}{section}
\crefalias{section}{appsec}
\section{Optimal convergence of the \Bezier projection operator}\label{sec:proof-cont-proj}

In this appendix, we prove that the \Bezier projection operator presented in this paper exhibits optimal convergence rates.  For the sake of brevity, we restrict our discussion to the one-dimensional setting.  By using the methods presented in \cite{BaBeCoHuSa06} and \cite{BeBuSaVa12}, one may extend our theory to the multi-dimensional, rational, and analysis-suitable T-spline settings.  In what follows, $\mathcal{T}$ denotes a univariate B-spline space consisting of $C^{\alpha}$-continuous piecewise polynomials of degree $p$.  We assume that the parametric space $\widehat{\Omega}$ and physical space $\Omega$ coincide, simplifying our exposition.  We additionally employ the simplified notation $\Pi_B$ to represent the \Bezier projection operator $\Pi_B[\mathcal{F},\mathcal{T}] : \mathcal{F} \rightarrow \mathcal{T}$ wherein it is assumed that $\mathcal{F} = L^2(\Omega)$.  Throughout this appendix, we exploit the notion of a \Bezier element.  We denote each \Bezier element using index $e$ and the domain of each \Bezier element using $\Omega^e$, and we define $\mathsf{E}_I$ to be the set of all \Bezier elements.  We will also need the notion of a support extension.  For a \Bezier element $e$, the support extension $\widetilde{\Omega}^e$ is the union of the supports of basis functions whose support intersects $\Omega^e$.

The first ingredient we need in proving optimal convergence of the \Bezier projection operator is approximability.  The following lemma states the local approximation properties of the spline space $\mathcal{T}$, and a proof of the lemma may be found in \cite{BaBeCoHuSa06}.
\begin{lemma}
\label{lemma:approximation}
Let $k$ and $l$ be integer indices with $0 \leq k \leq l \leq p + 1$ and $l \leq \alpha + 1$.  For each \Bezier element $e \in \mathsf{E}_I$, there exists an $s \in \mathcal{T}$ such that
\begin{equation}
| v - s |_{H^k(\widetilde{\Omega}^e)} \leq C_{app} h^{l-k}_e | v |_{H^l(\widetilde{\Omega}^e)}
\end{equation}
where $h_e$ is the mesh size of element $e$, $\widetilde{\Omega}^e$ is the support extension of element $e$, and $C_{app}$ is a constant independent of $h$ but possibly dependent on the shape regularity of the mesh, polynomial degree, continuity, and the parameters $k$ and $l$.
\end{lemma}

\noindent The next two ingredients we need in proving optimal convergence are idempotence and local stability.  The following lemma states these properties for the \Bezier projection operator.  The proof of the lemma is rather involved, so we postpone the proof until later in the appendix.

\begin{lemma}
\label{lemma:continuity}
We have:
\begin{alignat}{10}
\Pi_B (s) &= s, & \hspace{5pt} & \forall s \in \mathcal{T} &&& \hspace{5pt} & \textup{ (spline-preserving property)}\\
\| \Pi_B (v) \|_{L^2(\Omega^e)} &\leq C_{stab} \| v \|_{L^2(\widetilde{\Omega}^e)}, & \hspace{5pt} & \forall v \in L^2(\Omega^e), & \hspace{5pt} & \forall e \in \mathsf{E}_I & \hspace{5pt} & \textup{ (local stability property)}
\end{alignat}
where $C_{stab}$ is a constant independent of $h$ but possibly dependent on the shape regularity of the mesh, polynomial degree, and continuity.
\end{lemma}

\noindent Using \cref{lemma:approximation,lemma:continuity} we can state the convergence properties of the \Bezier projection operator.

\begin{theorem}
\label{theorem:optimal}
Let $k$ and $l$ be integer indices with $0 \leq k \leq l \leq p + 1$ and $l \leq \alpha + 1$.  For each \Bezier element $e \in \mathsf{E}_I$, the following inequality holds:
\begin{equation}
\| f - \Pi_B (f) \|_{H^k(\Omega^e)} \leq C_{int} h_e^{l-k} \| f \|_{H^l(\widetilde{\Omega}^e)}, \hspace{15pt} \forall f \in H^l(\widetilde{\Omega}^e)
\end{equation}
where $h_e$ is the mesh size of element $e$, $\widetilde{\Omega}^e$ is the support extension of element $e$, and $C_{int}$ is a constant independent of $h_e$ but possibly dependent on the shape regularity of the mesh, polynomial degree, continuity, and the parameters $k$ and $l$.
\end{theorem}

\begin{proof}
Let $s \in \mathcal{T}$ be as in Lemma \ref{lemma:approximation}.  Then, by Lemma \ref{lemma:continuity},
\begin{align}
| v - \Pi_B (v) |_{H^k(\Omega^e)} & = | v - s - \Pi_B \left( v - s \right) |_{H^k(\Omega^e)} \nonumber \\
& \leq \left| v - s \right|_{H^k(\Omega^e)} + \left| \Pi_B \left( v - s \right) \right|_{H^k(\Omega^e)} \nonumber \\
& = I + II \nonumber
\end{align}
By Lemma \ref{lemma:approximation}, we immediately have
\begin{equation}
I \leq C_{app} h_e^{l - k}  \| f \|_{H^l(\widetilde{\Omega}^e)}. \nonumber
\end{equation}
The standard inverse inequality for polynomials yields
\begin{equation}
II \leq C_{inv} h_e^{-k} \left\| \Pi_B \left( v - s \right) \right\|_{L^2(\Omega^e)} \nonumber
\end{equation}
where $C_{inv}$ is a constant which only depends on polynomial degree.  By Lemma \ref{lemma:continuity}, we then have
\begin{equation}
II \leq C_{inv} C_{stab} h_e^{-k} \left\| v - s \right\|_{L^2(\widetilde{\Omega}^e)} \leq C_{inv} C_{stab} C_{app} h_e^{l - k} \| f \|_{H^l(\widetilde{\Omega}^e)}. \nonumber
\end{equation}
Thus the theorem holds with $C_{int} = C_{app} \left(1  + C_{inv} C_{stab} \right)$.
\end{proof}

We now return to the proof of Lemma \ref{lemma:continuity}.  While the
spline-preserving property holds trivially, the local stability
property is more technical to establish.  To proceed, we will need the
results of the following two lemmata.  The first lemma concerns the
stability of the local \Bezier element extraction operator and its
inverse while the second concerns the stability of local
$L^2$-projection onto the Bernstein basis. 

\begin{lemma}
For each \Bezier element $e \in \mathsf{E}_I$, the norm of the element extraction operator $\textup{\textbf{C}}^e$ (and its inverse) is independent of the mesh size $h$ but possibly dependent on the shape regularity of the mesh, polynomial degree, and continuity.
\end{lemma}

\begin{proof}
The lemma is a consequence of the fact that the element extraction operators for a given B-spline space are invariant under constant scalings of the domain $\Omega$.
\end{proof}

\begin{lemma}
For each \Bezier element $e \in \mathsf{E}_I$, the local Bernstein coefficient vector $\boldsymbol{\beta}^e$ associated with the local $L^2$-projection of a function $f \in L^2(\Omega^e)$ onto the space of polynomials of degree $p$ satisfies the inequality
\begin{equation}
\left\| \boldsymbol{\beta}^e \right\|_{\infty} \leq \frac{C_p}{h_e^{1/2}} \| f \|_{L^2(\Omega^e)}
\end{equation}
where $h_e$ is the mesh size of element $e$ and $C_p$ is a constant only dependent upon the polynomial degree $p$.
\end{lemma}

\begin{proof}
The vector of coefficients $\vec{\beta}^e = [\beta^{p,e}_i]$ are defined through the equation:
\begin{equation}
\vec{\beta}^e = \textbf{G}^{-1}\textbf{b} \nonumber
\end{equation}
where $\textbf{G} = [G_{ij}]$ with
\begin{equation}
G_{ij} = \int_{-1}^1 B^p_i(\xi) B^p_j(\xi) d\xi \nonumber
\end{equation}
 and $\textbf{b} = [b_i]$ with 
\begin{equation}
b_i = \int_{-1}^1 B^p_i(\xi) \left( f \circ \phi_e^{-1} \right) (\xi) d\xi. \nonumber
\end{equation}
In the above equation, $\phi_e : [-1,1] \rightarrow \Omega^e$ is the standard affine map from the biunit interval onto element $e$.  Since $B^p_i \leq 1$, we have that
\begin{equation}
\| \textbf{b} \|_{\infty} \leq \int_{-1}^1 \left|\left( f \circ \phi_e^{-1} \right) (\xi) \right| d\xi = \| f \circ \phi_e^{-1} \|_{L^1([-1,1])} = 2h_e^{-1} \| f \|_{L^1(\Omega^e)} \leq 2h_e^{-1/2} \| f \|_{L^2(\Omega^e)}. \nonumber
\end{equation}
The last inequality above follows from H\"{o}lder's inequality.  Moreover, by scaling, we have that  $\| \textbf{G}^{-1} \|_{\infty} \leq C_G$ where $C_{G}$ is a constant only dependent upon the polynomial degree $p$.  The desired result immediately follows with $C_p = 2C_{G}$.
\end{proof}

\noindent With the preceding two lemmata established, we are now in a position to complete the proof of Lemma \ref{lemma:continuity}.

\begin{proof}[Proof of Lemma \ref{lemma:continuity}]
The spline-preserving property holds trivially.  Hence, it remains to prove the local stability property.  Let $e \in \mathsf{E}_I$ and $f \in L^2(\Omega^e)$.  We have that:
\begin{equation}
\Pi_B (f)|_{\Omega^e} = \sum_{A \in I_{e}} \left[ \sum_{e' \in E_A} \omega_A^{e'} \lambda_A^{e'}(f) \right] N_A \nonumber
\end{equation}
where $I_{e}$ is the index set of all B-spline basis functions whose support overlaps with element $e$.  For each $e' \in E_A$, the weights $\omega_A^{e'}$ satisfy $|\omega_A^{e'}| \leq 1$ and the vector of coefficients $\boldsymbol{\lambda}^{e'} = [\lambda_A^{e'}(f)]$ is defined by
\begin{equation}
\boldsymbol{\lambda}^{e'} = (\mat{C}^{e'})^{-\trans} \boldsymbol{\beta^{e'}}(f). \nonumber
\end{equation}
where $\boldsymbol{\beta}^{e'}$ is the local Bernstein coefficient vector associated with the local $L^2$-projection of $f$ onto element $e'$.  The results of the previous two lemmas dictate that
\begin{equation}
\left\| \boldsymbol{\lambda}^{e'} \right\|_{\infty} \leq \frac{C_{\lambda}}{h_{e'}^{1/2}} \| f \|_{L^2(\Omega^{e'})} \nonumber
\end{equation}
where $h_{e'}$ is the mesh size of element $e'$ and $C_{\lambda}$ is a constant which only depends on the polynomial degree $p$, the continuity $\alpha$, and the shape regularity of the parametric mesh.  Since the B-spline basis functions are positive and form a partition of unity, we consequently have that
\begin{align}
|\Pi_B (f)| & \leq \left|\sum_{A \in I_{e}} \left[ \sum_{e' \in E_A} w_A^{e'} \lambda_A^{e'}(f) \right] N_A\right| \nonumber \\
& \leq \left|\sum_{A \in I_e} \left[ \sum_{e' \in E_A} \frac{C_{\lambda}}{h_{e'}^{1/2}} \| f \|_{L^2(\Omega^{e'})} \right] N_A\right| \nonumber \\
& \leq \max_{A \in I_e} \left[ \sum_{e' \in E_A} \frac{C_{\lambda}}{h_{e'}^{1/2}} \| f \|_{L^2(\Omega^{e'})} \right] \left|\sum_{A \in I_e} N_A\right| \nonumber \\
& \leq \max_{A \in I_e} \left[ \sum_{e' \in E_A} \frac{C_{\lambda}}{h_{e'}^{1/2}} \| f \|_{L^2(\Omega^{e'})} \right] \nonumber \\
& \leq C_{reg} \frac{C_{\lambda}}{h_e^{1/2}} \| f \|_{L^2(\widetilde{\Omega}^e)} \nonumber
\end{align}
over element $e$ where $C_{reg}$ is a constant only dependent on the shape regularity of the mesh and $\widetilde{\Omega}^e$ is the support extension of element $e$.
Therefore:
\begin{align}
\| \Pi_B (f) \|_{L^2(\Omega^e)} & = \left( \int_{\Omega_{e}} |\Pi_B (f)|^2 d\xi \right)^{1/2} \nonumber \\
& \leq C_{reg} \frac{C_{\lambda}}{h_e^{1/2}} \| f \|_{L^2(\widetilde{\Omega}^e)} h_e^{1/2} \nonumber \\
& = C_{reg} C_{\lambda} \| f \|_{L^2(\widetilde{\Omega}^e)}. \nonumber
\end{align}
Thus the local stability result holds with $C_{stab} = C_{reg} C_{\lambda}$.
\end{proof}
\end{appendices}
\bibliographystyle{elsarticle-harv}
%\bibliography{../isopapers/bibliography/bibliography}
%\bibliography{../bibliography/bibliography}
\bibliography{./refs}

\begin{thebibliography}{95}
\expandafter\ifx\csname natexlab\endcsname\relax\def\natexlab#1{#1}\fi
\expandafter\ifx\csname url\endcsname\relax
  \def\url#1{\texttt{#1}}\fi
\expandafter\ifx\csname urlprefix\endcsname\relax\def\urlprefix{URL }\fi

\bibitem[{Autodesk(2012)}]{TSManual12}
Autodesk, 2012. {Autodesk T-Splines Plug-in for Rhino user manual}. Autodesk.

\bibitem[{{Autodesk, {I}nc.}(2014)}]{Autodesk360}
{Autodesk, {I}nc.}, 2014. {Autodesk Fusion 360}. {Autodesk, {I}nc.}

\bibitem[{Barrera et~al.(2008)Barrera, Ib\'{a}\~{n}ez, Sablonni\`{e}re, and
  Sbibih}]{barrera2008}
Barrera, D., Ib\'{a}\~{n}ez, M., Sablonni\`{e}re, P., Sbibih, D., 2008.
  Near-best univariate spline discrete quasi-interpolants on nonuniform
  partitions. Constructive Approximation 28~(3), 237--251.

\bibitem[{Bazilevs et~al.(2006)Bazilevs, {Beirao de Veiga}, Cottrell, Hughes,
  and Sangalli}]{BaBeCoHuSa06}
Bazilevs, Y., {Beirao de Veiga}, L., Cottrell, J., Hughes, T. J.~R., Sangalli,
  G., 2006. Isogeometric analysis: approximation, stability and error estimates
  for $h$-refined meshes. Mathematical Models and Methods in Applied Sciences
  16, 1031--1090.

\bibitem[{Bazilevs et~al.(2010)Bazilevs, Calo, Cottrell, Evans, Hughes, Lipton,
  Scott, and Sederberg}]{Bazilevs2009}
Bazilevs, Y., Calo, V.~M., Cottrell, J.~A., Evans, J.~A., Hughes, T. J.~R.,
  Lipton, S., Scott, M.~A., Sederberg, T.~W., 2010. Isogeometric analysis using
  {T}-splines. Computer Methods in Applied Mechanics and Engineering 199~(5-8),
  229--263.

\bibitem[{Bazilevs et~al.(2012)Bazilevs, Hsu, and Scott}]{BaHsSc12}
Bazilevs, Y., Hsu, M.~C., Scott, M.~A., 2012. Isogeometric fluid-structure
  interaction analysis with emphasis on non-matching discretizations, and with
  application to wind turbines. {Computer Methods in Applied Mechanics and
  Engineering} 249 - 252, 28 -- 41.

\bibitem[{Beir{\~a}o~da Veiga et~al.(2012)Beir{\~a}o~da Veiga, Buffa, Cho, and
  Sangalli}]{BeBuChSa12}
Beir{\~a}o~da Veiga, L., Buffa, A., Cho, D., Sangalli, G., 2012.
  Analysis-suitable {T}-splines are dual-compatible. Computer Methods in
  Applied Mechanics and Engineering 249 -- 252, 42 -- 51.

\bibitem[{Benson et~al.(2010)Benson, Bazilevs, De~Luycker, Hsu, Scott, Hughes,
  and Belytschko}]{BeBaDeHsScHuBe09}
Benson, D.~J., Bazilevs, Y., De~Luycker, E., Hsu, M.~C., Scott, M.~A., Hughes,
  T. J.~R., Belytschko, T., 2010. {A generalized finite element formulation for
  arbitrary basis functions: {F}rom isogeometric analysis to XFEM}.
  International Journal for Numerical Methods in Engineering 83, 765--785.

\bibitem[{Berdinsky et~al.(2014)Berdinsky, wan Kim, Bracco, Cho, Mourrain, Oh,
  and Kiatpanichgij}]{BeKiBrChMoOhKi14}
Berdinsky, D., wan Kim, T., Bracco, C., Cho, D., Mourrain, B., Oh, M.,
  Kiatpanichgij, S., 2014. Dimensions and bases of hierarchical tensor-product
  splines. Journal of Computational and Applied Mathematics 257, 86 -- 104.

\bibitem[{B{\'e}zier(1966)}]{bez66}
B{\'e}zier, P., 1966. D{\'e}finition num{\'e}rique des courbes et surfaces {I}.
  Automatisme XI, 625--632.

\bibitem[{B{\'e}zier(1967)}]{bez67}
B{\'e}zier, P., 1967. D{\'e}finition num{\'e}rique des courbes et surfaces
  {II}. Automatisme XII, 17--21.

\bibitem[{Borden et~al.(2011)Borden, Scott, Evans, and Hughes}]{Borden:2010nx}
Borden, M.~J., Scott, M.~A., Evans, J.~A., Hughes, T. J.~R., 2011. Isogeometric
  finite element data structures based on {B}\'ezier extraction of {NURBS}.
  International Journal for Numerical Methods in Engineering, 87, 15 -- 47.

\bibitem[{Borden et~al.(2012)Borden, Scott, Verhoosel, Landis, and
  Hughes}]{BoScLaHuVe11}
Borden, M.~J., Scott, M.~A., Verhoosel, C.~V., Landis, C.~M., Hughes, T. J.~R.,
  2012. A phase-field description of dynamic brittle fracture. {Computer
  Methods in Applied Mechanics and Engineering} 217, 77 -- 95.

\bibitem[{Bracco et~al.(2014)Bracco, Berdinsky, Cho, Oh, and wan
  Kim}]{BrBeChOhKi14}
Bracco, C., Berdinsky, D., Cho, D., Oh, M., wan Kim, T., 2014. Trigonometric
  generalized {T}-splines. Computer Methods in Applied Mechanics and
  Engineering 268, 540 -- 556.

\bibitem[{Bressan(2013)}]{Br13}
Bressan, A., 2013. Some properties of {LR}-splines. Computer Aided Geometric
  Design 30~(8), 778 -- 794.

\bibitem[{Buffa et~al.(2014)Buffa, Sangalli, and V{\'a}zquez}]{BuSaVa12}
Buffa, A., Sangalli, G., V{\'a}zquez, R., 2014. Isogeometric methods for
  computational electromagnetics: {B}-spline and {T}-spline discretizations.
  Journal of Computational Physics 257, Part B, 1291 -- 1320.

\bibitem[{Burkhart et~al.(2010)Burkhart, Hamann, and Umlauf}]{BuHaUm10}
Burkhart, D., Hamann, B., Umlauf, G., 2010. Isogeometric finite element
  analysis based on {C}atmull-{C}lark subdivision solids. In: Computer Graphics
  Forum. Vol.~29. Wiley Online Library, pp. 1575--1584.

\bibitem[{Catmull and Clark(1978)}]{CaCl78}
Catmull, E., Clark, J., 1978. Recursively generated {B}-spline surfaces on
  arbitrary topological meshes. Computer Aided Design 10, 350--355.

\bibitem[{Cirak et~al.(2000)Cirak, Ortiz, and
  Schr$\ddot{\text{o}}$der}]{CiOrShr00}
Cirak, F., Ortiz, M., Schr$\ddot{\text{o}}$der, P., 2000. Subdivision surfaces:
  {A} new paradigm for thin shell analysis. International Journal for Numerical
  Methods in Engineering 47, 2039--2072.

\bibitem[{Cohen et~al.(1980)Cohen, Lyche, and Riesenfeld}]{CoLyRi80}
Cohen, E., Lyche, T., Riesenfeld, R., 1980. Discrete {B}-splines and
  subdivision techniques in computer-aided geometric design and computer
  graphics. Computer Graphics and Image Processing 14~(2), 87 -- 111.

\bibitem[{Cohen et~al.({2010})Cohen, Martin, Kirby, Lyche, and
  Riesenfeld}]{CoMaKiLyRi10}
Cohen, E., Martin, T., Kirby, R.~M., Lyche, T., Riesenfeld, R.~F., {2010}.
  {Analysis-aware modeling: Understanding quality considerations in modeling
  for isogeometric analysis}. Computer Methods in Applied Mechanics and
  Engineering {199}~({5-8}), {334--356}.

\bibitem[{Constantini et~al.({2010})Constantini, Manni, Pelosi, and
  Sampoli}]{CoMaPeSa10}
Constantini, P., Manni, C., Pelosi, F., Sampoli, M.~L., {2010}.
  {Quasi-interpolation in isogeometric analysis based on generalized
  B-splines}. Computer Aided Geometric Design {27}~({8}), {656--668}.

\bibitem[{Cottrell et~al.(2009)Cottrell, Hughes, and
  Bazilevs}]{Cottrell:2009rp}
Cottrell, J.~A., Hughes, T. J.~R., Bazilevs, Y., 2009. Isogeometric analysis:
  Toward {I}ntegration of {CAD} and {FEA}. Wiley, Chichester.

\bibitem[{Cottrell et~al.(2007)Cottrell, Hughes, and Reali}]{CoHuRe07}
Cottrell, J.~A., Hughes, T. J.~R., Reali, A., 2007. {S}tudies of refinement and
  continuity in isogeometric analysis. Computer Methods in Applied Mechanics
  and Engineering 196, 4160--4183.

\bibitem[{Cottrell et~al.(2006)Cottrell, Reali, Bazilevs, and
  Hughes}]{CoReBaHu05}
Cottrell, J.~A., Reali, A., Bazilevs, Y., Hughes, T. J.~R., 2006. Isogeometric
  analysis of structural vibrations. Computer Methods in Applied Mechanics and
  Engineering 195, 5257--5296.

\bibitem[{da~Veiga et~al.(2013)da~Veiga, Buffa, Sangalli, and
  V\'azquez}]{BeBuSaVa12}
da~Veiga, L.~B., Buffa, A., Sangalli, G., V\'azquez, R., 2013.
  Analysis-suitable {T}-splines of arbitrary degree: definition, linear
  independence, and approximation properties. Mathematical Models and Methods
  in Applied Sciences 23~(11), 1979 -- 2003.

\bibitem[{de~Boor(1972)}]{Bo72}
de~Boor, C., 1972. On calculating with {B}-splines. Journal of Approximation
  Theory 6~(1), 50 -- 62.

\bibitem[{de~Boor(1990)}]{Bo90}
de~Boor, C., 1990. Quasiinterpolants and approximation power of multivariate
  splines. In: Dahmen, W., Gasca, M., Micchelli, C. (Eds.), Computation of
  Curves and Surfaces. Vol. 307 of NATO ASI Series. Springer Netherlands, pp.
  313--345.

\bibitem[{{de Boor} and Fix(1973)}]{deboor1973b}
{de Boor}, C., Fix, G., 1973. Spline approximation by quasiinterpolants.
  Journal of Approximation Theory 8~(1), 19 -- 45.

\bibitem[{de~Casteljau(1963)}]{dec63}
de~Casteljau, P., 1963. Courbes et surfaces a poles. Tech. rep., A. Citroen.

\bibitem[{Deng et~al.(2008)Deng, Chen, Li, Hu, Tong, Yang, and
  Feng}]{htspline2}
Deng, J., Chen, F., Li, X., Hu, C., Tong, W., Yang, Z., Feng, Y., 2008.
  Polynomial splines over hierarchical {T}-meshes. Graphical Models 74, 76--86.

\bibitem[{Dimitri et~al.(2014)Dimitri, Lorenzis, Scott, Wriggers, Taylor, and
  Zavarise}]{DiLoScWrTaZa13}
Dimitri, R., Lorenzis, L.~D., Scott, M.~A., Wriggers, P., Taylor, R., Zavarise,
  G., 2014. Isogeometric large deformation frictionless contact using
  {T}-splines. Computer methods in applied mechanics and engineering 269, 394
  -- 414.

\bibitem[{Doha et~al.(2011)Doha, Bhrawy, and Saker}]{doha2011}
Doha, E., Bhrawy, A., Saker, M., 2011. Integrals of bernstein polynomials: An
  application for the solution of high even-order differential equations.
  Applied Mathematics Letters 24~(4), 559 -- 565.

\bibitem[{Dokken et~al.(2013)Dokken, Lyche, and Pettersen}]{DoLyPe13}
Dokken, T., Lyche, T., Pettersen, K.~F., 2013. Polynomial splines over locally
  refined box-partitions. Computer Aided Geometric Design 30~(3), 331--356.

\bibitem[{Donea et~al.(1982)Donea, Giuliani, and Halleux}]{DGH82}
Donea, J., Giuliani, S., Halleux, J.~P., 1982. An arbitrary
  {L}agrangian-{E}ulerian finite element method for transient dynamics
  fluid-structure interactions. Computer Methods in Applied Mechanics and
  Engineering 33, 689--723.

\bibitem[{D\"orfel et~al.(2009)D\"orfel, J\"uttler, and Simeon}]{DoJuSi09}
D\"orfel, M., J\"uttler, B., Simeon, B., 2009. Adaptive isogeometric analysis
  by local {\it h}-refinement with {T}-splines. Computer Methods in Applied
  Mechanics and Engineering 199~(5--8), 264--275.

\bibitem[{Eck and Hadenfeld(1995)}]{EcHa95}
Eck, M., Hadenfeld, J., 1995. Knot removal for {B}-spline curves. Computer
  Aided Geometric Design 12~(3), 259 -- 282.

\bibitem[{Evans et~al.(2014)Evans, Scott, Li, and Thomas}]{EvScLiTh13}
Evans, E.~J., Scott, M.~A., Li, X., Thomas, D.~C., 2014. Hierarchical
  analysis-suitable {T}-splines: Formulation, {B}\'{e}zier extraction, and
  application as an adaptive basis for isogeometric analysis.
  arXiv:math.NA/1404.4346, submitted.

\bibitem[{Evans et~al.(2009)Evans, Bazilevs, Babu{\v s}ka, and
  Hughes}]{EvBaBaHu09}
Evans, J.~A., Bazilevs, Y., Babu{\v s}ka, I., Hughes, T. J.~R., 2009. {\it
  n}-widths, sup-infs, and optimality ratios for the {\it k}-version of the
  isogeometric finite element method. Computer Methods in Applied Mechanics and
  Engineering 198~(21-26), 1726--1741.

\bibitem[{Farouki(2012)}]{farouki2012}
Farouki, R.~T., 2012. The {Bernstein} polynomial basis: A centennial
  retrospective. Computer Aided Geometric Design 29~(6), 379 -- 419.

\bibitem[{Farouki and Neff(1990)}]{farouki1990}
Farouki, R.~T., Neff, C.~A., 1990. On the numerical condition of
  {Bernstein-Bezier} subdivision processes. Mathematics of Computation
  55~(192), 637--647.

\bibitem[{Forsey and Bartels(1988)}]{FoBa88}
Forsey, D.~R., Bartels, R.~H., 1988. {Hierarchical B-spline refinement}. ACM
  SIGGRAPH Computer Graphics 22~(4), 205--212.

\bibitem[{Giannelli et~al.(2012)Giannelli, J{\"u}ttler, and
  Speleers}]{GiJuSp12}
Giannelli, C., J{\"u}ttler, B., Speleers, H., 2012. {THB}--splines: The
  truncated basis for hierarchical splines. Computer Aided Geometric Design
  29~(7), 485 -- 498.

\bibitem[{Giannelli et~al.(2013)Giannelli, J{\"u}ttler, and
  Speleers}]{GiJuSp13}
Giannelli, C., J{\"u}ttler, B., Speleers, H., 2013. Strongly stable bases for
  adaptively refined multilevel spline spaces. Advances in Computational
  Mathematics, 1--32.

\bibitem[{Ginnis et~al.(2014)Ginnis, Kostas, Politis, Kaklis, Belibassakis,
  Gerostathis, Scott, and Hughes}]{GiKoPoKaBeGeScHu14}
Ginnis, A.~I., Kostas, K.~V., Politis, C.~G., Kaklis, P.~D., Belibassakis,
  K.~A., Gerostathis, T.~P., Scott, M.~A., Hughes, T. J.~R., 2014. Isogeometric
  boundary-element analysis for the wave-resistance problem using {T}-splines.
  {Computer Methods in Applied Mechanics and Engineering} submitted.

\bibitem[{Goldman and Lyche(1993)}]{GoLy93}
Goldman, R., Lyche, T., 1993. Knot Insertion and Deletion Algorithms for
  {B}-Spline Curves and Surfaces. Society for Industrial and Applied
  Mathematics.

\bibitem[{Govindjee et~al.(2012)Govindjee, Strain, Mitchell, and
  Taylor}]{govindjee2012}
Govindjee, S., Strain, J., Mitchell, T.~J., Taylor, R.~L., 2012. Convergence of
  an efficient local least-squares fitting method for bases with compact
  support. Computer Methods in Applied Mechanics and Engineering 213-216,
  84--92.

\bibitem[{Grinspun et~al.(2002)Grinspun, Krysl, and Schr{\"o}der}]{GrKrSch02}
Grinspun, E., Krysl, P., Schr{\"o}der, P., 2002. {CHARMS}: a simple framework
  for adaptive simulation. {ACM} Transactions on Graphics 21~(3), 281--290.

\bibitem[{Hosseini et~al.(2014)Hosseini, Remmers, Verhoosel, and
  de~Borst}]{HoReVeBo14}
Hosseini, S., Remmers, J.~J., Verhoosel, C.~V., de~Borst, R., 2014. An
  isogeometric continuum shell element for non-linear analysis. Computer
  Methods in Applied Mechanics and Engineering 271, 1 -- 22.

\bibitem[{Huang et~al.(2005)Huang, Hu, and Martin}]{HuHuMa05}
Huang, Q.-X., Hu, S.-M., Martin, R.~R., 2005. Fast degree elevation and knot
  insertion for {B}-spline curves. Computer Aided Geometric Design 22~(2), 183
  -- 197.

\bibitem[{Hughes et~al.(2005)Hughes, Cottrell, and Bazilevs}]{HuCoBa04}
Hughes, T. J.~R., Cottrell, J.~A., Bazilevs, Y., 2005. Isogeometric analysis:
  {CAD}, finite elements, {NURBS}, exact geometry, and mesh refinement.
  Computer Methods in Applied Mechanics and Engineering 194, 4135--4195.

\bibitem[{Hughes et~al.(2014)Hughes, Evans, and Reali}]{HuEvRe13}
Hughes, T. J.~R., Evans, J.~A., Reali, A., 2014. Finite element and {NURBS}
  approximations of eigenvalue, boundary-value, and initial-value problems.
  Computer Methods in Applied Mechanics and Engineering 272, 290 -- 320.

\bibitem[{Jaxon and Qian(2014)}]{JaQi14}
Jaxon, N., Qian, X., 2014. Isogeometric analysis on triangulations.
  Computer-Aided Design 46, 45 -- 57.

\bibitem[{Johnson and Tezduyar(1994)}]{JohTez94}
Johnson, A.~A., Tezduyar, T.~E., 1994. Mesh update strategies in parallel
  finite element computations of flow problems with moving boundaries and
  interfaces. Computer Methods in Applied Mechanics and Engineering 119,
  73--94.

\bibitem[{J{\"u}ttler(1998)}]{juttler1998}
J{\"u}ttler, B., 1998. The dual basis functions for the bernstein polynomials.
  Advances in Computational Mathematics 8~(4), 345--352.

\bibitem[{Kang et~al.(2013)Kang, Chen, and Deng}]{KaChDe13}
Kang, H., Chen, F., Deng, J., 2013. Modified {T}-splines. Computer Aided
  Geometric Design 30~(9), 827 -- 843.

\bibitem[{Kiss et~al.(2014)Kiss, Giannelli, and J{\"u}ttler}]{KiGiJu14}
Kiss, G., Giannelli, C., J{\"u}ttler, B., 2014. Algorithms and data structures
  for truncated hierarchical {B}-splines. In: Floater, M., Lyche, T., Mazure,
  M.-L., M{\o}rken, K., Schumaker, L. (Eds.), Mathematical Methods for Curves
  and Surfaces. Vol. 8177 of Lecture Notes in Computer Science. Springer Berlin
  Heidelberg, pp. 304--323.

\bibitem[{Lee et~al.(2000)Lee, Lyche, and M{\o}rken}]{lee2000}
Lee, B.-G., Lyche, T., M{\o}rken, K., 2000. Some examples of quasi-interpolants
  constructed from local spline projectors. In: In Mathematical Methods in
  CAGD: Oslo 2000, Vanderbilt. University Press, pp. 243--252.

\bibitem[{Li et~al.(2006)Li, Deng, and Chen}]{htspline5}
Li, X., Deng, J., Chen, F., 2006. The dimension of spline spaces over 3d
  hierarchical {T}-meshes. Journal of Information and Computational Science 3,
  487--501.

\bibitem[{Li et~al.(2007)Li, Deng, and Chen}]{htspline3}
Li, X., Deng, J., Chen, F., 2007. Surface modeling with polynomial splines over
  hierarchical {T}-meshes. The Visual Computer 23, 1027--1033.

\bibitem[{Li et~al.(2010)Li, Deng, and Chen}]{htspline4}
Li, X., Deng, J., Chen, F., 2010. Polynomial splines over general {T}-meshes.
  The Visual Computer 26, 277--286.

\bibitem[{Li and Scott(2014)}]{LiScSe12}
Li, X., Scott, M.~A., 2014. Analysis-suitable {T}-splines: characterization,
  refineability, and approximation. Mathematical Models and Methods in Applied
  Science 24~(06), 1141--1164.

\bibitem[{Li et~al.(2012)Li, Zheng, Sederberg, Hughes, and
  Scott}]{LiZhSeHuSc10}
Li, X., Zheng, J., Sederberg, T.~W., Hughes, T. J.~R., Scott, M.~A., 2012. On
  linear independence of {T}-spline blending functions. Computer Aided
  Geometric Design 29, 63 -- 76.

\bibitem[{Liu et~al.(2014)Liu, Zhang, Hughes, Scott, and
  Sederberg}]{LiZhHuScSe14}
Liu, L., Zhang, Y., Hughes, T. J.~R., Scott, M.~A., Sederberg, T.~W., 2014.
  Volumetric {T}-spline construction using boolean operations. In: Sarrate, J.,
  Staten, M. (Eds.), Proceedings of the 22nd International Meshing Roundtable.
  Springer International Publishing, pp. 405--424.

\bibitem[{Loop(1987)}]{Lo87}
Loop, C., 1987. Smooth subdivision surfaces based on triangles. Master's
  thesis, University of Utah.

\bibitem[{Lutterkort et~al.(1999)Lutterkort, Peters, and Reif}]{lutterkort1999}
Lutterkort, D., Peters, J., Reif, U., 1999. Polynomial degree reduction in the
  {L2}-norm equals best {Euclidean} approximation of {Bezier} coefficients.
  Computer Aided Geometric Design 16~(7), 607 -- 612.

\bibitem[{Manni et~al.(2011)Manni, Pelosi, and Sampoli}]{MaPeSa11}
Manni, C., Pelosi, F., Sampoli, M.~L., 2011. Generalized {B}-splines as a tool
  in isogeometric analysis. Computer Methods in Applied Mechanics and
  Engineering 200~(5--8), 867 -- 881.

\bibitem[{Marco and Mart\'{i}nez(2007)}]{marco2007}
Marco, A., Mart\'{i}nez, J.-J., 2007. A fast and accurate algorithm for solving
  {Bernstein-Vandermonde} linear systems. Linear Algebra and its Applications
  422~(23), 616 -- 628.

\bibitem[{Marco and Mart\'{i}nez(2010)}]{marco2010}
Marco, A., Mart\'{i}nez, J.-J., 2010. Polynomial least squares fitting in the
  {Bernstein} basis. Linear Algebra and its Applications 433~(7), 1254 -- 1264.

\bibitem[{Peters and Reif(2000)}]{peters2000}
Peters, J., Reif, U., 2000. Least squares approximation of {Bezier}
  coefficients provides best degree reduction in the {L2}-norm. Journal of
  Approximation Theory 104~(1), 90 -- 97.

\bibitem[{Piegl(1991)}]{Pi91}
Piegl, L., Jan. 1991. On {NURBS}: A survey. IEEE Comput. Graph. Appl. 11~(1),
  55--71.
\newline\urlprefix\url{http://dx.doi.org/10.1109/38.67702}

\bibitem[{Piegl and Tiller(1997)}]{PiegTil97}
Piegl, L., Tiller, W., 1997. The NURBS Book. Springer-Verlag, New York.

\bibitem[{Prautzsch(1984)}]{Pr84}
Prautzsch, H., 1984. Degree elevation of {B}-spline curves. Computer Aided
  Geometric Design 1~(2), 193 -- 198.

\bibitem[{Riesenfeld(1973)}]{Ri73}
Riesenfeld, R.~F., 1973. Applications of {B}-spline approximation to geometric
  problems of computer-aided design. Ph.D. thesis, Syracuse University,
  Syracuse, NY, USA.

\bibitem[{Sablonni\`{e}re(2005)}]{sablonniere2005}
Sablonni\`{e}re, P., 2005. Recent progress on univariate and multivariate
  polynomial and spline quasi-interpolants. In: Mache, D.~H., Szabados, J.,
  Bruin, M.~G. (Eds.), Trends and Applications in Constructive Approximation.
  Vol. 151 of ISNM International Series of Numerical Mathematics. Birkhauser
  Basel, pp. 229--245.

\bibitem[{Schillinger et~al.(2012)Schillinger, Ded\'{e}, Scott, Evans, Borden,
  Rank, and Hughes}]{SchDeScEvBoRaHu12}
Schillinger, D., Ded\'{e}, L., Scott, M.~A., Evans, J.~A., Borden, M.~J., Rank,
  E., Hughes, T. J.~R., 2012. An isogeometric design-through-analysis
  methodology based on adaptive hierarchical refinement of {NURBS}, immersed
  boundary methods, and {T}-spline {CAD} surfaces. {Computer Methods in Applied
  Mechanics and Engineering} 249 -- 252, 116 -- 150.

\bibitem[{Schillinger et~al.(2013)Schillinger, Evans, Reali, Scott, and
  Hughes}]{SchEvReScHu13}
Schillinger, D., Evans, J.~A., Reali, A., Scott, M.~A., Hughes, T. J.~R., 2013.
  Isogeometric collocation: Cost comparison with {G}alerkin methods and
  extension to adaptive hierarchical {NURBS} discretizations. {Computer Methods
  in Applied Mechanics and Engineering} 267, 170 -- 232.

\bibitem[{Schmidt et~al.(2012)Schmidt, W{\"u}chner, and Bletzinger}]{ScWuBl12}
Schmidt, R., W{\"u}chner, R., Bletzinger, K.-U., 2012. Isogeometric analysis of
  trimmed {NURBS} geometries. Computer Methods in Applied Mechanics and
  Engineering 241--244, 93 -- 111.

\bibitem[{Scott(2011)}]{sc11}
Scott, M.~A., 2011. T-splines as a {D}esign-{T}hrough-{A}nalysis technology.
  Ph.D. thesis, The University of Texas at Austin.

\bibitem[{Scott et~al.(2011)Scott, Borden, Verhoosel, Sederberg, and
  Hughes}]{ScBoHu10}
Scott, M.~A., Borden, M.~J., Verhoosel, C.~V., Sederberg, T.~W., Hughes, T.
  J.~R., 2011. {Isogeometric Finite Element Data Structures based on B\'{e}zier
  Extraction of {T}-splines}. International Journal for Numerical Methods in
  Engineering, 88, 126 -- 156.

\bibitem[{Scott et~al.(2012)Scott, Li, Sederberg, and Hughes}]{ScLiSeHu10}
Scott, M.~A., Li, X., Sederberg, T.~W., Hughes, T. J.~R., 2012. Local
  refinement of analysis-suitable {T}-splines. Computer Methods in Applied
  Mechanics and Engineering 213, 206 -- 222.

\bibitem[{Scott et~al.(2013)Scott, Simpson, Evans, Lipton, Bordas, Hughes, and
  Sederberg}]{ScSiEvLiBoHuSe12}
Scott, M.~A., Simpson, R.~N., Evans, J.~A., Lipton, S., Bordas, S. P.~A.,
  Hughes, T. J.~R., Sederberg, T.~W., 2013. Isogeometric boundary element
  analysis using unstructured {T}-splines. {Computer Methods in Applied
  Mechanics and Engineering} 254, 197 -- 221.

\bibitem[{Scott et~al.(2014)Scott, Thomas, and Evans}]{ScThEv13}
Scott, M.~A., Thomas, D.~C., Evans, E.~J., 2014. Isogeometric spline forests.
  {Computer Methods in Applied Mechanics and Engineering} 269, 222 -- 264.

\bibitem[{Sederberg et~al.(2004)Sederberg, Cardon, Finnigan, North, Zheng, and
  Lyche}]{SeCaFiNoZhLy04}
Sederberg, T.~W., Cardon, D.~L., Finnigan, G.~T., North, N.~S., Zheng, J.,
  Lyche, T., August 2004. T-spline simplification and local refinement. ACM
  Trans. Graph. 23, 276--283.

\bibitem[{Sederberg et~al.(2003)Sederberg, Zheng, Bakenov, and
  Nasri}]{SeZhBaNa03}
Sederberg, T.~W., Zheng, J., Bakenov, A., Nasri, A., July 2003. T-splines and
  {T-NURCC}s. ACM Trans. Graph. 22, 477--484.

\bibitem[{Simpson et~al.(2014)Simpson, Scott, Taus, Thomas, and
  Lian}]{SiScTaThLi14}
Simpson, R.~N., Scott, M.~A., Taus, M., Thomas, D.~C., Lian, H., 2014. Acoustic
  isogeometric boundary element analysis. Computer Methods in Applied Mechanics
  and Engineering 269, 265--290.

\bibitem[{Speleers et~al.(2013)Speleers, Manni, and Pelosi}]{SpCaFr13}
Speleers, H., Manni, C., Pelosi, F., 2013. From {NURBS} to {NURPS} geometries.
  Computer Methods in Applied Mechanics and Engineering 255, 238--254.

\bibitem[{Speleers et~al.(2012)Speleers, Manni, Pelosi, and
  Sampoli}]{SpMaPeSa12}
Speleers, H., Manni, C., Pelosi, F., Sampoli, M.~L., 2012. Isogeometric
  analysis with {P}owell--{S}abin splines for advection--diffusion--reaction
  problems. Computer Methods in Applied Mechanics and Engineering 221,
  132--148.

\bibitem[{Szafnicki(2005)}]{szafnicki2005}
Szafnicki, B., 2005. On the degree elevation of {Bernstein} polynomial
  representation. Journal of Computational and Applied Mathematics 180~(2), 443
  -- 459.

\bibitem[{Verhoosel et~al.(2011{\natexlab{a}})Verhoosel, Scott, de~Borst, and
  Hughes}]{Verhoosel:2010ly}
Verhoosel, C.~V., Scott, M.~A., de~Borst, R., Hughes, T. J.~R.,
  2011{\natexlab{a}}. An isogeometric approach to cohesive zone modeling.
  International Journal for Numerical Methods in Engineering, 87, 336 -- 360.

\bibitem[{Verhoosel et~al.(2011{\natexlab{b}})Verhoosel, Scott, Hughes, and {de
  Borst}}]{Verhoosel:2010vn}
Verhoosel, C.~V., Scott, M.~A., Hughes, T. J.~R., {de Borst}, R.,
  2011{\natexlab{b}}. An isogeometric analysis approach to gradient damage
  models. International Journal for Numerical Methods in Engineering, 86,
  115--134.

\bibitem[{Versprille(1975)}]{Ve75}
Versprille, K.~J., 1975. Computer-aided design applications of the rational
  {B}-spline approximation form. Ph.D. thesis, Syracuse University, Syracuse,
  NY, USA, aAI7607690.

\bibitem[{Vuong et~al.(2011)Vuong, Giannelli, J{\"u}ttler, and
  Simeon}]{VuGiJuSi11}
Vuong, A., Giannelli, C., J{\"u}ttler, B., Simeon, B., 2011. A hierarchical
  approach to adaptive local refinement in isogeometric analysis. Computer
  Methods in Applied Mechanics and Engineering 200~(49 -- 52), 3554 -- 3567.

\bibitem[{Wall et~al.(2008)Wall, Frenzel, and Cyron}]{Wall08}
Wall, W.~A., Frenzel, M.~A., Cyron, C., 2008. Isogeometric structural shape
  optimization. Computer Methods in Applied Mechanics and Engineering 197,
  2976--2988.

\bibitem[{Wang et~al.(2011)Wang, Zhang, Scott, and Hughes}]{WaZhScHu11}
Wang, W., Zhang, Y., Scott, M.~A., Hughes, T. J.~R., 2011. Converting an
  unstructured quadrilateral mesh to a standard {T}-spline surface.
  Computational {M}echanics, 48, 477 -- 498.

\end{thebibliography}
% Emacs 24.2.1 (Org mode 8.0.3)
\end{document}